\documentclass[12pt,reqno]{amsart}

\usepackage{amssymb, amsmath, amsthm, amsfonts}
\usepackage{mathabx}
\usepackage{thmtools}
\usepackage{mathtools}
\usepackage[centering, margin=1.1in]{geometry}
\usepackage{enumitem}
\usepackage[space]{cite}
\usepackage{verbatim}

\usepackage[colorlinks, linkcolor=black, citecolor=black, hypertexnames=false]{hyperref}
\usepackage[capitalize, nameinlink]{cleveref}

\usepackage[alphabetic, nobysame]{amsrefs}

\AtBeginDocument{\def\MR#1{}}

\mathtoolsset{showonlyrefs} 

\declaretheorem[name=Theorem, numberwithin=section]{theorem}

\declaretheorem[name=Lemma, sibling=theorem]{lemma}
\declaretheorem[name=Proposition, sibling=theorem]{prop}

\declaretheorem[name=Remark, sibling=theorem]{remark}
\declaretheorem[name=Conjecture, sibling=theorem]{conjecture}

\declaretheorem[name=Claim, numbered=no]{claim*}

\crefname{section}{Section}{Sections}
\Crefname{section}{Section}{Sections}

\crefname{appendix}{Appendix}{Appendices}
\Crefname{appendix}{Appendix}{Appendices}

\numberwithin{equation}{section}

\renewcommand{\Im}{\operatorname{Im}}
\renewcommand{\Re}{\operatorname{Re}}
\newcommand{\ord}{\operatorname{ord}}

\renewcommand{\pmod}[1]{\, (\mathrm{mod}\ #1)}
\newcommand\norm[1]{\left\Vert {#1} \right\Vert}

\newcommand{\Z}{\mathbb{Z}}
\newcommand{\Q}{\mathbb{Q}}
\newcommand{\R}{\mathbb{R}}
\newcommand{\C}{\mathbb{C}}
\newcommand{\N}{\mathbb{N}}

\DeclareMathOperator*{\Res}{Res}
\DeclareMathOperator{\vol}{vol}
\DeclareMathOperator{\SL}{SL}

\newcommand{\sumtwo}{\operatorname*{\sum\sum}}

\newcommand{\pMatrix}[4]{\left(\begin{matrix}#1 & #2 \\ #3 & #4\end{matrix}\right)}

\renewcommand{\pmatrix}[4]{\left(\begin{smallmatrix}#1 & #2 \\ #3 & #4\end{smallmatrix}\right)}

\begin{document}

\author{Alexandre de Faveri}
\address{EPFL SB MATH, Station 10, 1015 Lausanne, Switzerland}
\email{\url{alexandre.defaveri@epfl.ch}}

\author{Alexander Dunn}
\address{School of Mathematics, Georgia Institute of Technology,
    Atlanta, GA, USA}
\email{\url{adunn61@gatech.edu}}

\author{Jeffrey Hoffstein}
\address{Mathematics Department, Brown University,
Providence, RI, USA}
\email{\url{jeffrey\_hoffstein@brown.edu}}

\title[Non-orthogonality of the cubic and quartic large sieves]{Non-orthogonality of the cubic and quartic \\ large sieves via Rankin--Selberg}

\begin{abstract}
    We show unconditionally that the cubic and quartic large sieves are \emph{not} perfectly orthogonal. The main obstruction to perfect orthogonality comes from the bias exhibited by Gauss sums.

    Our proof requires two main inputs: a Lindel\"{o}f-on-average upper bound for the second moment of Kubota's Dirichlet series, and a tight average lower bound for the Fourier coefficients of a certain Rankin--Selberg convolution of metaplectic theta functions. The latter input is particularly important in the quartic case, where much less is known about Fourier coefficients of metaplectic theta functions. To establish both of these inputs, we adapt a Rankin--Selberg regularization method due to Zagier (1981).

    In addition to the cubic and quartic cases considered in this paper, we expect that the family of Hecke characters of each fixed order $n \geq 3$ over a number field $K \supset \Q(\zeta_n)$ is \emph{not} perfectly orthogonal. We provide a precise conjecture for the operator norm of these ensembles for each $n$.
\end{abstract}

\maketitle

\tableofcontents


\section{Introduction}

Let $\mathcal{X}$ be a finite set. Given a family $\mathcal{F}$ of harmonics $f: \mathcal{X} \rightarrow \mathbb{C}$,
a \emph{large sieve inequality} is an $L^2$-estimate 
for linear forms in $f \in \mathcal{F}$. That is, an inequality of the form
\begin{equation*}
\sum_{f \in \mathcal{F}} \Big | \sum_{x \in \mathcal{X}} a _x f(x) \Big |^2 \leq \Xi(\mathcal{F},\mathcal{X})  \sum_{x \in \mathcal{X}} |a_x|^2
\end{equation*}
for any $\mathbb{C}$-valued sequence $\boldsymbol{a}:=(a_x)_{x \in \mathcal{X}}$, and some quantity $\Xi(\mathcal{F},\mathcal{X}) \geq 0$.
The Cauchy--Schwarz inequality, and general principles for bilinear forms \cite[\S 7.3]{IK}, imply that
\begin{equation} \label{eq:bounds}
|\mathcal{F}|+|\mathcal{X}| \ll \Xi(\mathcal{F},\mathcal{X}) \leq |\mathcal{F}| |\mathcal{X}|.
\end{equation}

The asymptotic size of $\Xi(\mathcal{F},\mathcal{X})$ quantitatively measures
the orthogonality exhibited by $\mathcal{F}$ in a format that is useful in applications. The upper bound $|\mathcal{F}| |\mathcal{X}|$ in \eqref{eq:bounds} is the \emph{trivial bound}, and is often not useful. We refer to an estimate of the form
\begin{equation} \label{eq:perfect}
    \Xi(\mathcal{F},\mathcal{X}) \ll_\varepsilon (|\mathcal{F}| |\mathcal{X}|)^{\varepsilon} \cdot (|\mathcal{F}|+|\mathcal{X}|)
\end{equation}
that holds for all fixed $\varepsilon > 0$ as a perfectly orthogonal large sieve bound for the family $\mathcal{F}$, since it suffices for most applications\footnote{Though removing the $(|\mathcal{F}| |\mathcal{X}|)^{\varepsilon}$ term is important in some cases, see for instance \cite{Xia24,SS24}.}. The classical multiplicative large sieve (for Dirichlet characters) is perfectly orthogonal, and has a long history and numerous important consequences -- we refer the reader to \cite{Mon71} for more details. 

Let $n \geq 2$ be an integer and $K$ be a number field containing the $n$-th roots of unity, with ring of integers $\mathcal{O}_K$. The $n$-th power residue symbol in $K$ is defined, for each prime ideal $\mathfrak{p} \unlhd \mathcal{O}_K$ coprime to $n$ and $\alpha \in \mathcal{O}_K$ coprime to $\mathfrak{p}$, as the unique $n$-th root of unity satisfying\footnote{Note that $N(\mathfrak{p}) \equiv 1 \pmod{n}$ since $K$ contains the $n$-th roots of unity.}
\begin{equation*}
    \Big(\frac{\alpha}{\mathfrak{p}}\Big)_n \equiv \alpha^{\frac{N(\mathfrak{p})-1}{n}} \pmod{\mathfrak{p}}.
\end{equation*}
If $(\alpha, \mathfrak{p}) \neq 1$, we set $\big(\frac{\alpha}{\mathfrak{p}}\big)_n := 0$. The symbol can be extended multiplicatively to any ideal $\mathfrak{b} \unlhd \mathcal{O}_K$ with $(\mathfrak{b}, n) = 1$ by
\begin{equation*}
     \Big(\frac{\alpha}{\mathfrak{b}}\Big)_n  := \prod_{\substack{\mathfrak{p} \text{ prime}}} \Big(\frac{\alpha}{\mathfrak{p}}\Big)^{\ord_\mathfrak{p}(\mathfrak{b})}_n.
\end{equation*}

The $n$-th power residue symbols can be extended to a family of Hecke characters of order $n$. This extension was constructed by Fisher and Friedberg \cite{FF04} for $n=2$, and by Friedberg, the third author, and Lieman \cite{FHL} for general $n$. A detailed description is given in \cite[Section 2]{GL13}.

For the family of quadratic characters, corresponding to our setting with $n=2$, an influential result of Heath-Brown \cite{HB95} gives a perfectly orthogonal large sieve inequality for $K= \Q$. This was extended to number fields by Goldmakher and Louvel \cite{GL13}. Heath-Brown \cite{HB} also proved a large sieve inequality for cubic characters ($n=3$) over $\Q(\zeta_3)$ which gave partial orthogonality. The same bound in the cubic case was later established by Blomer, Goldmakher, and Louvel \cite{BGL} for every $n \geq 3$ and every number field $K \supset \Q(\zeta_n)$. Recently, Liu \cite{Liu} made the $(|\mathcal{F}| |\mathcal{X}|)^{\varepsilon}$-term in \eqref{eq:perfect} explicit for Heath-Brown's quadratic large sieve \cite{HB95}.

Whether a given family $\mathcal{F}$ satisfies \eqref{eq:perfect} is not a well-understood phenomenon.
Our knowledge of families $\mathcal{F}$ that demonstrably violate \eqref{eq:perfect} is limited to a few sporadic examples in the literature. It was a folklore belief \cite{BGL,HB} that the cubic large sieve could be improved to a perfectly orthogonal bound. This belief was disproved in 2024 under the Generalized Riemann Hypothesis (GRH) by the second author and Radziwi{\l\l} \cite{DR}, who showed that Heath-Brown's cubic bound is essentially optimal. This inspired their resolution (under GRH) of Patterson's 1978 conjecture \cite{Pat3} concerning the bias of cubic Gauss sums over primes \cite{DR}. Another notable example of a non-perfectly-orthogonal ensemble is the
$\Gamma_1(q)$ spectral large sieve of Iwaniec and Li \cite{IL07}.

\subsection{Results}

In this paper we consider the $n$-th order large sieve for $n=3$ and $n=4$. We show unconditionally that the cubic large sieve is essentially optimal, and that the quartic large sieve is not perfectly orthogonal. We also give a precise conjecture for the operator norm of the $n$-th order large sieve and describe how our methods may apply to prove the lower bound.

To describe our results, let $n \in \{3, 4\}$. We work over the simplest possible (quadratic) fields $K := \Q(\zeta_n)$, so $\mathcal{O}_{K}=\mathbb{Z}[\zeta_n]$. Denote the norm $N(a):=N_{K/\mathbb{Q}}(a)= |a|^2$ for $a \in K$. We study the operator norm of the corresponding $n$-th order large sieves over $K$, namely
\begin{equation}\label{eq:Xi_3_def}
    \Xi_3(A,B):= \sup_{\boldsymbol{\beta} \neq 0} \frac{1}{\norm{\boldsymbol{\beta}}^2_2} \sum_{\substack{ a \in \mathbb{Z}[\zeta_3] \\ a \equiv 1 \pmod{3} \\ N(a) \leq A}} \mu^2(a) \cdot \Bigg | \sum_{\substack{ b \in \mathbb{Z}[\zeta_3] \\ b \equiv 1 \pmod{3} \\ N(b) \leq B }} \mu^2(b) \beta_b \Big(\frac{a}{b}\Big)_3  \Bigg |^2
\end{equation}
and (denoting $\lambda := 1 + i$)
\begin{equation}\label{eq:Xi_def}
    \Xi_4(A,B):= \sup_{\boldsymbol{\beta} \neq 0} \frac{1}{\norm{\boldsymbol{\beta}}^2_2} \sum_{\substack{ a \in \mathbb{Z}[i] \\ a \equiv 1 \pmod{\lambda^3} \\ N(a) \leq A}} \mu^2(a) \cdot \Bigg | \sum_{\substack{ b \in \mathbb{Z}[i] \\ b \equiv 1 \pmod{\lambda^3} \\ N(b) \leq B }} \mu^2(b) \beta_b \Big(\frac{a}{b}\Big)_4  \Bigg |^2.
\end{equation}

The best known upper bound in both cases is
\begin{equation} \label{eq:BGL_upper_bound}
    \Xi_n(A,B) \ll_{\varepsilon} (AB)^{\varepsilon}(A+B+(AB)^{2/3}),
\end{equation}
due to Blomer, Goldmakher, and Louvel \cite[Theorem~1.3]{BGL}, 
and previously to Heath-Brown \cite[Theorem~2]{HB} in the cubic case.
We prove the following lower bounds, ruling out perfect orthogonality -- that is, $\Xi_n(A,B)$ is \textbf{not} $ \ll_\varepsilon (AB)^{\varepsilon} (A + B)$
for $n \in \{3,4\}$.

\begin{theorem}[Cubic large sieve lower bound] \label{thm:cubic_main}
    Let $\varepsilon > 0$. Then for any $A, B \geq 1$,
    \begin{equation} \label{eq:cubic_lower_bound}
        \Xi_3(A,B)  \gg_\varepsilon (AB)^{-\varepsilon} (A + B + (AB)^{2/3}).
    \end{equation}
\end{theorem}

Note that Theorem \ref{thm:cubic_main} makes \cite[Theorem~1.4]{DR} unconditional\footnote{Unconditional versions of \cite[Theorems~1.1--1.3]{DR}, which concern primes, remain well out of reach.}.

\begin{theorem}[Quartic large sieve lower bound] \label{thm:quartic_main}
     Let $\varepsilon > 0$. Then for any $A, B \geq 1$,
    \begin{equation} \label{eq:quartic_lower_bound}
        \Xi_4(A,B)  \gg_\varepsilon (AB)^{-\varepsilon} (A + B + A^{3/4} B^{1/2} + A^{1/2} B^{3/4}).
    \end{equation}
\end{theorem}

The obstruction to perfect orthogonality in \eqref{eq:cubic_lower_bound} and \eqref{eq:quartic_lower_bound} is most pronounced for $A \asymp B \asymp X$, when the lower bound in \eqref{eq:cubic_lower_bound} and \eqref{eq:quartic_lower_bound} is $\gg_\varepsilon X^{4/3-\varepsilon}$ and $\gg_\varepsilon X^{5/4-\varepsilon}$ respectively. 

Note that \cref{thm:cubic_main} essentially matches the upper bound \eqref{eq:BGL_upper_bound}, showing optimality of the cubic large sieve. \cref{thm:quartic_main} \emph{does not} match \eqref{eq:BGL_upper_bound}, but we believe that it is essentially tight. We therefore make the following conjecture.

\begin{conjecture}[Quartic large sieve upper bound] \label{opnormconj}
     Let $\varepsilon > 0$. Then for any $A, B \geq 1$,
    \begin{equation} \label{Xistep}
        \Xi_4(A,B) \ll_{\varepsilon} (AB)^{\varepsilon}(A + B + A^{3/4} B^{1/2} + A^{1/2} B^{3/4}).
    \end{equation}
\end{conjecture}

\subsection{The general case}
 
Theorems \ref{thm:cubic_main} and \ref{thm:quartic_main} and \cref{opnormconj} should generalize in a
natural way to the previously described family of Hecke characters of fixed order $n \geq 3$ over $K \supset \Q(\zeta_n)$. More precisely, we expect the analogous operator norm (considered in \cite[Theorem~1.3]{BGL}) to satisfy
\begin{equation}\label{eq:general_n_conjecture}
    \Xi_n(A,B) = (AB)^{o(1)}(A+B+A^{1-1/n} B^{2/n}+A^{2/n} B^{1-1/n}) \quad  \text{as} \quad \min(A, B) \to \infty.
\end{equation}

We have chosen to focus on $n \in \{3, 4\}$ to simplify the exposition and elucidate the new ideas, since many technical inputs are already available in the literature, or adaptable from it. It is also convenient that $\Q(\zeta_n)$ is a quadratic field of class number one. With more technical work, the lower bound of \eqref{eq:general_n_conjecture} should be within reach of current tools for all $n \geq 5$. We discuss some of the necessary ingredients in \cref{sec:cubic}, but do not prove any results in this generality. Any power-saving improvement on the $(AB)^{2/3+\varepsilon}$ term in \eqref{eq:BGL_upper_bound}
for any $n \geq 4$ would be a substantial advance in the theory.

In the next section we give a quick proof of \cref{thm:cubic_main} using work of David, Stucky, and the first two authors \cite{DDDS24} on cubic characters. The proof uses Patterson's evaluation of the Fourier coefficients of the cubic metaplectic theta function \cite{Pat1}. No such evaluation is available for $n\geq 4$, and this is the main obstruction we must overcome (for $n = 4$) in the rest of the paper. For the application to the large sieve lower bound, it suffices to show that these coefficients are large on average. We obtain this from Rankin--Selberg theory applied to quartic metaplectic Eisenstein series. Naturally some regularization is necessary, so we develop a version of Zagier's regularization method \cite{Zag} for a class of discrete subgroups of $\SL_2(\C)$, which may be of independent interest.

\subsection*{Acknowledgments}
The first author thanks Farrell Brumley for fruitful discussions on scattering matrices, and the Centre de recherches math{\'e}matiques (CRM) at Universit{\'e} de Montr{\'e}al for excellent working conditions.
The second author was supported by the NSF
Standard Grant DMS-2452303, an AMS-Simons Travel Grant, and the Richard A. Duke
Endowed Fund at the Georgia Tech School of Mathematics.


\section{The cubic lower bound and discussion of the proof} \label{sec:cubic}

\subsection{The cubic large sieve}
We work over $K = \Q(\zeta_3)$ in this section only. For any $r, \alpha \in \Z[\zeta_3]$ with $\alpha\equiv 1 \pmod{3}$ squarefree, denote
\begin{equation} \label{psizetadef}
    \psi^{(3)}_\alpha(r, s) := \sum_{\substack{c \in \Z[\zeta_3] \\ c\equiv 1 \pmod{3} \\ (c, \alpha)=1}} \frac{g_3(r, c)}{N(c)^s} \qquad \text{ for } \qquad g_3(r,c):=\sum_{d \pmod{c}} \Big(\frac{d}{c}\Big)_3 \check{e} \Big( \frac{r d}{c} \Big).
\end{equation}
Here, $\check{e}(z):=e^{2 \pi i \text{Tr}_{K/\Q}(z)}=e^{2 \pi i(z+\overline{z})}$. Then $\psi^{(3)}_\alpha(r, s)$ converges absolutely for $\Re(s) > \frac{3}{2}$. We also denote $\psi^{(3)}(r, s) := \psi^{(3)}_1(r, s)$ and $\widetilde{g}_3(c) := g_3(1, c) \cdot N(c)^{-1/2}$, where $N(c) := |{c}|^2$ is the norm in $K/\Q$. 

\begin{proof}[Proof of \cref{thm:cubic_main}]
    The duality principle for the large sieve \cite[(7.9) to (7.11)]{IK} and reciprocity for the cubic symbol imply \cite[Lemma~3.1]{BGL} that
    \begin{equation} \label{eq:duality_cubic}
        \Xi_3(A, B) \asymp \Xi_3(B, A).
    \end{equation}
    Thus we may assume $A \geq B \geq 1$.
    
    If $A \geq B^2$, then the trivial bound $\Xi_3(A, B) \gg A$ suffices for \cref{thm:cubic_main}. 
    This bound is realized by choosing the sequence $\boldsymbol{\beta}$ such that $\beta_b:=\mathbf{1}_{b=1}$. 
    
    Thus we may assume from now on that $B \leq A < B^2$. Choose $\boldsymbol{\beta}$ given by
    \begin{equation*}
        \beta_b = \overline{\widetilde{g}_3(b)} \cdot \mathbf{1}_{b \equiv 1 \pmod{3}} \cdot H\Big(\frac{N(b)}{B} \Big),
    \end{equation*}
    where $H:(0,\infty) \to \mathbb{R}_{\geq 0}$ is a smooth function with compact support in $(\frac{1}{2},1)$, which is fixed and not identically zero. 
    It follows from $|{\widetilde{g}_3(b)}| = \mu^2(b)$ for $b \equiv 1 \pmod{3}$ that
    \begin{equation}\label{eq:beta_L2_norm_cubic}
        \norm{\boldsymbol{\beta}}^2_2 \asymp B
    \end{equation}
    for $B$ sufficiently large. We assume that this is the case, otherwise $1 \leq A, B \ll 1$ and \cref{thm:cubic_main} is trivial.

    The complex conjugate of the sum over $b$ in \eqref{eq:Xi_3_def} is equal to
    \begin{align}\label{eq:b_sum_cubic}
        \sum_{\substack{ b \in \mathbb{Z}[\zeta_3] \\ b \equiv 1 \pmod{3}}}  \overline{\Big(\frac{a}{b}\Big)_3} \widetilde{g}_3(b) H\Big(\frac{N(b)}{B}\Big) =  \mathcal{P}_B(a)  + \mathcal{R}_B(a, \tfrac{1}{2}),
    \end{align}
    where we applied \cite[Lemma 6.7]{DDDS24} to obtain a polar term
    \begin{align}\label{eq:polar_term_cubic}
        \mathcal{P}_B(a) := \frac{C_1 \widetilde{H}(\tfrac{5}{6})\overline{\widetilde{g}_3(a)}}{N(a)^{1/6}} B^{5/6} \prod_{\substack{\pi \text{ prime} \\ \pi\equiv 1 \pmod{3} \\ \pi \mid a}} \Big(1 + \frac{1}{N(\pi)} \Big)^{-1},
    \end{align}
    and a remainder term
    \begin{align}\label{eq:remainder_term_cubic}
        \mathcal{R}_B(a, \tfrac{1}{2}) \ll_{\varepsilon} B^{1/2+\varepsilon} \sum_{\substack{d \in \Z[\zeta_3] \\ d\equiv 1\pmod{3} \\ d\mid a}} \frac{1}{N(d)^{1/2+\varepsilon}} \int_{-\infty}^\infty \frac{|\psi^{(3)}(\frac{a}{d}, 1+\varepsilon +iy)|}{(1+|y|)^{100}} \, dy.
    \end{align}
    From the proof of \cite[Lemma 6.7]{DDDS24}, it follows that the constant $C_1$ is given by $c_0 c_1$, where $c_0 \neq 0$ is given in \cite[(6.10)]{DDDS24} and $c_1 = 3^3$ by \cite[(6.8)]{DDDS24}. Hence $C_1 \neq 0$. Thus since the product over primes in \eqref{eq:polar_term_cubic} is $\gg_\varepsilon N(a)^{-\varepsilon}$ and $|{\widetilde{g}_3(a)}| = \mu^2(a)$, we have $\mathcal{P}_B(a) \gg_\varepsilon \mu^2(a) N(a)^{-1/6-\varepsilon} B^{5/6}$.

   It follows from the AM-GM inequality that $|{x+y}|^2 \geq \frac{1}{2}|{x}|^2 - |{y}|^2$, and so we conclude from \eqref{eq:b_sum_cubic} that
    \begin{equation} \label{eq:AMGM_cubic}
        \sum_{\substack{ a \in \mathbb{Z}[\zeta_3] \\ a \equiv 1 \pmod{3} \\ N(a) \leq A}} \mu^2(a) \cdot \Bigg | \sum_{\substack{ b \in \mathbb{Z}[\zeta_3] \\ b \equiv 1 \pmod{3} \\ N(b) \leq B }} \mu^2(b) \beta_b \Big(\frac{a}{b}\Big)_3  \Bigg |^2 \geq \frac{\mathcal{M}_3(A, B)}{2} - \mathcal{E}_3(A, B),
    \end{equation}
    where from \eqref{eq:polar_term_cubic} we have  
    \begin{equation} \label{eq:Mdef_cubic}
        \mathcal{M}_3(A, B):= \sum_{\substack{ a \in \mathbb{Z}[\zeta_3] \\ a \equiv 1 \pmod{3} \\ N(a) \leq A}} \mu^2(a) \cdot |{\mathcal{P}_B(a)}|^2 \gg_\varepsilon \sum_{\substack{ a \in \mathbb{Z}[\zeta_3] \\ a \equiv 1 \pmod{3} \\ N(a) \leq A}} \frac{\mu^2(a) B^{5/3}}{N(a)^{1/3+\varepsilon}} \gg A^{2/3 - \varepsilon} B^{5/3},
    \end{equation}
    and from \eqref{eq:remainder_term_cubic} and Cauchy--Schwarz (twice) we have
    \begin{align*}
        \mathcal{E}_3(A, B) := \sum_{\substack{ a \in \mathbb{Z}[\zeta_3] \\ a \equiv 1 \pmod{3} \\ N(a) \leq A}} \mu^2(a) \cdot |{\mathcal{R}_B(a, \tfrac{1}{2})}|^2 \ll_\varepsilon B^{1+\varepsilon} \int_{-\infty}^\infty \sumtwo_{\substack{ a, d \in \mathbb{Z}[\zeta_3] \\ a, d \equiv 1 \pmod{3} \\ N(a) \leq A, \ d \mid a}} \frac{|\psi^{(3)}(\frac{a}{d}, 1+\varepsilon +iy)|^2}{(1+|y|)^{100}} \, dy.
    \end{align*}
    Writing $m = \frac{a}{d}$, using a divisor bound, and applying the second moment bound of \cite[Lemma 6.6]{DDDS24} in dyadic intervals, observe that
    \begin{equation}\label{eq:E_bound_cubic}
        \mathcal{E}_3(A, B) \ll_\varepsilon (AB)^{1+\varepsilon} \int_{-\infty}^\infty \sum_{\substack{0 \neq m \in \mathbb{Z}[\zeta_3] \\ N(m) \leq A}} \frac{|\psi^{(3)}(m, 1+\varepsilon +iy)|^2}{N(m) \cdot (1+|y|)^{100}} \, dy \ll_\varepsilon (AB)^{1+\varepsilon}.
    \end{equation}

    We now consider two cases. First, if $\mathcal{M}_3(A, B) \geq 4 \cdot \mathcal{E}_3(A, B)$, then \eqref{eq:beta_L2_norm_cubic}, \eqref{eq:AMGM_cubic}, and \eqref{eq:Mdef_cubic} imply
    \begin{equation*}
        \Xi_3(A, B) \gg \frac{\mathcal{M}_3(A, B)}{\norm{\boldsymbol{\beta}}_2^2} \gg_\varepsilon \frac{A^{2/3 - \varepsilon} B^{5/3}}{B} \gg (AB)^{2/3-\varepsilon}.
    \end{equation*}
    This implies \cref{thm:cubic_main}, since $B \leq A < B^2$. 
    
    Now consider the remaining case $\mathcal{M}_3(A, B) < 4 \cdot \mathcal{E}_3(A, B)$. Using \eqref{eq:Mdef_cubic} and \eqref{eq:E_bound_cubic}, we obtain
    \begin{equation*}
        A^{2/3 - \varepsilon} B^{5/3} \ll_\varepsilon \mathcal{M}_3(A, B) \ll \mathcal{E}_3(A, B) \ll_\varepsilon (AB)^{1+\varepsilon},
    \end{equation*}
    hence $B^{2/3-\varepsilon} \ll_\varepsilon A^{1/3+2\varepsilon}$. In this range of parameters, the trivial bound $\Xi_3(A, B) \gg A$ suffices to obtain \cref{thm:cubic_main} after adjusting $\varepsilon > 0$. This finishes the proof.
\end{proof}

\subsection{Discussion of the proof and generalizations}

The proof of \cref{thm:cubic_main} relied on two main ingredients: a Lindel\"{o}f-on-average upper bound (in the twist aspect) for the second moment of $\psi^{(3)}$, of the form 
\begin{equation}\label{eq:second_moment}
    \sum_{\substack{0 \neq a \in \mathbb{Z}[\zeta_3] \\ N(a) \leq A}} |{\psi^{(3)}(a, 1 + \varepsilon +it)}|^2 \ll_\varepsilon A^{1+\varepsilon} (1+|{t}|)^{R}
\end{equation}
for some fixed $R \in \R_{>0}$, and the exact evaluation of the Fourier coefficients $\tau_3(a)$ of the cubic metaplectic theta function, due to Patterson \cite[Theorem~8.1]{Pat1}. 

We point out that \eqref{eq:second_moment} was not used in \cite{DR}. 
Let $\zeta_{\Q(\zeta_3)}(s)$ denote the Dedekind zeta function attached to $\Q(\zeta_3)$.
Instead of \eqref{eq:second_moment}, only the convexity bound \cite[(6.13)]{DDDS24}
\begin{equation} \label{convexbdsketch}
\zeta_{\Q(\zeta_3)}(3s-2) \cdot \psi^{(3)}(a,s) \ll_\varepsilon N(a)^{3/4-\Re(s)/2+\varepsilon} (1+|s|)^{100},
\end{equation}
for $1/2+\varepsilon \leq \Re(s) \leq 3/2$ and $|s-4/3| \geq \varepsilon$, was used in the proof of \cite[Theorem~1.4]{DR}. This meant that 
in the contour shifting argument (analogous to \eqref{eq:b_sum_cubic}) that occurs in \cite{DR}, one needs to move the contour further to the left and understand $\psi^{(3)}(a,s)$ in the half-plane $\Re(s) \geq 5/6+\varepsilon$, instead of $\Re(s) \geq 1+\varepsilon$. The GRH assumption in \cite[Theorem~1.4]{DR}
is needed to rule out poles of $\psi^{(3)}(a, s)$ coming from potential zeros of $\zeta_{\Q(\zeta_3)}(3s-2)$ in \eqref{convexbdsketch}. Interestingly, a similar argument using GRH and convexity does not seem to give the correct large sieve lower bound for $n\geq 4$.

An analogue of \eqref{eq:second_moment} is available for $n=4$ through Rankin--Selberg theory, since the corresponding Dirichlet series $\psi^{(4)}(a,s)$
 are Fourier coefficients of quartic metaplectic Eisenstein series. We expect similar arguments to work for all $n \geq 5$.
Note that the quartic large sieve of Blomer, Goldmakher, and Louvel \cite[Theorem~1.3]{BGL} could also be applied to obtain the Lindel\"{o}f-on-average bound for $\psi^{(4)}(a,s)$, as in the proof of \eqref{eq:second_moment} in \cite[Lemma 6.6]{DDDS24}.
However, it is more convenient to use Rankin--Selberg theory, since it is already needed to obtain lower bounds for \eqref{eq:tau_second_moment} below.

Determining the Fourier coefficients $\tau_4(a)$ of the quartic theta function is a central open problem in the theory of metaplectic forms; the analogous problem becomes even harder for general $n \geq 4$ \cite{BH,Deli,EckPat,KP,Pat2}. Thus we need a different approach to show the presence of a bias coming from the poles of $\psi^{(4)}(a,s)$. We once again achieve this through Rankin--Selberg theory, now applied to the quartic theta function. This gives a \emph{lower bound} of the expected order of magnitude for
\begin{equation}\label{eq:tau_second_moment}
    \sum_{\substack{0 \neq a \in \mathbb{Z}[i] \\ N(a) \leq A}} |{\tau_4(a)}|^2.
\end{equation}
We expect an analogous result to hold for general $n \geq 5$ by similar arguments.

However, there are two significant technical complications that we have so far ignored (in addition to congruence conditions, which we continue to ignore without further comment). First, we actually need a lower bound for the sum in \eqref{eq:tau_second_moment} restricted to \emph{squarefree} $a$, due to the presence of $\mu^2(a)$ in the large sieve \eqref{eq:Xi_def}. Second, the term $\tau_4(a)$ must be replaced by a linear combination of terms of the form $\tau_4(a/d)$. This is more clearly illustrated in the case $n=3$, where the sum in \eqref{eq:b_sum_cubic} is related by Perron's formula to the Dirichlet series $\psi_a^{(3)}(a, s)$ given in \eqref{psizetadef}, instead of the simpler $\psi^{(3)}(a, s)$, as our heuristics may suggest. To understand the polar term, one needs to decompose the former in terms of the latter as in \cite[Corollary 6.2]{DDDS24}, obtaining a linear combination of terms of the form $\psi^{(3)}(a/d, s)$ for $d\mid a$.

It should be possible to overcome both issues (for general $n$) using Rankin--Selberg theory for metaplectic forms with level structure. Indeed, after sieving for $\mu^2(a)$ in the first issue or opening the square in the second issue, it would suffice to obtain asymptotics, with some uniformity in $r_1$ and $r_2$, for sums roughly of the form
\begin{equation*}
    \sum_{\substack{0 < N(a) \leq A}} \tau_n(r_1 a) \overline{\tau_n(r_2 a)}.
\end{equation*} 

In this paper we avoid the technical difficulties of considering level structure (with the required uniformity) by using two tricks in the case $n=4$. In \cref{lemma:squarefree_tau_4}, we apply partial information on $\tau_4(a)$ obtained by Suzuki \cite{Suz1} to restrict \eqref{eq:second_moment} directly to squarefree $a$. Moreover, following \cite[Corollary 6.2]{CDD26}, we observe that in the exceptional case $n=4$, the linear combination of the terms $\tau_4(a/d)$ collapses to a multiple of $\tau_4(a)$. In fact, the proof of that corollary suggests that the analogous Dirichlet series with coprimality condition $\psi_a^{(n)}(a, s)$ should be a linear combination of terms roughly of the form $\psi^{(n)}(a d^{n-4}, s)$ for $d \mid a$. 
Hence, when $n=4$, no additional level structure arising from the coprimality condition with $a$ is needed, which vastly simplifies our work.


\section{Notation}

\subsection{Number fields}

Throughout the rest of the paper, let $K:=\mathbb{Q}(i)$ denote the Gaussian quadratic field, which has ring of integers $\mathcal{O}_{K}=\mathbb{Z}[i]$ and discriminant $d_K=-4$. Let $N(a):=N_{K/\mathbb{Q}}(a)=|a|^2$ denote the norm of $a \in K$, and $\boldsymbol{\mu}_K:=\{\pm 1, \pm i \}$ denote the group of units in $\mathcal{O}_K$. Let $\lambda:=1+i$, which generates the unique ramified prime ideal in $K$. Each ideal $0 \neq \mathfrak{c} \unlhd \mathcal{O}_K$ is principal, and if $(\mathfrak{c},\lambda)=1$ then there is a unique $c \in \mathcal{O}_K$ such that $\mathfrak{c}=(c)$ and $c \equiv 1 \pmod{\lambda^{3}}$.

\subsection{Gauss sums}

Let $\check{e}(z):=e^{2 \pi i \text{Tr}_{K/\Q}(z)}=e^{2 \pi i(z+\overline{z})}$ as before. For $c \in \Z[i]$ with $(c,\lambda) =1$ and $\nu \in \lambda^{-2} \mathbb{Z}[i]$, the quartic and quadratic Gauss sums over $\Z[i]$ are defined respectively as
\begin{equation*}
    g_4(\nu,c) := \sum_{d \pmod c} \Big(\frac{d}{c}\Big)_4 \check{e} \Big(\frac{\nu d}{c}\Big) \qquad \text{and} \qquad g_2(\nu,c):= \sum_{d \pmod c} \Big(\frac{d}{c}\Big)^2_4 \check{e} \Big(\frac{\nu d}{c} \Big).
\end{equation*}
When $\nu =1$, we simply write $g_4(c)$ and $g_2(c)$, respectively. We denote the normalized versions $\widetilde{g}_4(\nu, c) := N(c)^{-1/2} g_4(\nu, c)$ and $\widetilde{g}_4(c) := N(c)^{-1/2} g_4(c)$, and analogously for $\widetilde{g}_2$.

\subsection{Asymptotic inequalities}
Throughout the paper, the quantity $\varepsilon>0$ denotes an arbitrarily small and fixed number, possibly different in each instance.
For $n \in \mathbb{N}$ and $N>0$, we write $n \sim N$ to mean $N \leq n < 2N$. Moreover $f \ll g$ means that there exists a constant $c > 0$ such that $|f| \leq c |g|$, while $f \asymp g$ denotes $f \ll g \ll f$. Dependence of implied constants on parameters will be indicated as subscripts on the notations above. Implied constants in the body of the paper are allowed to depend on the implicit constants in $\asymp$ or $\ll$ notation.

 
\section{Group actions and Eisenstein series}

In this section we present some requisite material on Eisenstein series.

\subsection{Isometries of $\mathbb{H}^3$ and the Laplacian} \label{H3sec}
Let $\mathbb{H}^{3}$ denote the hyperbolic 3-space, which we model as the upper half-space $\mathbb{C} \times \mathbb{R}^{+}$.
Embed $\mathbb{C}$ and $\mathbb{H}^3$ in the (real) Hamilton quaternions  by identifying 
$i=\sqrt{-1}$ with $\hat{i}$ and 
$w=(z,v)=(x+iy,v) \in \mathbb{H}^3$ with $x+y \hat{i}+v \hat{j}$, where 
$1,\hat{i},\hat{j},\hat{k}$ denote the unit quaternions.  
Alternatively, a point $w=(z,v)$ is represented by the matrix $w=\pmatrix z {-v} {v} {\overline{z}}$,
and $u \in \mathbb{C}$ is represented by the matrix $\widetilde{u}=\pmatrix u 0 0 {\overline{u}}$.
The metric is $ds^2 = \frac{|dz|^2 + dv^2}{v^2}$, and $\operatorname{SL}_2(\mathbb{C})$ acts isometrically on $\mathbb{H}^3$ by
\begin{equation*}
\gamma w=(\widetilde{a}w+\widetilde{b})(\widetilde{c}w+\widetilde{d})^{-1}, \qquad  \text{ where }\gamma=\begin{pMatrix}
a b
c d 
\end{pMatrix} \in \operatorname{SL}_2(\mathbb{C}) \text{ and }
w \in \mathbb{H}^3.
\end{equation*}
In coordinates,
\begin{equation*}
\gamma w= \bigg( \frac{(az+b) \overline{(cz+d)}+a \overline{c} v^2}{|cz+d|^2 +|c|^2 v^2}, 
\frac{v}{|cz+d|^2+|c|^2 v^2} \bigg)
\qquad \text{ for } w=(z,v).
\end{equation*}
The action of $\operatorname{SL}_2(\mathbb{C})$ on
$\mathbb{H}^3$ is transitive, and the stabilizer of a point is a conjugate of $\operatorname{SU}_2(\mathbb{C})$.
The Laplace operator
$\Delta:=v^2( \partial^2/\partial x^2+
\partial^2/\partial y^2+\partial^2/\partial v^2)-v \partial/\partial v$
then acts on $C^{\infty}(\mathbb{H}^3)$
and commutes with the action of $\operatorname{SL}_2(\mathbb{C})$ on $C^{\infty}(\mathbb{H}^3)$.

Two important subgroups of $\operatorname{SL}_2(\mathbb{C})$ are the Borel subgroup
\begin{equation*}
    B(\C) := \Big\{ \left(\begin{smallmatrix}
    \alpha & \mu \\
    0 & \alpha^{-1}
\end{smallmatrix}\right)  : 0\neq \alpha \in \C, \mu \in \C \Big \},
\end{equation*}
which is the stabilizer of $\infty$, and its unipotent subgroup
\begin{equation*}
    N(\C) := \Big\{ \left(\begin{smallmatrix}
    1 & \mu \\
    0 & 1
\end{smallmatrix}\right)  : \mu \in \C \Big \}.
\end{equation*}

\subsection{Eisenstein series} \label{metaeissec}
Here we follow \cite[\S2-\S4]{Pat1} and \cite[\S 2]{Dia} with some mild generalizations
and modifications. Let $K=\mathbb{Q}(i)$ be as above and
$\Gamma \subseteq \operatorname{SL}_2(\mathcal{O}_{K})$ be a co-finite but not co-compact subgroup.
Let $d\mu(w)=\frac{dx\, dy\, dv}{v^3}$ be the standard volume measure on $\mathbb{H}^3$, which is $\operatorname{SL}_2(\mathbb{C})$-invariant.

Let $\kappa \in \C \cup \{\infty\}$ be a cusp of $\Gamma$ with scaling matrix $\sigma_{\kappa} \in \operatorname{SL}_2(\mathcal{O}_{K})$ (which exists since $K$ has class number $1$)
and stabilizer subgroup $\Gamma_{\kappa}$. In other words, $\sigma_\kappa \infty = \kappa$ and
\begin{equation*}
    \Gamma_{\kappa} := \{ \gamma \in \Gamma : \gamma \kappa = \kappa\} = \Gamma \cap \sigma_\kappa B(\C) \sigma_\kappa^{-1}.
\end{equation*}
Consider also the unipotent subgroup
\begin{equation*}
    \Gamma'_\kappa := \Gamma \cap \sigma_\kappa N(\C) \sigma_\kappa^{-1}.
\end{equation*}
Observe that $\sigma_\kappa^{-1} \Gamma_\kappa \sigma_\kappa = (\sigma_\kappa^{-1} \Gamma \sigma_\kappa)_\infty$ and $\sigma_\kappa^{-1} \Gamma'_\kappa \sigma_\kappa = (\sigma_\kappa^{-1} \Gamma \sigma_\kappa)'_\infty$. For this reason, we make two simplifying assumptions on our group: assume $\Gamma$ is a \textit{normal} subgroup of $\operatorname{SL}_2(\mathcal{O}_K)$, and that $\Gamma_\infty = \Gamma'_\infty$. This implies that $\Gamma_{\kappa} = \sigma_\kappa \Gamma_\infty \sigma_\kappa^{-1} = \sigma_\kappa \Gamma'_\infty \sigma_\kappa^{-1} = \Gamma'_{\kappa}$ for every cusp $\kappa$.

Set 
\begin{equation*}
    \Lambda=\Big \{ \mu \in \mathbb{C}:   \pmatrix 1 \mu 0 1  \in \Gamma_\infty \Big \},
\end{equation*}
which is a lattice in $\C$ by the definition of a cusp, and consider its dual lattice 
\begin{equation*}
    \Lambda^{*} = \Big \{ \delta \in \mathbb{C}: \mathrm{Tr}(\mu \delta) \in \mathbb{Z} \text{ for all } \mu \in \Lambda \Big\}. 
\end{equation*}
Write
\begin{equation} \label{voldef}
    \mathcal{V}:=\vol(\mathbb{C}/\Lambda)
\end{equation}
for the co-volume of the lattice $\Lambda$ with respect to the measure $dx \, dy$.

Let $\psi: \Gamma \to S^1$
be a unitary character. The function
\begin{equation*}
    \Lambda \to S^1: \mu \mapsto \psi \big(\sigma_{\kappa} \pmatrix {1} {\mu} {0} {1}\sigma^{-1}_{\kappa} \big)
\end{equation*}
is a homomorphism, so there exists $h_{\kappa} \in \mathbb{C}$ such that
\begin{equation} \label{homomorph}
    \psi \big(\sigma_{\kappa} \pmatrix {1} {\mu} {0} {1}\sigma^{-1}_{\kappa} \big) =\check{e}(h_{\kappa} \mu ),
\end{equation}
where $\check{e}(z):=e^{2 \pi i (z+\overline{z})}$ for $z \in \mathbb{C}$.
If the homomorphism in \eqref{homomorph} is trivial (i.e.\ $h_{\kappa}\in \Lambda^*$), then the cusp $\kappa$ is said 
to be an essential cusp with respect to $\psi$. Note that $h_\kappa$ depends on $\sigma_\kappa \in \SL_2(\mathcal{O}_K)$, but a different choice does not change whether a cusp is essential (it is equivalent to $\psi(\Gamma_\kappa) = 1$). All cusps of $\Gamma$ are essential with respect to the trivial character.

Fix a set $\{\eta_1, \eta_2, \dots, \eta_m\}$ of representatives for the (inequivalent) cusps of $\Gamma$
that are essential with respect to $\psi$. Denote $\sigma_i := \sigma_{\eta_i}$ (which we fix once and for all) and $\Gamma_i := \Gamma_{\eta_i}$. 
The Eisenstein series on $\Gamma$ attached to the
essential cusp $\eta_i$ is given by 
\begin{equation} \label{eisdef}
    E_{i}(w,s,\psi):=\sum_{\gamma \in \Gamma_{i} \backslash \Gamma } \overline{\psi}(\gamma) v(\sigma_{i}^{-1} \gamma w)^s
\end{equation}
for $w=(z, v) \in \mathbb{H}^3$ and $\Re(s)>2$. Note that it does not depend on the choice of $\sigma_i \in \SL_2(\mathcal{O}_K)$. When $\psi$ is the trivial character, we suppress it from the notation.
The Eisenstein series satisfies the transformation law
\begin{equation*}
    E_{i}(\gamma w,s,\psi)=\psi(\gamma) E_{i}(w,s,\psi) \qquad \text{ for any } \gamma \in \Gamma,
\end{equation*}
and $E_{i}(w,s,\psi)$ is an eigenfunction of the Laplacian $\Delta$.

A slight generalization of \cite[\S3.4, Theorem 4.1]{EGM} gives the Fourier expansion at the essential cusp $\eta_j$, namely 
\begin{align} \label{eisfourier}
    E_{i}(\sigma_{j} w,s,\psi) = \mathbf{1}_{i j} v^s + \sum_{\substack{\delta \in \Lambda^* }} v^{2-s} K (\delta v,s) \varphi_{i j}(\delta,s,\psi) \check{e}( \delta z ), 
\end{align}
where $\mathbf{1}_{i j}$ is the indicator of $i = j$, for $\delta \in \Lambda^*$ and $\Re(s)>2$ we have
\begin{equation} \label{varphidefgen}
    \varphi_{i j}(\delta,s,\psi):=\mathcal{V}^{-1} \sum_{\substack{\gamma=\pmatrix {a} {b} {c} {d}  \in \Gamma_{\infty} \backslash \sigma^{-1}_{i} \Gamma \sigma_{j} / \Gamma_{\infty} \\ c \neq 0 }} \overline{\psi}(\sigma_i \gamma \sigma_j^{-1}) \check{e} \Big( \frac{\delta d}{c} \Big)  N(c)^{-s},   
\end{equation}
and for $u \in \mathbb{C}$ we have
\begin{equation} \label{Kintdef}
    K(u,s):=\int_{\mathbb{C}} (|z|^2+1)^{-s} \check{e}(-uz) \, dx \, dy = \frac{(2 \pi)^s |u|^{s-1} K_{s-1}(4 \pi |u|)}{\Gamma(s)}.
\end{equation}
Here $K_{s-1}$ denotes the standard Bessel function, so in particular $K(0,s)=\frac{\pi}{s-1}$. It will be convenient to set 
\begin{equation} \label{rhozero}
    \rho_{i j}(0,v,s,\psi):=\mathbf{1}_{i j} v^s+v^{2-s} K(0,s)  \varphi_{i j}(0,s,\psi),
\end{equation}
and for $0 \neq \delta \in \Lambda^*$,
\begin{equation} \label{rhononzero}
    \rho_{i j}(\delta,v,s,\psi):= v^{2-s} K (\delta v,s) \varphi_{i j}(\delta,s,\psi).
\end{equation}

\begin{remark}\label{rmk:non_essential_cusps}
    We can also obtain the Fourier expansion of $E_i$ at a non-essential cusp $\kappa$ of $\Gamma$, by observing that $\check{e}(-h_\kappa z) E_i(\sigma_\kappa w, s, \psi)$ is $\Lambda$-periodic in $z$. This gives
    \begin{equation*}
        E_{i}(\sigma_{\kappa} w,s,\psi) = \sum_{\substack{\delta - h_\kappa \in \Lambda^*}} v^{2-s} K (\delta v,s) \varphi_{i \kappa}(\delta,s,\psi) \check{e}( \delta z ).
    \end{equation*} 
    The dual lattice $\Lambda^*$ is shifted by $h_\kappa$, hence there is no constant term, and the decay of $K(\delta v, s)$ for $\delta \neq 0$ implies that $E_i(\sigma_\kappa w, s, \psi)$ decays exponentially in $v$ for $\Re(s) > 2$, since the analogue of \eqref{varphidefgen} gives $\varphi_{i \kappa}(\delta, s, \psi) \ll_{\Re(s)} 1$. 
\end{remark}

A standard argument using the Maass--Selberg formula \cite{Sel,Pat1} shows that the Eisenstein series $E_{i}(w,s,\psi)$ have a meromorphic continuation to all of $\mathbb{C}$, and hence their Fourier coefficients do as well. Furthermore, the poles of $E_{i}(w,s,\psi)$ are among those of
$K(0,s) \varphi_{i j}(0,s,\psi)$ for integers $1\leq j \leq m$. Finally, the Eisenstein series satisfy the following functional equation relating $s \mapsto 2-s$. Consider the $m \times 1$ column vector
\begin{equation} \label{topvec}
    \boldsymbol{E}(w,s,\psi):=\big(E_{i}(w,s,\psi) \big)^{\top}
\end{equation}
and the $m \times m$ scattering matrix
\begin{equation} \label{Phidef}
    \boldsymbol{\Phi}(s,\psi):=\big(K(0,s) \varphi_{i j}(0,s,\psi) \big),
\end{equation}
where rows and columns are indexed respectively by $i$ and $j$. Then
\begin{equation} \label{Efunceq}
    \boldsymbol{E}(w,s,\psi)= \boldsymbol{\Phi}(s,\psi) \cdot \boldsymbol{E}(w,2-s,\psi)
\end{equation}
and
\begin{equation} \label{phifunc}
    \boldsymbol{\Phi}(s,\psi) \cdot \boldsymbol{\Phi}(2-s,\psi)=I.
\end{equation}

\subsection{The case of trivial character}\label{trivial_char_section}

We include some additional details in the special case when $\psi$ is trivial. These details are mainly used in \cref{zagregsec}. Recall that we suppress the trivial character $\psi$ from the notation.

\subsubsection{Standard results} \label{standard}

The following facts can be found\footnote{A change of variable $s \mapsto s+1$ is needed to reconcile \eqref{eisdef} with the definition in \cite[\S 3.2, (2.5)]{EGM}.} in \cite[\S 6.1, Theorem 1.11]{EGM}, and we also provide precise references in \cite{SarCoh}. The poles of the Eisenstein series $E_{i}(w,s)$ are among those of the entries of $\boldsymbol{\Phi}(s)$, by \cite[Corollary 1.66]{SarCoh}. They have no poles in the half-plane $\Re(s)>1$, except for finitely many poles on the real line that all occur in the interval $(1,2]$, see \cite[Proposition 1.56]{SarCoh}. All such poles are simple \cite[Corollary 1.68]{SarCoh}.  For each integer $1\leq i\leq m$, the poles of $E_{i}(w,s)$ in $(1,2]$ coincide with the poles of $\varphi_{i i}(0, s)$ \cite[\S 6.1, Theorem 1.11 (2, 4)]{EGM}. Both $E_{i}(w,s)$ and $\boldsymbol{\Phi}(s)$ are holomorphic on the line $\Re(s)=1$, and $\boldsymbol{\Phi}(s)$ is unitary there \cite[Corollary 1.65]{SarCoh}. Furthermore, $\boldsymbol{\Phi}(s)$ is symmetric \cite[Corollary 1.64]{SarCoh}, and the entries of $\boldsymbol{\Phi}(s)$ are bounded \cite[\S 6.3, Lemma 3.5]{EGM} in the region $\{ s \in \mathbb{C} : \Re(s) \geq 1 \text{ and } |{\Im(s)}| \geq \varepsilon \}$, for any $\varepsilon > 0$. 

Each $E_{i}(w,s)$ has a simple pole at $s=2$ \cite[Theorem 1.58]{SarCoh}, whose residue is the constant function of $w$ given by
\begin{equation} \label{constantres}
    \Res_{s=2} E_{i}(w,s)=\frac{\mathcal{V} }{\vol( \Gamma \backslash \mathbb{H}^3)},
\end{equation}
where $\mathcal{V}$ was defined in \eqref{voldef} and the volume in the denominator is with respect to the $\operatorname{SL}_2(\mathbb{C})$-invariant measure $\frac{dx\, dy\, dv}{v^3}$. Thus also
\begin{equation}\label{rhores}
    \Res_{s=2} \rho_{ij}(0, v, s) = \Res_{s=2} \frac{1}{\mathcal{V}}\int_{\C / \Lambda} E_i(\sigma_j (w+z), s) \, dx \, dy= \frac{\mathcal{V} }{\vol( \Gamma \backslash \mathbb{H}^3)}.
\end{equation}

Combining \eqref{rhozero} and \eqref{rhononzero} with \eqref{Kintdef} we have  
\begin{equation} \label{constterm}
    \rho_{i j}(0,v,s)=\mathbf{1}_{i j} \cdot v^s+ \frac{\pi}{s-1} \varphi_{i j}(0,s) \cdot v^{2-s}
\end{equation}
and
\begin{equation} \label{rhodeltatriv}
    \rho_{i j}(\delta,v,s)=\frac{(2 \pi)^s}{\Gamma(s)} N(\delta)^{(s-1)/2} v K_{s-1}(4 \pi |{\delta}|v)  \varphi_{i j}(\delta,s)
\end{equation}
for $0 \neq \delta \in \Lambda^*$, where 
\begin{align}\label{varphidef}
    \varphi_{i j}(\delta,s) &= \sum_{\substack{ \gamma=\pmatrix {*} {*} {c} {d} \in \Gamma_{\infty} \backslash \sigma_{i}^{-1} \Gamma \sigma_{j} / \Gamma_{\infty} \\ c \neq 0 } } \frac{\check{e}(\frac{\delta d}{c}) }{N(c)^s}.
\end{align}

In the present case of trivial $\psi$, recall that we denote
\begin{equation*}
    \boldsymbol{E}(w,s):=\big(E_{i}(w,s) \big)^{\top},
\end{equation*}
so that \eqref{Efunceq} and \eqref{phifunc} become
\begin{equation} \label{eisstarfunc}
    \boldsymbol{E}(w,s)=\boldsymbol{\Phi}(s)\cdot  \boldsymbol{E}(w,2-s) \qquad \text{and} \qquad \boldsymbol{\Phi}(s) \cdot \boldsymbol{\Phi}(2-s)=I.
\end{equation}
Let
\begin{equation*}
    \boldsymbol{\rho}(0,v,s):=\big(\rho_{i j}(0,v,s) \big).
\end{equation*}
Considering constant terms in \eqref{eisstarfunc} gives 
\begin{equation} \label{rhofunc}
    \boldsymbol{\rho}(0,v,s)=\boldsymbol{\Phi}(s) \cdot \boldsymbol{\rho}(0,v,2-s).
\end{equation}

\subsubsection{Uniform bounds} 

We will require certain uniform bounds for Eisenstein series, which are well-known to experts but whose precise forms do not seem to be available in the literature. These bounds are developed in detail in the present section. 

Let
\begin{equation}\label{eq:PSL2_fund_domain}
    \mathcal{F}' := \Big\{w = (z, v) \in \mathbb{H}^3 :  |{\Re(z)}| \leq \frac{1}{2}, \ 0 \leq \Im(z) \leq \frac{1}{2}, \text{ and } |z|^2 + v^2 \geq 1 \Big\},
\end{equation}
which by \cite[\S 7.3, Proposition 3.9]{EGM} is a fundamental domain for $\mathrm{PSL}_2(\mathcal{O}_K) \backslash \mathbb{H}^3$.

Let $\{\kappa_1:=\infty, \kappa_2, \dots,\kappa_h\}$ denote the set of all cusps of $\Gamma$,
and let $\sigma_i := \sigma_{\kappa_i} \in \SL_2(\mathcal{O}_K)$ be a choice of scaling matrices. To be consistent with the previous notation, we assume without loss of generality that the cusps are enumerated so that $\kappa_j = \eta_j$ for each integer $1 \leq j \leq m$.
Since $\Gamma \subseteq \SL_2(\mathcal{O}_K)$ has finite index, we can choose a finite set $\mathcal{T} \subset \SL_2(\mathcal{O}_K)_\infty$ of coset representatives for $\Gamma_\infty \backslash \SL_2(\mathcal{O}_K)_\infty / \{\pm I \}$. Moreover, since $\Gamma \subseteq \SL_2(\mathcal{O}_K)$ is normal and $\Gamma'_\infty = \Gamma_\infty$, it can be checked that
\begin{equation}\label{eq:fund_domain_def}
    \mathcal{F} := \bigcup_{1\leq j \leq h} \bigcup_{\tau \in \mathcal{T}} \sigma_j \tau \cdot \mathcal{F}'
\end{equation}
is a fundamental domain for $\Gamma \backslash \mathbb{H}^3$.

For each integer $1\leq i\leq h$ and any $A \geq 1$, let $E_i^A(w, s)$ denote the truncated Eisenstein series at height $A$. It is by definition a $\Gamma$-invariant function of $w \in \mathbb{H}^3$. 
Using \eqref{eq:fund_domain_def}, each $w \in \mathcal{F}$ can be written as $w = \sigma_j \tau w'$ for some $\sigma_j$ and $\tau$ as above, and some $w' \in \mathcal{F}^{\prime}$.
It is first defined on $\mathcal{F}$ by
\begin{equation} \label{eq:trunc_eis_series}
    E_i^A(\sigma_j \tau w', s) := 
    \begin{cases}
         E_i(\sigma_j \tau w', s) - \rho_{ij}(0, v(w'), s) & \quad \text{if } v(w') \geq A, \\
         E_i(\sigma_j \tau w', s) & \quad \text{if } v(w') < A,
    \end{cases}
\end{equation}
and then extended by $\Gamma$-invariance to all of $\mathbb{H}^3$.

Before proceeding, we need a straightforward lemma about the scattering matrix.

\begin{lemma} \label{omegalem}
    The entries of $\boldsymbol{\Phi}(s)$ satisfy $|\phi_{ij}(s)| \ll_{\Gamma, \varepsilon} 1$ for $s \in \Omega_{\varepsilon}$, where
    \begin{equation} \label{Omegadef}
        \Omega_{\varepsilon} := \{s \in \mathbb{C}: \text{$1+\varepsilon \leq \Re(s) \leq 10$, and $|s-s_0| \geq \varepsilon$ for each pole $s_0$ of $\Phi(\cdot)$} \}. 
    \end{equation}
\end{lemma}
\begin{proof}
    We have the claimed bound in the set $\{ s \in \mathbb{C} : \Re(s) \geq 1 \text{ and } |{\Im(s)}|  \geq \varepsilon \}$, as described in \cref{standard}. Furthermore, boundedness on the set $\Omega_{\varepsilon} \cap \{s \in \mathbb{C}: |{\Im(s)}| \leq \varepsilon\}$ follows by compactness.
\end{proof}

We have good control of the $L^2$-norm of $E_i^A(w, s)$ due to the bound below.

\begin{lemma}[Maass--Selberg bound]\label{lemma:Maass_Selberg}
    Let $\varepsilon > 0$ and $s \in \Omega_{\varepsilon}$, where $\Omega_{\varepsilon}$ is given in \eqref{Omegadef}. Then for any integer $1\leq i \leq h$ and $A \geq 1$ we have
    \begin{equation*}
        \int_\mathcal{F} |{E_i^A(w, s)}|^2 \, d\mu(w) \ll_{\Gamma, \varepsilon} A^{2(\Re(s)-1)}.
    \end{equation*}
\end{lemma}

\begin{proof}
    This follows directly from the Maass--Selberg relation \cite[pg.\ 41]{SarCoh}. Denoting the entries of $\mathbf{\Phi}(s)$ by $\phi_{ij}(s) := K(0, s) \varphi_{i j}(0, s)$ and $s = \sigma + it$ with $\sigma, t \in \R$, it gives
    \begin{equation*}
        \int_\mathcal{F} |{E_i^A(w, s)}|^2 \, d\mu(w) = \frac{A^{2(\sigma - 1)} - A^{-2(\sigma-1)} \sum_{k=1}^h |{\phi_{ik}(s)}|^2}{2(\sigma-1)} + \frac{\overline{\phi_{ii}(s)} A^{2it} - \phi_{ii}(s) A^{-2it} }{2it}.
    \end{equation*}
    Since $|\phi_{ij}(s)| \ll_{\Gamma, \varepsilon} 1$ for $s \in \Omega_{\varepsilon}$ by Lemma \ref{omegalem}, if we assume that $|t| > \varepsilon$, then we directly obtain the desired result. 
    
    On the other hand, if $s \in \Omega_\varepsilon$ with $|{t}| \leq \varepsilon$, then the first term in the Maass--Selberg relation is still $\ll_{\Gamma, \varepsilon} A^{2(\sigma-1)}$, but some care must be taken with the second term (in particular, it is understood as a continuous extension in $t$ by taking the limit when $t=0$). Taking the Taylor expansion at $s=\sigma$ and using the relation $\overline{\phi_{ii}(s)} = \phi_{ii}(\overline{s})$, we deduce that
    \begin{equation*}
   \lim_{t \rightarrow 0}  \frac{\overline{\phi_{ii}(s)} A^{2it} - \phi_{ii}(s) A^{-2it} }{2it} = 2 \phi_{ii}(\sigma) \log{A} - \phi_{ii}'(\sigma).
    \end{equation*}
   By compactness (and since we stay away from the poles of $\phi_{ii}(s)$), 
   both $\phi_{ii}(s)$ and $\phi^{\prime}_{ii}(s)$ are $\ll_{\Gamma,\varepsilon} 1$
   for $s \in \Omega_\varepsilon$ with $|{t}| \leq \varepsilon$.
   Thus the display above is $\ll_{\Gamma, \varepsilon} 1+\log{A} \ll_{\Gamma, \varepsilon} A^{2(\sigma-1)}$ when $t=0$, since $\sigma \geq 1+\varepsilon$.
  Now assume that $s \in \Omega_{\varepsilon}$ with $0<|t| \leq \varepsilon$. 
  Using
\(\overline{\phi_{ii}(s)}=\phi_{ii}(\overline{s})\) and \(|A^{2it}|=1\), we have
\begin{equation*}
\Big | \frac{\overline{\phi_{ii}(s)}A^{2it}-
\phi_{ii}(s)A^{-2it}}{2it}
\Big| \leq \Big| \frac{\phi_{ii}(\overline{s})-\phi_{ii}(s)}{2it} \Big|+ |\phi_{ii}(s)| \cdot \Big | \frac{A^{2it}-A^{-2it}}{2it} \Big|.
\end{equation*}
The first quotient is bounded uniformly, since
\begin{equation*}
\frac{\phi_{ii}(\sigma-it)-\phi_{ii}(\sigma+it)}{2it}
=-\frac{1}{2t}\int_{-t}^{t} \phi_{ii}^{\prime}(\sigma+iu)\,du
\ll_{\Gamma,\varepsilon} 1.
\end{equation*}
For the second quotient, 
\begin{equation*}
\Big| \frac{A^{2it}-A^{-2it}}{2it} \Big| =\Big| \frac{\sin(2t \log A)}{t} \Big | \leq 2 \log A.
\end{equation*}
The result follows.
\end{proof}

In the sequel we will need a uniform bound for $E_i(w, s)$. 

\begin{lemma}[Uniform pointwise bound]\label{lemma:Eis_pointwise_bound}
     Let $\varepsilon > 0$ and $s \in \Omega_{\varepsilon}$, where $\Omega_{\varepsilon}$ is given in \eqref{Omegadef}. Then for any integer $1\leq i \leq h$ and $w \in \mathbb{H}^3$, we have
    \begin{equation*}
        E_i(w,s) \ll_{\Gamma, \varepsilon} (v_\mathrm{max}(w))^{\Re(s)} |{s}|^{1/2} \Big(1  + \frac{|{s}|}{v_\mathrm{max}(w)} \Big),
    \end{equation*}
where 
\begin{equation*}
v_\mathrm{max}(w) := \max_{1\leq j\leq h} \max_{\gamma \in \Gamma} v(\sigma_j^{-1} \gamma w).
\end{equation*}
\end{lemma}

\begin{proof}
    Let $w_\ell = (z_\ell, v_\ell) \in \mathbb{H}^3$ for $\ell \in \{1, 2\}$, and let $\Psi: \R_{\geq 0} \to \R_{\geq 0}$ be a smooth function with compact support. Consider the point-pair invariant $k(w_1, w_2) := \Psi\big(\frac{|{z_1-z_2}|^2+|{v_1-v_2}|^2}{v_1 v_2}\big)$, and the corresponding kernel $K_\Gamma(w_1, w_2) := \sum_{\gamma \in \Gamma} k(w_1, \gamma w_2)$. Let $h$ denote the Selberg--Harish-Chandra transform of $\Psi$, so that\footnote{See for instance \cite[\S 3.5, (5.2) and (5.3)]{EGM}. Note that we follow the conventions of \cite[(1.18)]{SarCoh}, which are equivalent to \cite[\S 1.1, Propositions 1.6 and 1.9]{EGM} after performing $s \mapsto s+1$, $\Psi(x) \mapsto \Psi(1+ \frac{x}{2})$, and $H(x) \mapsto H(x(2-x))$.}
    \begin{equation}\label{eq:Selberg_kernel}
        \int_\mathcal{F} K_\Gamma(w, u) E_i(u, s) \, d\mu(u) = H(s) E_i(w, s).
    \end{equation} 

    By \eqref{eq:fund_domain_def}, we may write $u = \sigma_j \tau u' \in \mathcal{F}$ for $u' \in \mathcal{F}'$. In particular,
    \begin{equation*}
        v(u') = \max_{\sigma \in \SL_2(\mathcal{O}_K)} v(\sigma u) = \max_{1\leq r \leq h} v(\sigma_r^{-1} u).
    \end{equation*} 
    Since $\Psi$ is non-negative and $\Gamma \subseteq \SL_2(\mathcal{O}_K)$, it follows that
    \begin{equation}\label{eq:kernel_trivial_bd}
        K_\Gamma(w, u) \leq K_{\SL_2(\mathcal{O}_K)}(w, u) = K_{\SL_2(\mathcal{O}_K)}(w, \sigma_j \tau u') = K_{\SL_2(\mathcal{O}_K)}(w, u').
    \end{equation}
    Assume that $\Psi$ is supported on $[0, \delta]$ for $\delta>0$ sufficiently small. Then $K_{\SL_2(\mathcal{O}_K)}(w, u') \neq 0$ implies by a direct computation that 
    $v(u') < 2 v_{\mathrm{max}}(w)$. 
    
    Taking $A(w) := 2 v_{\mathrm{max}}(w)$, it then follows from the observations above and Cauchy--Schwarz that 
    \begin{align*}
        |H(s) E_i(w, s)| &= \Big | \int_\mathcal{F} K_\Gamma(w, u) E^{A(w)}_i(u, s) \, d\mu(u) \Big | \\
        & \leq \Big(\int_\mathcal{F} |{K_\Gamma(w, u)}|^2 \, d\mu(u)\Big)^{1/2} \Big(\int_\mathcal{F} |{E^{A(w)}_i(u, s)}|^2 \, d\mu(u)\Big)^{1/2}.
    \end{align*}
    The second term is $\ll_{\Gamma,\varepsilon} A(w)^{\Re(s)-1}$ for $s \in \Omega_{\varepsilon}$ by \cref{lemma:Maass_Selberg}. We now consider the first term. Using \eqref{eq:kernel_trivial_bd}, we deduce that
    \begin{align*}
        \int_\mathcal{F} |{K_\Gamma(w, u)}|^2 \, d\mu(u) &\leq \sum_{j=1}^h \sum_{\tau \in \mathcal{T}} \int_\mathcal{F'} |{K_{\SL_2(\mathcal{O}_K)}(w, u')}|^2 \, d\mu(\sigma_j \tau u') \\
        &\ll_\Gamma \int_\mathcal{F'} |{K_{\SL_2(\mathcal{O}_K)}(w, u')}|^2 \, d\mu(u').
    \end{align*} 

    We pick $\delta = c|s|^{-2}$, for a sufficiently small absolute constant $c>0$. Choose a smooth $\Psi \geq 0$ supported on $[0, \delta]$ such that $\sup_x |{\Psi(x)}| \ll \delta^{-3/2}$ and $\int_{\mathbb{H}^3} k(w, u) \, d\mu(u) = 1$, as in \cite[\S 6.4, (4.37)]{EGM}. If $c>0$ is sufficiently small, then by \cite[Lemma 7.5]{SarCoh_chapt7} (or the proof of \cite[\S 3.5, Lemma 5.6]{EGM}) we have $|H(s)| \geq \frac{1}{2}$. 
    
    Furthermore, \cite[\S 6.4, Lemma 4.9]{EGM} gives
    \begin{equation*}
        \int_\mathcal{F'} |{K_{\SL_2(\mathcal{O}_K)}(w, u')}|^2 \, d\mu(u') \ll v_\mathrm{max}(w)^2 |s| + |s|^3.
    \end{equation*}
    Putting everything together gives the claimed result.
\end{proof}

Finally, we need uniform decay for truncated Eisenstein series at the cusps.

\begin{lemma}[Decay of truncated Eisenstein series]\label{lemma:truncated_Eis_decay}
  Let $\varepsilon > 0$ and $s \in \Omega_{\varepsilon}$, where $\Omega_{\varepsilon}$ is given in \eqref{Omegadef}. There exist parameters $m_{\Gamma} > 0$ and $C_{\Gamma}>0$, depending only on $\Gamma$, such that for any $1\leq i, j \leq h$, $A \geq 1$, and $v \geq \widetilde{v} := \max(A, C_{\Gamma} |s|^2)$, we have
    \begin{equation*}
        \int_{\C / \Lambda} |{E_i^A(\sigma_j (z, v), s)}|^2 \, dx \, dy \ll_{\Gamma, \varepsilon} e^{-m_\Gamma (v - \widetilde{v})} \widetilde{v}^{2 \Re(s)} |{s}|.
    \end{equation*}
\end{lemma}

\begin{proof}
    By \cite[\S10.40(iii)]{DLMF}, if $\nu \in \C$ and $y \geq 1 + |{\nu}|^2$ then 
    \begin{equation}\label{eq:Bessel_bound}
        K_\nu(y) = \Big(\frac{\pi}{2y}\Big)^{1/2} e^{-y} \Big( 1 + O\Big(\frac{1 + |{\nu}|^2}{y}\Big)\Big),
    \end{equation}
    where the implied constant is absolute. In particular, if $y \geq D(1 + |{\nu}|^2)$ with $D>1$ sufficiently large, then for any $\Delta \geq 0$ we obtain
    \begin{equation}\label{eq:Bessel_comparison}
        |K_\nu(y+\Delta)| \leq 2 e^{-\frac{\Delta}{2}} |K_{\nu}(y)|.
    \end{equation}

    From the Fourier expansion \eqref{eisfourier}, \eqref{rhononzero}, \eqref{rhodeltatriv}, and the definition in \eqref{eq:trunc_eis_series}, we deduce that
    \begin{equation*}
        \frac{1}{\mathcal{V}} \int_{\C / \Lambda} |{E_i^A(\sigma_j (z, v), s)}|^2 \, dx \, dy  = \frac{(2 \pi)^{2\Re(s)} v^2}{|\Gamma(s)|^2}  \sum_{\substack{0 \neq \delta \in \Lambda^* }} N(\delta)^{\Re(s)-1} |{\varphi_{i j}(\delta,s)}|^2  \cdot |{K_{s-1}(4 \pi |{\delta}|v)}|^2,
    \end{equation*}
     since $v \geq A$. Let $\displaystyle c_\Gamma := \min_{0 \neq \delta \in \Lambda^*} 4\pi|{\delta}| > 0$ and $C_{\Gamma}:=D c^{-1}_{\Gamma}$.
    By hypothesis $v \geq \widetilde{v} > C_{\Gamma}(1 + |{s-1}|^2)$. Hence \eqref{eq:Bessel_comparison} implies
    \begin{equation*}
        |{K_{s-1}(4 \pi |{\delta}|v)}|^2 \leq 4 e^{-4\pi |{\delta }|(v - \widetilde{v})} |{K_{s-1}(4 \pi |{\delta}|\widetilde{v})}|^2.
    \end{equation*} 
    Therefore, 
    \begin{equation}\label{eq:exp_gain}
        \int_{\C / \Lambda} |{E_i^A(\sigma_j (z, v), s)}|^2 \, dx \, dy \leq 4 e^{-c_\Gamma (v-\widetilde{v})} \Big( \frac{v}{\widetilde{v}} \Big)^2 \int_{\C / \Lambda} |{E_i^A(\sigma_j (z, \widetilde{v}), s)}|^2 \, dx \, dy.
    \end{equation}

    Finally, set as before $\phi_{ij}(s) := K(0, s) \varphi_{ij}(0, s)$, which are the entries of $\mathbf{\Phi}(s)$. We have $|{\phi_{ij}(s)}| \ll_{\Gamma, \varepsilon} 1$ 
    for $s \in \Omega_{\varepsilon}$ by Lemma \ref{omegalem}, and when combined with \cref{lemma:Eis_pointwise_bound} gives
    \begin{equation} \label{eq:EAbound}
        |{E_i^A(\sigma_j (z, \widetilde{v}), s)}| \leq |{E_i(\sigma_j (z, \widetilde{v}), s)}| + |{\mathbf{1}_{i j} \cdot \widetilde{v}^s+ \phi_{i j}(s) \cdot \widetilde{v}^{2-s}}|\ll_{\Gamma, \varepsilon} \widetilde{v}^{\Re(s)} |{s}|^{1/2},
    \end{equation}
    since $v_{\text{max}}(\sigma_j(z,\widetilde{v})) \asymp_{\Gamma} \widetilde{v}$ uniformly for $z \in \mathbb{C} / \Lambda$.
    The desired result follows after inserting the bound \eqref{eq:EAbound} into \eqref{eq:exp_gain}, and absorbing the $(v/\widetilde{v})^2$ factor into the exponential,
    resulting in a smaller constant $0<m_{\Gamma}<c_{\Gamma}$.
\end{proof}


\section{Rankin--Selberg regularization} \label{zagregsec}

We generalize Zagier's regularization method \cite{Zag} for automorphic functions on $\operatorname{SL}_2(\mathbb{Z}) \backslash \mathbb{H}^2$ not necessarily of rapid decay at the cusps to the setting of $\Gamma \backslash \mathbb{H}^3$. The notation from the previous section will be used throughout. 

\begin{prop}[Zagier regularization] \label{zagreg}
    Let $\Gamma \subseteq \operatorname{SL}_2(\mathcal{O}_{K})$ be a normal, co-finite but not co-compact subgroup such that $\Gamma_\infty = \Gamma'_\infty$. Let $\{\kappa_1 = \infty,\kappa_2,\ldots,\kappa_h\}$ be representatives for all cusps of $\Gamma$. Denote $\sigma_j := \sigma_{\kappa_j}$ and $\Gamma_j := \Gamma_{\kappa_j}$, with $\sigma_1 = I$. Let $F:\mathbb{H}^3 \to \mathbb{C}$ be a continuous $\Gamma$-invariant function with Fourier expansion at each cusp $\kappa_j$ given by 
    \begin{equation} \label{fourierexp}
        F(\sigma_{j}w) = \sum_{\delta \in \Lambda^{*}} a_{j}(\delta,v) \check{e}(\delta z), \qquad w=(z,v) \in \mathbb{H}^3.
    \end{equation}
    Suppose that for all $N \in \mathbb{N}$ we have 
    \begin{equation} \label{cuspasymp}
        F(\sigma_{j} w)=\varphi_{j}(v)+O_{F,N}(v^{-N}) \quad \text{as} \quad v \to \infty,
    \end{equation}
    where there are $c_{jk},\alpha_{jk} \in \mathbb{C}$ and $n_{jk} \in \mathbb{Z}_{\geq 0}$ such that
    \begin{equation} \label{varphiexp}
        \varphi_{j}(v)=\sum_{k=1}^{r} \frac{c_{jk}}{n_{jk}!} \cdot v^{\alpha_{jk}} \log^{n_{jk}}(v).
    \end{equation}
    Define the Rankin--Selberg transform of $F$ at the cusp $\kappa_j$ by
    \begin{align} \label{RStransform1}
    R_{j}(F,s) := \int_{0}^{\infty} \big(a_{j}(0,v)-\varphi_{j}(v) \big) v^{s-3}\, dv, 
    \end{align}
    where the integral converges absolutely for $\Re(s)$ sufficiently large, and let 
    \begin{equation*}
    \boldsymbol{R}(F,s):= \Big ( R_{j}(F,s) \Big)^{\top}_{1 \leq j \leq h}.
    \end{equation*}
    Then the entries of $\boldsymbol{R}(F,s)$ 
    can be meromorphically continued 
    to all $s \in \mathbb{C}$, and satisfy the vector functional equation 
    \begin{equation} \label{Rstarfunc}
        \boldsymbol{R}(F,s)=\boldsymbol{\Phi}(s) \cdot \boldsymbol{R}(F,2-s).
    \end{equation}
    
    The poles of the entries of $\boldsymbol{R}(F,s)$ can only occur at the points
    \begin{equation} \label{possiblepoles}
        s_0 \in  \bigcup_{\substack{1 \leq j  \leq h \\ 1 \leq k \leq r }}  \big\{\alpha_{jk}, 2-\alpha_{jk} \big\} \cup \big\{\emph{poles of }\boldsymbol{\Phi}(s) \big\} .
    \end{equation}
    More precisely,
    \begin{equation} \label{Rstarexp}
        \boldsymbol{R}(F,s)= - \boldsymbol{g}_1(s) - \boldsymbol{\Phi}(s) \cdot \boldsymbol{g}_1(2-s)  + \boldsymbol{V}(s),
    \end{equation}
    where
    \begin{equation} \label{hdef}
        \boldsymbol{g}_1(s) =-\Big({ \sum_{k=1}^{r} \frac{c_{jk}}{(2-\alpha_{jk}-s)^{n_{jk}+1}} }\Big )^{\top}_{1 \leq j \leq h}
    \end{equation}
    and $\boldsymbol{V}(s)$ is a vector of meromorphic functions on $\mathbb{C}$ whose poles can only occur at poles of $\boldsymbol{\Phi}(s)$.
    
    For the remaining claims, we assume that $\alpha_{jk} \not \in \{0,2\}$ for all integers $1 \leq j \leq h$ and $1 \leq k \leq r$. Then
   at $s=2$ each entry of $\boldsymbol{R}(F,s)$ has at most a simple pole. Let
     \begin{equation*} 
        \alpha:=\max_{j,k} \Re(\alpha_{jk}).
    \end{equation*}
    If $\alpha<2$, then for the constant vector $\boldsymbol{1} := (1)^{\top}_{1\leq j\leq h}$ we have
    \begin{equation} \label{Rstares2}
        \Res_{s=2} \boldsymbol{R}(F,s)=\frac{1}{\vol(\Gamma \backslash \mathbb{H}^3)} \cdot \int_{\Gamma \backslash \mathbb{H}^3} F(w)\, d \mu(w) \cdot \boldsymbol{1}.
    \end{equation}
    If $\alpha<1$, then for $\alpha<\Re(s)<2-\alpha$ with $s$ not a pole of $\boldsymbol{\Phi}(\cdot)$, we have that
    \begin{equation} \label{RSidentity}
        \boldsymbol{R}(F,s)= \frac{1}{\mathcal{V}}\int_{\Gamma \backslash \mathbb{H}^3} F(w) \boldsymbol{E}(w,s) \, d\mu(w).
    \end{equation}
\end{prop}

\begin{remark} \label{absconvrem}
    The proof of \cref{zagreg} will show that the integral expression for $R_{i}(F,s)$ in \eqref{RStransform1} converges absolutely for $\Re(s) > 2 + \max_{j,k} \left|\Re(\alpha_{jk})\right|$
    for each $i=1,\ldots,h$.
\end{remark}

\begin{remark}
The continuity assumption on $F$ will be used only in a soft way. It ensures that $F$ is bounded on compact subsets of $\Gamma\backslash \mathbb H^3$.
Together with $\Gamma$-invariance, it also ensures
that the cusp averages
\begin{equation*}
a_j(0,v)=\frac{1}{\mathcal{V}}\int_{\mathbb C/\Lambda} F(\sigma_j(z,v)) \, dx \, dy
\end{equation*}
are well-defined. These features will be used without further comment.
\end{remark}

For generalizations of the ideas of Zagier \cite{Zag} in other directions, see the works of Dutta Gupta \cite{DG} and Mizuno \cite{Miz}.
 
\begin{proof}[Proof of \cref{zagreg}]

Recall that each cusp $\kappa_j$ has scaling matrix $\sigma_{j} \in \operatorname{SL}_2(\mathcal{O}_{K})$ and stabilizer subgroup $\Gamma_{j}$, for each integer $1 \leq j \leq h$. By convention $\kappa_1=\infty$ and $\sigma_{1}=I$.


\medskip
\noindent\textbf{Step 1: the cusp at $\infty$.} Let $\mathcal{F}'$ denote the usual fundamental domain for $\mathrm{PSL}_2(\mathcal{O}_K)\backslash \mathbb{H}^3$, given in \eqref{eq:PSL2_fund_domain}, and let $\mathcal{F}$ denote the fundamental domain for $\Gamma \backslash \mathbb{H}^3$ given in \eqref{eq:fund_domain_def}. Set
\begin{equation*}
    \mathcal{D} = \bigcup_{\tau \in \mathcal{T}} \tau \cdot \Big\{ z \in \C : 0 \leq |{\Re(z)}|, \, \Im(z) \leq \frac{1}{2} \Big\}, 
\end{equation*}
which is a fundamental domain for $\mathbb{C}/\Lambda \simeq \Gamma_\infty\backslash\C$.  Thus $\mathcal{D} \times (0, \infty)$ is a fundamental domain for $\Gamma_\infty\backslash\mathbb{H}^3$. 

Denote the cuspidal region
\begin{equation} \label{Sinftydef}
    \mathcal{C}_{\infty}(T):= \big\{w \in \mathbb{H}^3: z \in \mathcal{D} \text{ and } v>T \big\},
\end{equation}
and more generally for each $1 \leq j \leq h$ denote
\begin{equation} \label{Skappajdef}
    \mathcal{C}_{\kappa_j}(T):=\sigma_{j} \cdot  \mathcal{C}_{\infty}(T).
\end{equation}
From now on assume that $T > 1$, so that the $\mathcal{C}_{\kappa_j}(T)$ are pairwise disjoint and contained\footnote{Since $\displaystyle\mathcal{C}_{\kappa_j}(T) \subset \bigcup_{\tau\in \mathcal{T}} \sigma_j \tau \cdot \mathcal{F}'$ for $T > 1$.} in $\mathcal{F}$. Consider the truncated domain 
\begin{equation} \label{DTdef}
    \mathcal{F}(T) := \mathcal{F}-\bigsqcup_{j=1}^{h} \mathcal{C}_{\kappa_j}(T).
\end{equation}
Then $\mathcal{F}(T)$ is a fundamental domain for the action of $\Gamma$ on $\mathbb{H}^3_T$, where
\begin{align*}
    \mathbb{H}_{T}^3 := \bigcup_{\gamma \in \Gamma} \gamma \cdot \mathcal{F}(T) = \Big\{w \in \mathbb{H}^3: \max_{\substack{1 \leq j \leq h \\ \gamma \in \Gamma}} v(\sigma_{j}^{-1} \gamma w ) \leq T \Big\}.
\end{align*}
Note that
\begin{equation*}
    \mathbb{H}_T^3 = \big\{w \in \mathbb{H}^3: v \leq T \big\} - \bigsqcup_{j=1}^h \bigsqcup_{\substack{a, c \in \mathcal{O}_{K},\, c \neq 0 \\ \pmatrix {a} {*} {c} {*} \in \Gamma \sigma_{j} / \Gamma_{\infty}  }} B_{a/c}, 
\end{equation*}
where $B_{a/c}$ is an open ball in $\mathbb{H}^3$ tangent to $\frac{a}{c} \in K$ with radius $(2 |c|^2 T)^{-1}$. Thus 
\begin{equation} \label{truncateident}
    \Gamma_{\infty} \backslash \mathbb{H}_T^3 = \big\{w \in \mathbb{H}^3: z \in \mathcal{D} \text{ and } v \leq T \big\}
    - \bigsqcup_{j=1}^h \bigsqcup_{\substack{a,c \in \mathcal{O}_{K}, \, c \neq 0 \\ \pmatrix {a} {*} {c} {*} \in \Gamma_{\infty} \backslash \Gamma \sigma_{j} /\Gamma_{\infty}  }} B_{a/c}. 
\end{equation}

Recall that for any integer $1\leq j\leq h$ and any $C>\alpha:=\max_{j,k}  \Re(\alpha_{jk})$, we have
\begin{equation} \label{inftyasymp}
    F(\sigma_{j} w) = O_F\big(v^{C}\big) \quad \text{as} \quad v \to \infty.
\end{equation} 
Note also that for any $w \in \mathbb{H}^3$,
\begin{equation} \label{maxstate}
    \max_{\delta \in \SL_2(\mathcal{O}_K)} v(\delta w)=\max_{\substack{\pmatrix {*} {*} {c} {d} \in \SL_2(\mathcal{O}_K)}} \frac{v}{|cz+d|^2+|c|^2 v^2} \leq \max \Big (v,\frac{1}{v}  \Big ).
\end{equation}
Indeed, \eqref{maxstate} follows since $c=0$ implies that $|c z+d|^2+|c|^2 v^2 = |d|^2 \geq 1$, while $c \neq 0$ implies that $|c z+d|^2+|c|^2 v^2 \geq |c|^2 v^2 \geq v^2$. Furthermore, using \eqref{eq:fund_domain_def} we observe that for each $w \in \mathbb{H}^3$ there exist $\gamma\in \Gamma$, an integer $1\leq j \leq h$, and $\tau \in \mathcal{T} \subset \SL_2(\mathcal{O}_K)_\infty$ such that $w \in \gamma^{-1} \sigma_j \tau \cdot \mathcal{F}'$, hence $v(\sigma_j^{-1} \gamma w) \gg 1$. Since $F$ is $\Gamma$-invariant, we deduce from \eqref{inftyasymp} and \eqref{maxstate} that 
\begin{equation} \label{zeroasymp}
    F(w) = F(\sigma_j (\sigma_j^{-1}\gamma w)) =  O_F\big(v^{-\max(0, C)}\big) \quad \text{as} \quad v \to 0. 
\end{equation}

It follows from \eqref{zeroasymp} that the integral 
\begin{equation*}
    \int_{0}^{T} \int_{\mathcal{D}} F(w) v^{s} \, d\mu(w)
\end{equation*}
converges absolutely and locally uniformly in the half-plane $\Re(s)>2+\max(0,C)$. Thus unfolding gives
\begin{equation} \label{RSidentT}
    \int_{\Gamma \backslash \mathbb{H}_T^3} F(w) E_{1}(w,s)\, d \mu(w)=\int_{\Gamma_{\infty} \backslash \mathbb{H}_T^3 } F(w) v^s\, d \mu(w)
\end{equation}
for $\Re(s)>2+\max(0,C)$. Then applying \eqref{truncateident} yields the identity
\begin{equation}\label{intid}
\begin{aligned}
    \int_{\mathcal{F}(T) } F(w) E_{1}(w,s) \, d \mu(w) =\ & \int_{0}^{T} \int_{\mathcal{D}} F(w) v^{s} \, d\mu(w) \\
    &- \sum_{j=1}^{h} \sumtwo_{\substack{a,c \in \mathcal{O}_{K}, \, c \neq 0 \\  \pmatrix {a} {*} {c} {*} \in \Gamma_{\infty} \backslash \Gamma \sigma_{j} / \Gamma_{\infty} }} 
    \int_{B_{a/c}} F(w) v^{s} \, d\mu(w)
\end{aligned}
\end{equation}
for $\Re(s)>2+\max(0,C)$. The integral over $B_{a/c}$ in \eqref{intid} can be transformed by fixing an element $\gamma_0= \pmatrix {a} {b} {c} {d} \in \Gamma \sigma_{j}$, so that $B_{a/c} = \gamma_0 \cdot \big\{w \in \mathbb{H}^3: v>T \big\}$. By the $\Gamma$-invariance of $F$, normality of $\Gamma \subseteq \SL_2(\mathcal{O}_K)$, and $\SL_2(K)$-invariance of $d \mu$ we deduce that 
\begin{align} \label{sphereidentity}
    \int_{B_{a/c}} F(w) v^{s} \, d\mu(w) &=\int_{T}^{\infty} \int_{\mathbb{C}} F(\sigma_{j} w) v(\gamma_0 w)^s \, d \mu(w) \nonumber \\
    &=\int_{\mathcal{C}_{\infty}(T)} F(\sigma_{j} w) \sum_{u \in \Gamma_{\infty} } v(\gamma_0 u w)^s \, d \mu(w) \nonumber \\
    &=\int_{\mathcal{C}_{\infty}(T)} F(\sigma_{j} w) \sum_{\gamma=\pmatrix {a} {*} {c} {*} \in \Gamma \sigma_{j} } 
    v(\gamma w)^s \, d \mu(w).
\end{align}
We substitute \eqref{sphereidentity} into \eqref{intid}, and then interchange the order of summation and integration by absolute convergence. Then the double sum over $a$ and $c$ in \eqref{intid} is equal to 
\begin{align} \label{intintermed2}
    &\int_{\mathcal{C}_{\infty}(T)} F(\sigma_{j} w) \sum_{\substack{\gamma \in \Gamma_{\infty} \backslash \Gamma \\ \gamma \kappa_{j} \neq \infty } } v(\gamma \sigma_{j} w)^{s} \, d \mu(w) \nonumber \\
    = \ & \int_{\mathcal{C}_{\infty}(T)} F(\sigma_{j} w) \cdot \big (E_{1}(\sigma_{j}w,s) - \mathbf{1}_{1 j} \cdot v^s \big) \, d\mu(w).
\end{align}
We perform the integral over $x$ and $y$ in the first term on the right side of \eqref{intid}, and then substitute \eqref{intintermed2} into that display for the double sum over $a$ and $c$. We then interchange the sum and integral by absolute convergence to deduce that 
\begin{align*}
    &\int_{\mathcal{F}(T) } F(w) E_{1}(w,s) \, d \mu(w) \\
    = \ & \mathcal{V} \cdot \int_{0}^{T} a_{1}(0,v) v^{s-3} \, dv -\int_{\mathcal{C}_{\infty}(T)} \sum_{j=1}^h F(\sigma_{j} w) \cdot \big(E_{1}(\sigma_{j}w,s) -\mathbf{1}_{1 j} \cdot v^s \big) \,  d \mu(w)
\end{align*}
for $\Re(s)>2+\max(0,C)$. Using the constant term \eqref{constterm} of the Fourier expansion of $E_{1}(\sigma_{j}w,s)$, we obtain 
\begin{align} \label{rearrange2}
    & \mathcal{V} \cdot  \Big(  \int_{0}^T a_{1}(0,v) v^{s-3} \, dv - \sum_{j=1}^{h} \phi_{1 j}(s) \int_{T}^{\infty} a_{j}(0,v) v^{-s-1} \, dv \Big) \\
    = \ & \int_{\mathcal{F}(T)} F(w) E_{1}(w,s) \, d \mu(w)+ \int_{\mathcal{C}_{\infty}(T)} \sum_{j=1}^{h} F(\sigma_{j} w) \cdot \big(E_{1}(\sigma_{j} w,s)-\rho_{1 j}(0,v,s) \big) \, d \mu(w) \nonumber
\end{align}  
for $\Re(s)>2+\max(0,C)$, where $\phi_{ij}(s) := K(0, s) \varphi_{ij}(0, s)$ are entries of $\mathbf{\Phi}(s)$, by \eqref{Phidef}.

We now evaluate the left side of \eqref{rearrange2}. Let $\boldsymbol{g}_T(s) := \big( g^{(j)}_T(s) \big)^{\top}_{1 \leq j \leq h }$, where
\begin{align} \label{gTcompute}
    g^{(j)}_T(s) := \ & \sum_{k=1}^{r} \frac{c_{jk}}{n_{jk}!} \cdot \Big(\frac{\partial}{\partial s}\Big)^{n_{jk}} \Big( \frac{T^{s+\alpha_{jk}-2}}{s+\alpha_{jk}-2} \Big)  \nonumber \\
    = \ & \sum_{k=1}^{r} c_{jk} \sum_{m=0}^{n_{jk}} \frac{(-1)^{n_{jk}-m}}{m!} \cdot \frac{T^{s+\alpha_{jk}-2} \cdot \log^m(T)}{(s+\alpha_{jk}-2)^{n_{jk}-m+1}}.
\end{align}
A direct computation shows that for $\Re(s) > 2 + \max_{j,k} \left|\Re(\alpha_{jk})\right|$ we have that
\begin{align} \label{a1}
    \int_{0}^T a_{1}(0,v) v^{s-3} \, dv &= \int_{0}^{T} \big(a_{1}(0,v)-\varphi_{1}(v) \big) v^{s-3} \, dv +\sum_{k=1}^{r} \frac{c_{1k}}{n_{1k}!} \int_{0}^{T} v^{s+\alpha_{1k}-3} \cdot \log^{n_{1k}}(v) \, dv \nonumber \\
    &=R_{1}(F,s) - \int_{T}^{\infty} \big(a_{1}(0,v) - \varphi_{1}(v) \big) v^{s-3} \, dv + g^{(1)}_{T}(s),
\end{align}
since by \eqref{zeroasymp} each integral is absolutely convergent in this half-plane (cf.\ \cref{absconvrem} in the case $i=1$). Similarly,
\begin{equation} \label{a2}
    \int_{T}^{\infty} a_{j}(0,v) v^{-s-1} \, dv = \int_T^{\infty} \big(a_{j}(0,v)-\varphi_{j}(v) \big) v^{-s-1} \, dv - g^{(j)}_{T}(2-s)
\end{equation}
for $\Re(s)>\max_{j, k} \left|\Re(\alpha_{jk})\right|$.
Substitute \eqref{a1} and \eqref{a2} into \eqref{rearrange2} to deduce that 
\begin{align} 
    &\mathcal{V} \cdot \Big(R_{1}(F,s)+ g^{(1)}_{T}(s)+\sum_{j=1}^{h} \phi_{1 j}(s) g^{(j)}_{T}(2-s) \Big) = \int_{\mathcal{F}(T)} F(w) E_{1}(w,s) \, d \mu(w)  \nonumber \\
    &+ \int_{\mathcal{C}_{\infty}(T)}   \sum_{j=1}^{h} F(\sigma_{j} w) \cdot \big(E_{1}( \sigma_{j}w,s)-\rho_{1 j}(0,v,s) \big) \, d \mu(w) \label{simplify1} \\
    &+ \mathcal{V} \cdot \Big( \int_{T}^{\infty} \big(a_{1}(0,v)-\varphi_{1}(v) \big) v^{s-3} \, dv + \sum_{j=1}^{h} \phi_{1 j}(s) \int_{T}^{\infty} \big(a_{j}(0,v)-\varphi_{j}(v) \big) v^{-s-1}\, dv \Big), \qquad \label{simplify3a}
\end{align}
for $\Re(s)>2 + \max_{j,k} \left|\Re(\alpha_{jk})\right|$. By \eqref{constterm}, the term in \eqref{simplify3a} is equal to 
\begin{equation*}
    \int_{\mathcal{C}_{\infty}(T)} \sum_{j=1}^h \big( a_{j}(0,v)-\varphi_{j}(v) \big) \cdot \rho_{1 j}(0,v,s) \, d \mu(w).
\end{equation*}
For the term in \eqref{simplify1}, note that the function $\rho_{1 j}(0,v,s)$
is independent of $z=x+iy$, hence its integral against $F(\sigma_{j} w)$
is the same as against $a_{j}(0,v)$. Using these observations we obtain 
\begin{equation}\label{simplify3}
\begin{aligned}
    \mathcal{V} \cdot \Big(R_{1}(F,s)+ g^{(1)}_{T}(s)+\sum_{j=1}^{h} \phi_{1 j}(s) g^{(j)}_{T}(2-s) \Big) =\int_{\mathcal{F}(T)} F(w) E_{1}(w,s) \, d \mu(w)  \\
    +\int_{\mathcal{C}_{\infty}(T)}  \sum_{j=1}^{h}
    \Big(F(\sigma_{j} w) E_{1}( \sigma_{j}w,s)-\varphi_{j}(v) \rho_{1 j}(0,v,s) \Big) \, d \mu(w)
\end{aligned}
\end{equation}
for $\Re(s)>2 + \max_{j,k} \left|\Re(\alpha_{jk})\right|$. Before proceeding, we extend this to all cusps.


\medskip
\noindent\textbf{Step 2: arbitrary cusps.} We will prove a version of the identity \eqref{simplify3} with $R_{1}(F,s)$ replaced by $R_{i}(F,s)$ for each integer $1\leq i \leq h$. For $w \in \mathbb{H}^3$ set
\begin{equation} \label{fidef}
    F_{i}(w):=F(\sigma_{i} w).
\end{equation}
Note that $F_i$ is invariant under $\sigma_{i}^{-1} \Gamma \sigma_{i} = \Gamma$ since $F$ is $\Gamma$-invariant, $\sigma_i \in \SL_2(\mathcal{O}_K)$, and $\Gamma \subseteq \SL_2(\mathcal{O}_K)$ is normal. For each pair of integers $(i,j)$ satisfying $1\leq i, j \leq h$, there is a unique integer $1\leq j^{(i)} \leq h$ such that $\sigma_i \sigma_j \infty \in \Gamma \kappa_{j^{(i)}}$, and these are distinct for fixed $i$ and distinct $j$, since $\gamma_1 \sigma_i \cdot \kappa_{j_1} = \gamma_2 \sigma_i \cdot \kappa_{j_2}$ implies $\kappa_{j_1} = \sigma_i^{-1} \gamma_1^{-1} \gamma_2 \sigma_i \cdot \kappa_{j_2} \in \Gamma \kappa_{j_2}$. Thus there exist $\gamma_j^i \in \Gamma$ and $\beta_j^i \in \SL_2(\mathcal{O}_K)_{\infty}$ such that $\sigma_i \sigma_j = \gamma_j^i \sigma_{j^{(i)}} \beta_{j}^{i}$, and we may take $\beta_1^i = I$. Observe that
\begin{equation}\label{vbeta}
    v(\beta_j^i w) = v
\end{equation}
for all $w \in \mathbb{H}^3$, which by \eqref{cuspasymp} implies for any $N \in \mathbb{N}$ that
\begin{equation} \label{varphident}
    F_{i}(\sigma_{j} w)= F(\sigma_i \sigma_j w) = F(\sigma_{j^{(i)}} \beta_j^i w) = \varphi_{j^{(i)}}(v)+O_{F,N}(v^{-N}) \quad \text{as} \quad v \to \infty.
\end{equation} 
Clearly $1^{(i)} = i$ and
\begin{equation} \label{fiobs}
    F_{i}(\sigma_1 w)=F(\sigma_{i} w) =  \sum_{\delta \in \Lambda^{*}} a_{i}(\delta,v) \check{e}(\delta z),
\end{equation}
so it follows from the definition \eqref{RStransform1} that
\begin{equation} \label{Rid}
    R_{1}(F_{i},s) = R_{i}(F,s),
\end{equation}
and by our previous step this converges absolutely for $\Re(s)>2+\max_{j, k} \left|\Re(\alpha_{jk})\right|$.

Moreover, for $w \in \mathbb{H}^3$ and $\Re(s)>2$ we have
\begin{align} \label{eisensteinid}
    E_{1}(\sigma_{j} w,s)& = \sum_{\gamma \in \Gamma_{\infty} \backslash \Gamma } v(\gamma \sigma_i^{-1} \gamma_j^i \sigma_{j^{(i)}} \beta_j^i w )^{s} = \sum_{\gamma \in \Gamma_{\infty} \backslash \Gamma } v(\sigma_{i}^{-1} \sigma_{i} \gamma \sigma_{i}^{-1} \gamma_j^i \sigma_{j^{(i)}} \beta_j^i w )^{s} \nonumber \\
    &=\sum_{\gamma \in \Gamma_{i} \backslash \Gamma } v(\sigma_{i}^{-1} \gamma \sigma_{j^{(i)}} \beta_j^i w )^{s} =E_{i}(\sigma_{j^{(i)}}\beta_j^i w,s),
\end{align}
and this can be extended to all $s \in \mathbb{C}$ by meromorphic continuation.
It follows from the preceding identity and \eqref{vbeta} that, for all $v>0$,
as meromorphic functions of $s \in \mathbb{C}$, the constant terms $\rho_{ij}(0,v,s)$
of $E_i(\sigma_j w,s)$ satisfy
\begin{equation} \label{rhoident}
\rho_{1 j}(0,v,s)= \rho_{i j^{(i)}}(0,v,s),
\end{equation}
so in particular $\phi_{1j}(s) = \phi_{i j^{(i)}}(s)$. Using \eqref{varphident}, \eqref{Rid}, and \eqref{rhoident}, we deduce from an application of \eqref{simplify3} with $F \mapsto F_{i}$ that  we have
\begin{equation} \label{basicstate2}
\begin{aligned}
    \mathcal{V} \cdot \Big(R_{i}(F,s)+ g^{(i)}_{T}(s)+\sum_{j=1}^{h} \phi_{i j}(s) g^{(j)}_{T}(2-s) \Big) = \int_{\mathcal{F}(T)} F(\sigma_i w) E_{i}(\sigma_i w,s) \, d \mu(w)  \\
    +\int_{\mathcal{C}_{\infty}(T)}  \sum_{j=1}^{h}
    \Big(F(\sigma_{j^{(i)}} \beta_j^i w) E_{i}( \sigma_{j^{(i)}} \beta_j^i w, s)- \varphi_{j^{(i)}}(v) \rho_{i j^{(i)}}(0,v,s) \Big) \, d \mu(w)
\end{aligned}
\end{equation}
for $\Re(s)>2 + \max_{j,k} \left|\Re(\alpha_{jk})\right|$.
Observe that $\sigma_i\cdot \mathcal{F}(T)$ is a fundamental domain for the action of $\sigma_i\Gamma\sigma_i^{-1} = \Gamma$ on $\sigma_i \cdot \mathbb{H}^3_T$. Furthermore, $w \in \sigma_i \cdot \mathbb{H}^3_T$ is equivalent to
\begin{equation*}
    T \geq \max_{\substack{1 \leq j \leq h \\ \gamma \in \Gamma}} v(\sigma_{j}^{-1} \gamma \sigma_i^{-1} w ) = \max_{\substack{1 \leq j \leq h \\ \gamma \in \Gamma}} v(\sigma_{j}^{-1}\sigma_i^{-1}  \gamma w ) = \max_{\substack{1 \leq j \leq h \\ \gamma \in \Gamma}} v(\sigma_{j^{(i)}}^{-1}  \gamma w ) = \max_{\substack{1 \leq j \leq h \\ \gamma \in \Gamma}} v(\sigma_{j}^{-1}  \gamma w ),
\end{equation*}
hence $\sigma_i \cdot \mathbb{H}^3_T = \mathbb{H}^3_T$. Therefore the first integral on the right side of \eqref{basicstate2} is equal to
\begin{equation}\label{mainterminti}
    \int_{\sigma_i \cdot \mathcal{F}(T)} F(w) E_{i}(w,s) \, d \mu(w) = \int_{\mathcal{F}(T)} F(w) E_{i}(w,s) \, d \mu(w).
\end{equation}
Similarly, $\beta_{j}^i \cdot \mathcal{C}_\infty(T)$ is a fundamental domain for the action of $\beta_{j}^i\, \Gamma_\infty (\beta_{j}^i)^{-1} = \Gamma_{\infty}$ on $\beta_{j}^i \cdot \{w \in \mathbb{H}^3 : v > T\} = \{w \in \mathbb{H}^3 : v > T\}$. Using the preceding fact, and also that
\begin{equation*}
v((\beta_{j}^i)^{-1} w)=v
\end{equation*}
for all $w \in \mathbb{H}^3$, we deduce that the second term on the right side of \eqref{basicstate2} is equal to
\begin{equation}\label{mainsuminti}
    \int_{\mathcal{C}_{\infty}(T)}  \sum_{j=1}^{h}
    \Big(F(\sigma_{j} w) E_{i}( \sigma_{j}  w, s)- \varphi_{j}(v) \rho_{i j}(0,v,s) \Big) \, d \mu(w).
\end{equation}


\medskip
\noindent\textbf{Step 3: normalized matrix identity.} We now convert our key identity to matrix form. 
Let $\boldsymbol{e}_j \in \mathbb C^h$ denote the $j$-th standard column basis vector.
Starting from \eqref{basicstate2}, and then using \eqref{mainterminti} and \eqref{mainsuminti}, we obtain 
\begin{align} \label{basicstate3}
    &\boldsymbol{R}(F,s) + \boldsymbol{g}_T(s) + \boldsymbol{\Phi}(s) \cdot \boldsymbol{g}_T(2-s) = \frac{1}{\mathcal{V}} \int_{\mathcal{F}(T) } F(w) \cdot \boldsymbol{E}(w,s) \, d \mu(w) \nonumber \\
    &+\frac{1}{\mathcal{V}}\int_{\mathcal{C}_{\infty}(T)} \sum_{j=1}^{h} \Big(F(\sigma_{j} w) \cdot \boldsymbol{E}(\sigma_{j}w,s) -\varphi_{j}(v) \cdot \boldsymbol{\rho}(0,v,s) \cdot \boldsymbol{e}_j  \Big) \, d \mu(w)
\end{align}
for $\Re(s)>2 + \max_{j,k} \left|\Re(\alpha_{jk})\right|$. 


\medskip
\noindent\textbf{Step 4: analytic consequences.} Note that $\mathcal{F}(T)$ is compact and the second integrand of \eqref{basicstate3} is of rapid decay $O_{F, s, N}(v^{-N})$ for $s \in \C$ away from the poles of $\boldsymbol{E}(\sigma_{j}w,s)$ for each integer $1 \leq j \leq h$. The latter follows from the decay of the $K$-Bessel function as in 
\eqref{eq:Bessel_bound}. Therefore both integrals on the right side of \eqref{basicstate3} converge absolutely and locally uniformly for such $s \in \C$.

Recall from \cref{trivial_char_section} that any poles of (the entries of) $\boldsymbol{E}(\sigma_{j}w,s)$ must occur as poles of (the entries of) $\boldsymbol{\Phi}(s)$, and that $\boldsymbol{E}(\sigma_{j}w,s)$ has no poles on $\Re(s)=1$. We deduce that both integrals on the right side of \eqref{basicstate3} converge away from
\begin{equation}\label{convergeaway}
    s_0 \in \big\{\text{poles of } \boldsymbol{\Phi}(s)\big\}.
\end{equation} 
Also note that each entry of $\boldsymbol{g}_{T}(s)$ is meromorphic on $\mathbb{C}$ by  \eqref{gTcompute}. Thus \eqref{basicstate3} gives the meromorphic continuation of $\boldsymbol{R}(F,s)$ to all of $\mathbb{C}$.

Consulting \eqref{eisstarfunc} and \eqref{rhofunc}, we deduce that 
\begin{equation}\label{inv1}
\begin{aligned} 
    & \boldsymbol{g}_T(s)+ \boldsymbol{\Phi}(s) \cdot \boldsymbol{g}_T(2-s)   \\ 
    = \ & \boldsymbol{\Phi}(s) \cdot  \Big( \boldsymbol{g}_T(2-s) + \boldsymbol{\Phi}(2-s) \cdot \boldsymbol{g}_T(s) \Big)
\end{aligned}
\end{equation}
and 
\begin{equation}\label{inv2}
\begin{aligned} 
    & F(\sigma_{j} w) \cdot \boldsymbol{E}(\sigma_{j}w,s)-\varphi_{j}(v) \cdot \boldsymbol{\rho}(0,v,s) \cdot \boldsymbol{e}_j \\
    = \ & \boldsymbol{\Phi}(s) \cdot \Big(F(\sigma_{j} w) \cdot \boldsymbol{E}(\sigma_{j}w,2-s) -\varphi_{j}(v) \cdot \boldsymbol{\rho}(0,v,2-s) \cdot \boldsymbol{e}_j \Big).
\end{aligned}
\end{equation}
Therefore we obtain the functional equation \eqref{Rstarfunc} from \eqref{eisstarfunc}, \eqref{basicstate3}, \eqref{inv1}, and \eqref{inv2}.

We infer from \eqref{gTcompute} that 
\begin{align} \label{g1compute}
    g^{(j)}_1(s)=-\sum_{k=1}^{r} \frac{c_{jk}}{(2-\alpha_{jk}-s)^{n_{jk}+1}},
\end{align}
and that each entry of the vector
\begin{equation} \label{entire}
    \boldsymbol{g}_{T}(s)-\boldsymbol{g}_{1}(s)= \Big( \sum_{k=1}^{r} \frac{c_{jk}}{n_{jk} !} \cdot \Big(\frac{\partial}{\partial s}\Big)^{n_{jk}} \Big( \frac{T^{s+\alpha_{jk}-2}-1}{s+\alpha_{jk}-2} \Big) \Big)^{\top}_{1 \leq j \leq h}
\end{equation}
is an entire function of $s$. Therefore the statement \eqref{possiblepoles} about the potential location of poles of $\boldsymbol{R}(F, s)$ follows from \eqref{basicstate3}, \eqref{convergeaway}, and \eqref{g1compute}.
Since the entries of $\boldsymbol{g}_T(s)-\boldsymbol{g}_1(s)$ are entire, we also obtain \eqref{Rstarexp} from \eqref{basicstate3}. 

Furthermore, we observe that if $\alpha_{jk} \not\in \{0, 2\}$ for all integers $1\leq j\leq h$ and $1\leq k \leq r$, then \eqref{basicstate3} shows that each entry of $\boldsymbol{R}(F, s)$ has at most a simple pole at $s=2$, since the same is true for the entries of $\boldsymbol{\Phi}(s)$, $\boldsymbol{E}(\sigma_jw, s)$, and $\boldsymbol{\rho}(0, v, s)$.

We now compute residues at $s=2$ in \eqref{basicstate3} under the assumption that $\alpha_{jk} \not \in \{0,2\}$ for all integers $1\leq j\leq h$ and $1\leq k \leq r$. 
Using \eqref{constantres} and \eqref{rhores}, we have
\begin{equation*}
    \Res_{s=2} \boldsymbol{E}(\sigma_{j} w,s)=\frac{\mathcal{V}}{\vol(\Gamma \backslash \mathbb{H}^3)} \cdot \boldsymbol{1} = \Res_{s=2} \boldsymbol{\rho}(0,v,s)
    \cdot \boldsymbol{e}_j,
\end{equation*}
and thus 
\begin{align}
    & \Res_{s=2} \boldsymbol{R}(F,s) = -\Res_{s=2} \boldsymbol{g}_T(s) -\Res_{s=2} \big( \boldsymbol{\Phi}(s) \cdot
    \boldsymbol{g}_T(2-s) \big)  \label{res2a} \\
    &+ \frac{1}{\vol(\Gamma \backslash \mathbb{H}^3)} \Big( \int_{\mathcal{F}(T)} F(w) \, d \mu(w)  
    + \int_{\mathcal{C}_{\infty}(T)} \sum_{j=1}^{h} \Big( F(\sigma_{j} w)- \varphi_{j}(v) \Big) \, d \mu(w) \Big) \cdot \boldsymbol{1}. \qquad \label{res2b}
\end{align}
The first term on the right side of \eqref{res2a} is independent of $T$, since $\boldsymbol{g}_T(s)-\boldsymbol{g}_1(s)$ is entire, and also vanishes under our hypothesis that $\alpha_{jk} \neq 0$ for all integers $1\leq j \leq h$ and $1 \leq k \leq r$. By \eqref{gTcompute}, the second term on the right side of \eqref{res2a} is a vector whose entries are finite $\mathbb{C}$-linear combinations of terms of the form $T^{\alpha_{jk}-2} \log^{m}(T)$. If $\alpha<2$ then each of these terms is $o(1)$ as $T \to \infty$, so letting $T \to \infty$ we deduce \eqref{Rstares2}. 
The identity \eqref{RSidentity} for $\alpha < 1$, $\alpha<\Re(s)<2-\alpha$ and $s$ not a pole of $\boldsymbol{\Phi}(\cdot)$,
follows from taking the limit $T \to \infty$ in \eqref{basicstate3} and observing by \eqref{gTcompute} that 
\begin{equation*}
    \lim_{T \to \infty} \boldsymbol{g}_{T}(s)=\lim_{T \to \infty} \boldsymbol{g}_{T}(2-s)=\boldsymbol{0}.
\end{equation*}
\end{proof}


We will need uniform control on the growth of $R_j(F, s)$ for $s$ in fixed vertical strips, which is given in the following result. 

\begin{prop}[Growth estimate for Rankin--Selberg]\label{prop:RS_growth_estimate}
    Use the same notation and hypotheses as in \cref{zagreg}. Let $\varepsilon > 0$, and let $\Omega_{\varepsilon}$ be as in \eqref{Omegadef}. Assume that there is a parameter $M \geq 1$ such that for each integer $1\leq j \leq h$ we have
    \begin{equation}\label{eq:improved_F_asympt}
        \int_{\C/\Lambda} |{F(\sigma_j (z, v)) - \varphi_j(v)}| \, dx \, dy \ll_{\Gamma, \varepsilon} M v^{-100} \qquad \text{ for } \qquad v \geq \varepsilon.
    \end{equation}
Let $A \geq 1$ be such that $|{\Re(\alpha_{jk})}| \leq A$ and $n_{jk}\leq A$ for all $j$ and $k$. Then for every integer $1\leq i \leq h$, we have
\begin{equation}\label{eq:RS_growth_estimate}
R_i(F, s) \ll_{\Gamma, A, \varepsilon,r} (M+\gamma) \cdot |{s}|^{2A + 100}
\end{equation}
uniformly for 
\begin{equation}
s \in \Omega_{\varepsilon} \cap \{s \in \mathbb{C}: \min_{j,k} \big( |s-\alpha_{jk}|,\ |s-(2-\alpha_{jk})| \big )\geq\varepsilon \}=:\mathcal{S}_{\varepsilon} ,
\end{equation}
where $\displaystyle\gamma :=  \max_{j, k} |c_{jk}|$.
\end{prop}

\begin{proof}
    Denote as before $\alpha:=\max_{j,k} \Re(\alpha_{jk})$. Our starting point is \eqref{basicstate3}, where we choose $T = 1$ and take the $i$-th row. After applying the triangle inequality, the contribution of the terms $|{\boldsymbol{g}_1(s)}| + |{\boldsymbol{\Phi}(s) \cdot \boldsymbol{g}_1(2-s)}|$ to \eqref{eq:RS_growth_estimate} is $\ll_{\Gamma, A, \varepsilon,r} \gamma$. Indeed, since $s \in \mathcal{S}_{\varepsilon}$, note that $s \in \C$ is at least $\varepsilon$ away from the relevant poles. Moreover, $A \geq n_{jk} \geq 0$ and the entries of $\boldsymbol{\Phi}(s)$ are bounded (depending on $\Gamma$ and $\varepsilon$) on the set $\Omega_{\varepsilon}$ by Lemma \ref{omegalem}.
Thus the claimed bound follows from \eqref{g1compute}.
    
  Now we estimate the remaining terms. The first term on the right side of \eqref{basicstate3} is
    \begin{equation*}
        \frac{1}{\mathcal{V}} \int_{\mathcal{F}(1)} F(w) E_i(w, s) \, d\mu(w) \ll_{\Gamma, \varepsilon} |{s}|^{3/2} \int_{\mathcal{F}(1)} |F(w)| \, d\mu(w) \ll_{\Gamma, A, \varepsilon,r} (M+\gamma) |{s}|^{3/2}.
    \end{equation*}
    Indeed, the first inequality above follows from \cref{lemma:Eis_pointwise_bound}. The second inequality above follows from \eqref{eq:improved_F_asympt}, and the estimate $|{\varphi_{j}(v)}| \ll_{A, \varepsilon,r} \gamma v^{\alpha + \varepsilon}$ for $v\geq \varepsilon$, since they imply
    \begin{equation*}
        \int_{\mathcal{F}(1)} |F(w)| \, d\mu(w) \ll \sum_{j=1}^h \int_{\frac{1}{10}}^1 \int_{\C/\Lambda} \big( |{F(\sigma_j w) - \varphi_j(v)}| + |{\varphi_j(v)}|\big) \, dx \, dy \, dv
        \ll_{\Gamma, A, \varepsilon,r} M + \gamma.
    \end{equation*}
    
    The second term on the right side of \eqref{basicstate3} can be written as 
    \begin{equation}\label{eq:RS_bound_interm}
        \frac{1}{\mathcal{V}} \sum_{j=1}^{h} \int_{1}^\infty \int_{\C/\Lambda} \Big( \big(F(\sigma_{j} w) - \varphi_{j}(v)\big) \cdot E_i(\sigma_j w, s) + \varphi_{j}(v) \cdot E_i^1(\sigma_{j}w,s) \Big) \, dx \, dy \, \frac{dv}{v^3} ,
    \end{equation}
    where $E_i^{A}$ denotes the height $A$ truncated Eisenstein series in \eqref{eq:trunc_eis_series}.
    We apply \cref{lemma:Eis_pointwise_bound} and \eqref{eq:improved_F_asympt} to the first integrand of \eqref{eq:RS_bound_interm}, as above. This time the resulting bound is 
    \begin{equation*}
        \ll_{\Gamma, A, \varepsilon} \int_{1}^\infty M v^{-50} |{s}|^{3/2} \, dv \ll M |{s}|^{3/2}.
    \end{equation*}
    Let $C_{\Gamma}> 0$ be the constant depending on $\Gamma$ in the statement of \cref{lemma:truncated_Eis_decay}. 
    We choose any constant $D_{\Gamma} \geq C_{\Gamma}$ such that $D_{\Gamma}|s|^2 \geq \max(1,C_{\Gamma}|s|^2)$. 
   For the second integrand of \eqref{eq:RS_bound_interm}, if $1 \leq v \leq D_{\Gamma} |s|^2$ then apply the same pointwise bound of \cref{lemma:Eis_pointwise_bound}, observing that it still holds for the truncation $E_i^1(\sigma_jw, s)$ since $\rho_{ij}(0, v, s) \ll_{\Gamma, \varepsilon} v^{\Re(s)}$. If $v >D_{\Gamma} |s|^2$ then apply Cauchy--Schwarz in $z$ followed by \cref{lemma:truncated_Eis_decay} instead. We conclude that the second integrand of \eqref{eq:RS_bound_interm} contributes
    \begin{align*}
        &\ll_{\Gamma, A, \varepsilon,r} \int_1^{D_{\Gamma} |s|^2} \gamma v^{\alpha+\varepsilon + \Re(s)} |{s}|^{1/2} \Big(1 + \frac{|{s}|}{v}\Big) \, \frac{dv}{v^3} +
         \int_{D_{\Gamma} |{s}|^2}^\infty \gamma v^{\alpha+\varepsilon} e^{-m_\Gamma \cdot \frac{v - D_{\Gamma}|{s}|^2}{2}}(D_{\Gamma} |{s}|^2)^{\Re(s)} |{s}|^{1/2} \,  \frac{dv}{v^3} \\
        & \ll_{\Gamma, A,r} \gamma \cdot |{s}|^{2|{\alpha}| + 100} .
    \end{align*} 
    Collecting all the estimates leads to \eqref{eq:RS_growth_estimate}, as desired.
\end{proof}


\section{Quartic Eisenstein series and theta functions}

\subsection{Quartic Eisenstein series}

Recall that $K=\mathbb{Q}(i)$. Let $0 \neq L \in \mathcal{O}_K = \Z[i]$ and let $\Gamma(L)$ denote the principal congruence group 
\begin{equation*}
    \Gamma(L) := \big\{ g \in \SL_2(\mathcal{O}_K) : g \equiv I \pmod{L}\big\}.
\end{equation*}
Note that $\Gamma(L)$ is a normal, co-finite but not co-compact subgroup of $\SL_2(\mathcal{O}_K)$, and if $L \nmid 2$ then $\Gamma(L)_\infty = \Gamma(L)'_\infty$, $\Lambda = L \mathcal{O}_K$, $\Lambda^* = (\lambda^2 L )^{-1}\mathcal{O}_K$, and $\mathcal{V} = \vol(\C/\Lambda) =  N(L)$.

Let $\chi_4: \Gamma(\lambda^4) \to S^1$ denote the quartic Kubota \cite{Kub} symbol 
\begin{equation*}
    \chi_4(\gamma)=
    \begin{cases} 
        \big( \frac{c}{a} \big)_4 & \quad \text{if} \quad c \neq 0, \\ 1 & \quad \text{if} \quad c=0,
    \end{cases} 
    \qquad \text{ for }\gamma=\pMatrix {a} {b} {c} {d} \in \Gamma(\lambda^{4}),
\end{equation*}
where $\big(\frac{c}{a}\big)_4$ denotes the standard quartic residue symbol on $K$. Then $\chi_4$ is a character of $\Gamma(\lambda^4)$, see \cite[Lemma 1]{Suz1}.

Let $\{\kappa_1:=\infty, \kappa_2, \dots,\kappa_h\}$ denote the set of all cusps of $\Gamma(\lambda^4)$ and $\{\eta_1:=\infty, \eta_2, \dots, \eta_m \}$ denote the subset of cusps of $\Gamma(\lambda^4)$ that are essential with respect to $\chi_4$. 
 We assume that the cusps are enumerated such that $\kappa_j = \eta_j$ for each integer $1 \leq j \leq m$. One can consult \cite[Proposition 1]{Suz1} for the complete list of $m = 24$ essential cusps, and we choose scaling matrices $\sigma_i = \sigma_{\eta_i}$ according to \cite[Table 1]{Suz1}. 

\begin{remark}[Choice of scaling matrices]
    Observe from \eqref{varphidefgen} that $\varphi_{ij}(\delta, s, \chi_4)$ depends on the choice of scaling matrix $\sigma_j$ (though not on $\sigma_i$), so to use the results of Suzuki \cite{Suz1} we need to choose the same ones. The dependence on $\sigma_j$ is mild, with a different choice of $\sigma_j \in \SL_2(\mathcal{O}_K)$, i.e.\ $\sigma_j \mapsto \sigma_j \beta$ for $\beta \in \SL_2(\mathcal{O}_K)_\infty$, giving $\varphi_{ij}(\delta, s, \chi_4) \mapsto \check{e}(\delta r)\varphi_{ij}(\pm \delta, s, \chi_4)$ for some $r \in \mathcal{O}_K$ and choice of sign $\pm$. Diaconu \cite{Dia} does not make the same choice as Suzuki. Instead, he allows multiplication on the right by $\left(\begin{smallmatrix}
    1 & \mu \\ 0 & 1
    \end{smallmatrix}\right)$ for an arbitrary $\mu \in \mathcal{O}_K$. This only leaves the ambiguity $\varphi_{ij}(\delta, s, \chi_4) \mapsto \check{e}(\delta \mu)\varphi_{ij}(\delta, s, \chi_4)$, so $\varphi_{ij}(0, s, \chi_4)$ is well-defined (and he writes explicit values for them at $\eta_j = 0, \frac{\lambda^4}{1+\lambda^3}$). For $\eta_1 = \infty$, Diaconu does fix $\sigma_1 = I$, so $\varphi_{i1}(\delta, s, \chi_4)$ is well-defined for all $\delta\in \Lambda^*$, and he also writes it explicitly for certain $i$. Thus despite the slightly different setups, we may use results from both sources \cite{Dia, Suz1} without conflict.
\end{remark}

With this choice of scaling matrices, the next result establishes that
the quartic Kubota symbol is invariant under conjugation by each $\sigma_i$ on $\Gamma(\lambda^4)$.

\begin{lemma}\label{chi4_conjugation_lemma}
    For any $g \in \Gamma(\lambda^4)$ and scaling matrix $\sigma_i$ given in \cite[Table 1]{Suz1} for each integer $1 \leq i \leq m$, we have
    \begin{equation} \label{chi4prop}
        \chi_4(\sigma_i g \sigma_i^{-1}) = \chi_4(g) = \chi_4(\sigma_i^{-1} g \sigma_i).
    \end{equation}
\end{lemma}

\begin{proof}
    For convenience, let us denote the quartic residue symbol $\big(\frac{\star}{\star}\big)_4$ by $\big(\frac{\star}{\star}\big)$. The second equality in the lemma follows from the first one by changing variables $g \mapsto \sigma_i^{-1} g \sigma_i \in \Gamma(\lambda^4)$, so we only need to show that $\chi_4(\sigma_i g \sigma_i^{-1}) = \chi_4(g)$.
  Let  $\sigma_i = \left(\begin{smallmatrix}
        \alpha & \beta \\ \gamma & \delta
    \end{smallmatrix}\right) \in \SL_2(\mathcal{O}_K)$ and $g = \left(\begin{smallmatrix}
        a& b \\ c & d
    \end{smallmatrix}\right) \in \Gamma(\lambda^4)$.
    
  We first establish \eqref{chi4prop} in the case when $c=0$.
Since $g\in\Gamma(\lambda^4)$, we have $a=d=1$ when $c=0$,
and hence $g\in\Gamma(\lambda^4)_\infty$. Thus $\chi_4(g)=1$. Moreover,
$ \sigma_i g\sigma_i^{-1}\in \sigma_i\Gamma(\lambda^4)_\infty\sigma_i^{-1}=\Gamma(\lambda^4)_i$,
and since the cusp $\eta_i$ is essential with respect to $\chi_4$, we have
$\chi_4(\sigma_i g\sigma_i^{-1})=1$.

Thus we may assume $c \neq 0$, and this case will follow by a direct computation using \cite[Lemmas 3 and 4]{Suz1}, which imply that
    \begin{equation*}
        \overline{\chi_4}(\sigma_i g \sigma_i^{-1}) = \begin{cases}
            \big(\frac{\alpha}{\gamma}\big)^{-1} \big(\frac{d\alpha - c\beta}{c\delta - d\gamma}\big) & \text{if } \lambda \nmid \gamma, \\
            \big(\frac{-\gamma}{\alpha}\big)^{-1} \big(\frac{c\delta - d\gamma}{d\alpha - c\beta}\big) & \text{if } \lambda \mid \gamma.
        \end{cases}
    \end{equation*} 
  
Next we use the congruence restrictions imposed by the choice of scaling matrices to apply quartic reciprocity\footnote{There is a typo in the second equation of \cite[(2.4)]{Suz1}, which should be $\big(\frac{i}{\delta}\big) = i (-1)^{a_4+a_5}$.} \cite[(2.1) and (2.2)]{Suz1} and the relation $\big(\frac{x}{y}\big) = \big(\frac{x}{y + wx}\big)$ for any $x\equiv 0\pmod{2}$ and $w \equiv 0 \pmod{4}$ \cite[(2.5)]{Suz1}. In what follows, let $(\cdot ,\cdot)_{\lambda}$ denote the Hilbert symbol
defined in \cite[pg.~72]{Suz1}.
        
If $\lambda \mid \gamma$ we infer from \cite[Table 1]{Suz1} that $\alpha \equiv 1\pmod{\lambda^3}$, $\beta \equiv 0 \pmod{4}$, $\gamma \equiv 0 \pmod{2}$, and $\delta \equiv 1 \pmod{\lambda^3}$. If $\beta = 0$, then $\alpha = \delta = 1$ and the result also follows. Otherwise $\beta\neq 0$, so denote $\beta' = \frac{\beta}{(\beta, d)} \equiv 0 \pmod{4}$ and $d' = \frac{d}{(\beta, d)} \equiv 1 \pmod{\lambda^3}$, where we choose the unique ideal generator $(\beta, d) \equiv 1 \pmod{\lambda^3}$. Also note that $c \equiv 0 \pmod{\lambda^4}$. We compute that
\begin{align*}
        \Big(\frac{-\gamma}{\alpha}\Big)^{-1} \Big(\frac{c\delta - d\gamma}{d\alpha - c\beta}\Big) &= \Big(\frac{-\gamma}{\alpha}\Big)^{-1} \Big(\frac{c\delta}{(\beta, d)}\Big) \Big(\frac{\beta'(c\delta - d\gamma)+\delta(d'\alpha - c\beta')}{d'\alpha - c\beta'}\Big)\Big(\frac{\beta'}{d'\alpha - c\beta'}\Big)^{-1} \\
        &= \Big(\frac{-\gamma}{\alpha}\Big)^{-1} \Big(\frac{c\delta}{(\beta, d)}\Big) \Big(\frac{d'}{d'\alpha - c\beta'}\Big)\Big(\frac{\beta'}{d'\alpha - c\beta'}\Big)^{-1} \\
        &= \Big(\frac{-\gamma}{\alpha}\Big)^{-1} \Big(\frac{c\delta}{(\beta, d)}\Big) \Big(\frac{-c\beta'}{d'}\Big) (d'\alpha - c\beta', d')_\lambda \Big(\frac{\beta'}{d'\alpha}\Big)^{-1} \\
        &= \Big(\frac{-\gamma}{\alpha}\Big)^{-1} \Big(\frac{c}{d}\Big) \Big(\frac{\delta}{(\beta, d)}\Big) \Big(\frac{-1}{d'}\Big) (d'\alpha, d')_\lambda \Big(\frac{\beta'}{\alpha}\Big)^{-1} \\
        & = \Big(\frac{c}{d}\Big) \Big(\frac{\delta}{(\beta, d)}\Big) \Big(\frac{-\gamma\beta'}{\alpha}\Big)^{-1} (\alpha, d')_\lambda  = \Big(\frac{c}{d}\Big) \Big(\frac{\delta}{(\beta, d)}\Big) \Big(\frac{(\beta, d)}{\alpha}\Big) (\alpha, d')_\lambda.
    \end{align*}

    Furthermore, since $b\equiv c\equiv 0 \pmod{4}$ we have
    \begin{equation*}
        \Big(\frac{c}{d}\Big) = \Big(\frac{c}{ad}\Big) \Big(\frac{c}{a}\Big)^{-1} = \Big(\frac{c}{ad-bc}\Big) \Big(\frac{c}{a}\Big)^{-1} = \Big(\frac{c}{a}\Big)^{-1} = \overline{\chi_4}(g),
    \end{equation*}
    and since $d\equiv 1\pmod{4}$ we have $(\alpha, d')_\lambda = (\alpha, d)_\lambda (\alpha, (\beta, d))^{-1}_\lambda = (\alpha, (\beta, d))_\lambda$, hence
    \begin{equation*}
        \Big(\frac{\delta}{(\beta, d)}\Big) \Big(\frac{(\beta, d)}{\alpha}\Big) (\alpha, d')_\lambda = \Big(\frac{\alpha\delta}{(\beta, d)}\Big) = \Big(\frac{\alpha\delta -\beta\gamma}{(\beta, d)}\Big) = 1.
    \end{equation*}
    This gives the result for $\lambda \mid \gamma$. The case $\lambda \nmid \gamma$ follows from a similar computation using the conditions $\beta \equiv \gamma \equiv 1 \pmod{2}$ and $\delta \equiv 0 \pmod{4}$ imposed by the choice of $\sigma_i$ \cite[Table~1]{Suz1}.
\end{proof}

Consider the Eisenstein series $E_{i}(w,s,\chi_4)$ on $\Gamma(\lambda^4)$ attached to the essential cusp $\eta_i$. Let us show how to write the Fourier expansion of an Eisenstein series at an arbitrary essential cusp in terms of the expansion of another Eisenstein series at $\infty$.

For integers $1\leq i, j\leq m$, we have $\chi_4(\sigma_{j}^{-1}\sigma_i g \sigma_{i}^{-1}\sigma_j) = \chi_4(g) = 1$ for any $g \in \Gamma(\lambda^4)_{\infty}$ by \cref{chi4_conjugation_lemma}. Thus the cusp $\sigma_{j}^{-1}\sigma_i \cdot\infty$ is essential with respect to $\chi_4$, and so there exist $\gamma_{j}^i \in \Gamma(\lambda^4)$ and $1\leq j^{(i)} \leq m$ such that $\gamma_{j}^i \cdot\sigma_{j}^{-1}\sigma_i \cdot \infty = \eta_{j^{(i)}}$.  We deduce that there is $\beta_j^i \in \SL_2(\mathcal{O}_K)_\infty$ such that 
\begin{equation}\label{cusp_shuffling}
    \gamma_{j}^i \cdot\sigma_{j}^{-1} \sigma_i \cdot \beta_{j}^i  =  \sigma_{j^{(i)}}, \qquad \text{hence} \qquad  \sigma_{i}^{-1} \sigma_j   =   \beta_{j}^i \cdot \sigma_{j^{(i)}}^{-1} \cdot \gamma_{j}^i.
\end{equation}

\begin{lemma}\label{Eis_change_cusp_lemma}
    For any integers $1\leq i, j\leq m$, we have
    \begin{equation*}
        E_i(\sigma_jw, s, \chi_4) = \chi_4(\gamma_j^i) \cdot E_{j^{(i)}}(w, s, \chi_4).
    \end{equation*}
\end{lemma}

\begin{proof}
By meromorphic continuation, it suffices to establish the result for $\Re(s) > 2$. For simplicity, we denote $\Gamma := \Gamma(\lambda^4)$ and $\Gamma_i = \sigma_i \Gamma(\lambda^4)_\infty \sigma_i^{-1}$. Then
\begin{align*}
    E_i(\sigma_jw, s, \chi_4) = \sum_{\gamma \in \Gamma_i \backslash\Gamma} v(\sigma_i^{-1} \gamma \sigma_j w)^s \cdot \overline{\chi_4}(\gamma) = \sum_{\gamma' \in \sigma_j^{-1}\sigma_i \Gamma_\infty \sigma_i^{-1}\sigma_j\backslash\Gamma} v(\sigma_i^{-1} \sigma_j \gamma' w)^s \cdot \overline{\chi_4}(\sigma_j\gamma'\sigma_j^{-1}),
\end{align*}
where $\gamma' = \sigma_j^{-1}\gamma\sigma_j \in \Gamma$. By \cref{chi4_conjugation_lemma} we have $\overline{\chi_4}(\sigma_j\gamma'\sigma_j^{-1}) = \overline{\chi_4}(\gamma')$. Substituting \eqref{cusp_shuffling} we obtain $\sigma_j^{-1}\sigma_i \cdot \Gamma_\infty \cdot \sigma_i^{-1}\sigma_j = (\gamma_j^i)^{-1}\sigma_{j^{(i)}}\cdot \Gamma_\infty\cdot \sigma_{j^{(i)}}^{-1}\gamma_j^i$ and $v(\sigma_i^{-1} \sigma_j \gamma' w) = v(\sigma_{j^{(i)}}^{-1} \gamma_j^i \gamma' w)$. Changing variables to $\gamma = \gamma_j^i \gamma' \in \Gamma$ we conclude that
\begin{equation*}
    E_i(\sigma_jw, s, \chi_4) = \chi_4(\gamma_j^i) \cdot \sum_{\gamma \in \Gamma_{j^{(i)}} \backslash\Gamma} v(\sigma_{j^{(i)}}^{-1} \gamma w)^s \overline{\chi_4}(\gamma) = \chi_4(\gamma_j^i) \cdot E_{j^{(i)}}(w, s, \chi_4).
\end{equation*}
\end{proof}

It will be useful to define certain normalized Eisenstein series. To this end, let
\begin{equation*} 
    \zeta^{*}_{K}(s):=|d_K|^{\frac{s}{2}} \Gamma_\C(s) \zeta_{K}(s),
\end{equation*}
where $d_K = -4$, $\Gamma_\C(s) := 2(2 \pi)^{-s} \Gamma(s)$ and $\zeta_{K}(s)$ denotes the Dedekind zeta function of $K$. The completed zeta function $\zeta^{*}_{K}(s)$ is holomorphic on $\mathbb{C}$ except for simple poles at $s=0$ and $s=1$, and satisfies the functional equation
\begin{equation} \label{zetafunc}
    \zeta^{*}_{K}(s)=\zeta^{*}_{K}(1-s).
\end{equation}

Let
\begin{equation} \label{fancyE}
    \widetilde{E}_{i}(w,s,\chi_4):= G(s) \zeta_{\lambda} (4s -3) E_{i}(w,s,\chi_4), 
\end{equation}
where we denote
\begin{equation*}
    G(s):= \Gamma_\C \Big(s-\frac{3}{4} \Big )  \Gamma_\C \Big(s-\frac{1}{2} \Big)  \Gamma_\C \Big(s-\frac{1}{4} \Big),
\end{equation*}
and for $\Re(s)>1$,
\begin{equation*}
    \zeta_{\lambda}(s) := \sum_{m \equiv 1 \pmod{\lambda^{3}}} \frac{1}{N(m)^s} = \zeta_K(s) \cdot (1-2^{-s}).
\end{equation*}
Analogously, for $\delta \in \Lambda^{*}$ define 
\begin{equation}\label{rhotilde}
    \widetilde{\varphi}_{ij}(\delta,s,\chi_4) := G(s) \zeta_{\lambda} (4s -3) \varphi_{ij}(\delta,s,\chi_4).
\end{equation}

\begin{lemma}[Constant term computation] \label{scatter}
    For any integers $1 \leq i, j\leq m$, we have
    \begin{equation} \label{varphigenell}
        \varphi_{ij}(0,s,\chi_4) = P_{ij}(s) \cdot \frac{\zeta_{\lambda}(4s-4)}{\zeta_{\lambda}(4s-3)},
    \end{equation} 
    where either $P_{ij}(s) = c_{ij}$ or $P_{ij}(s) = c_{ij}(2^{4s-4}-1)^{-1}$ for some constant $c_{ij} \in \C$ (possibly equal to zero) that depends only on $i$ and $j$. In particular, for $\eta_{i_1} = 0$ and $\eta_{i_2} = \frac{\lambda^4}{1+\lambda^3}$ (which are essential), recalling that $\eta_1 = \infty$ we have 
    \begin{equation} \label{zeroinfty}
        \varphi_{i_1 1}(0,s,\chi_4)=\mathcal{V}^{-1} \cdot \frac{\zeta_{\lambda}(4s-4)}{\zeta_{\lambda}(4s-3)} \quad \text{and} \quad \varphi_{i_2 i_2}(0,s,\chi_4) = \frac{2 \mathcal{V}^{-1}}{(2^{4s-4} - 1)}\cdot  \frac{\zeta_{\lambda}(4s-4)}{\zeta_{\lambda}(4s-3)}.
    \end{equation}
    The functions $\widetilde{\varphi}_{ij}(0, s, \chi_4)$ have meromorphic continuation to $\C$ with at most one possible (simple) pole, at $s = \frac{5}{4}$. Moreover, they are bounded when $|\Im(s)|$ is large in vertical strips of finite width.
\end{lemma}

\begin{remark}
    The cusp pairs $(0, \infty)$ and $(\frac{\lambda^4}{1+\lambda^3}, \frac{\lambda^4}{1+\lambda^3})$ are important because they will be used in our application. The corresponding constant terms having a pole at $s = \frac{5}{4}$ will imply that the main term of the bias in the quartic large sieve is not identically zero.
\end{remark}

\begin{proof}
    The first claim \eqref{varphigenell} and the properties of $P_{ij}(s)$ follow from the explicit computations in \cite[Section 5]{Suz1}, more precisely from equations (5.2) to (5.10) there\footnote{Note that there is a typo in \cite[(5.10)]{Suz1}, where the factor $(2^{4s-1}-1)^{-1}$ should be $(2^{4s-4}-1)^{-1}$.}. The claim \eqref{zeroinfty} is a particular case which is made explicit in \cite[Proposition 4 and (5.1)]{Suz1}. For the meromorphic continuation of $\widetilde{\varphi}_{ij}(0, s, \chi_4)$, observe from the Gamma quadruplication formula
    \begin{equation*}
        \Gamma(4z) = 2^{8z-5/2} \pi^{-3/2} \Gamma(z)\Gamma\Big(z+\frac{1}{4}\Big) \Gamma\Big(z+\frac{1}{2}\Big)\Gamma\Big(z+\frac{3}{4}\Big)
    \end{equation*}
    that
    \begin{equation}\label{eq:G_Gamma_identity}
        G(s) = \frac{|d_K|^{\frac{4s-4}{2}} \Gamma_\C(4s-4)}{2^{11s-14} \pi^{1-s} \Gamma(s-1)},
    \end{equation}
    hence
    \begin{equation*}
        \widetilde{\varphi}_{ij}(0, s, \chi_4) = P_{ij}(s) \cdot (1-2^{4-4s}) \cdot \zeta_K(4s-4) \cdot G(s) = \frac{P_{ij}(s) \cdot (2^{4s-4}-1) \cdot \zeta^*_K(4s-4)}{2^{15s-18} \pi^{1-s} \Gamma(s-1)}.
    \end{equation*}
    Since $\zeta_K^*(4s-4)$ is meromorphic with (simple) poles only at $s=\frac{5}{4}$ and $s=1$, with the latter canceled by the pole of $\Gamma(s-1)$, we obtain the desired pole structure from the description of $P_{ij}(s)$ given above. Boundedness (away from $s=\frac{5}{4}$) in fixed vertical strips follows from the rapid decay of $G(s)$ and polynomial bounds for $\zeta_K(4s-4)$.
\end{proof}

Applying the functional equation \eqref{Efunceq} for the Eisenstein series with $w \mapsto \sigma_k w$, taking the Fourier coefficient at $0 \neq \delta \in \Lambda^*$ in both sides, and combining it with \eqref{eq:G_Gamma_identity}, \eqref{Kintdef}, and the functional equation \eqref{zetafunc} for $\zeta_K^*(s)$, we obtain
\begin{equation}\label{eq:Eisenstein_functional_equation}
    \widetilde{\varphi}_{ik}(\delta, s, \chi_4) = N(\delta)^{1-s} \frac{2^{30 - 28s}}{(1 - 2^{4s-5})} \sum_{j=1}^m \widetilde{\varphi}_{jk}(\delta, 2-s, \chi_4) \cdot P_{ij}(s) \cdot (2^{4s-4}-1)
\end{equation}
for all integers $1\leq i, k \leq m$.

\begin{prop}[Bounds for non-zero Fourier coefficients] \label{merovarphi}  
    Let $0 \neq \delta \in \Lambda^* = \lambda^{-6}\mathcal{O}_K$ and $1\leq i, j\leq m$ be integers. Then functions $\widetilde{\varphi}_{i j}(\delta,s,\chi_4)$ can be meromorphically continued to $\mathbb{C}$, with at most two possible (simple) poles, at $s=\frac{3}{4}$ and $s=\frac{5}{4}$. Moreover, they are bounded when $|{\Im(s)}|$ is large in vertical strips of finite width. For any $0<\varepsilon<\frac{1}{100}$, we have
    \begin{equation} \label{convexbd}
        \varphi_{i j}(\delta,s,\chi_4)  \ll_{\varepsilon, \ord_{\lambda}(\delta)} \big(N(\delta)^{\frac{1}{2}} (|s|^3+1)\big)^{\frac{3}{2}-\Re(s)+\varepsilon}
    \end{equation}
    uniformly for $1+\varepsilon \leq \Re(s) \leq \frac{3}{2}+\varepsilon$ and $|s- \frac{5}{4} | \geq \varepsilon$, while $\varphi_{i j}(\delta,s,\chi_4) \ll_{\varepsilon, \ord_{\lambda}(\delta)} 1$ for $\Re(s) \geq \frac{3}{2}+\varepsilon$. 
\end{prop}

\begin{proof}
    The first part of the statement for $j=1$, describing the analytic properties of $\widetilde{\varphi}_{i 1}(\delta,s,\chi_4)$, is given in \cite[Theorem 2.1]{Dia}. The proof is analogous to the cubic case worked out in detail in \cite[Theorem~6.1]{Pat1}, utilizing the constant terms described in \cref{scatter} and the functional equation \eqref{eq:Eisenstein_functional_equation}. The convexity bound \eqref{convexbd} for $j=1$ was then computed in \cite[Proposition 4.3]{DDHL}. 
    
    The same proofs work for the case of a general integer $1\leq j\leq m$. Since we could not locate a detailed reference, we use a trick to deduce it from the case $j=1$. Indeed, taking Fourier expansions in \cref{Eis_change_cusp_lemma}, for any $0 \neq \delta \in \Lambda^* = \lambda^{-6}\mathcal{O}_K$ we have
    \begin{equation*}
        \varphi_{ij}(\delta, s, \chi_4) = \chi_4(\gamma_j^i)\cdot \varphi_{j^{(i)} 1}(\delta, s, \chi_4)
    \end{equation*}
    for some $\gamma_j^i \in \Gamma(\lambda^4)$ and some integer $1\leq j^{(i)}\leq m$. Thus all of the desired properties follow from their counterparts in the case $j=1$.
\end{proof}

\subsection{Quartic theta functions}

Kubota proved \cite[Chapter 6]{Kub} that the only possible singularity of $E_{i}(w,s,\chi_4)$ in the half-plane $\Re(s)>1$ is a simple pole at $s=\frac{5}{4}$ (see also \cref{scatter}). The residue of this pole is known as the quartic metaplectic theta function (attached to the cusp $\eta_i$), given by
\begin{equation} \label{thetadef}
    \vartheta_{i}(w):=\Res_{s=\frac{5}{4}} E_{i}(w,s,\chi_4)
\end{equation}
 for each $w \in \mathbb{H}^3$.
Using \eqref{eisfourier}, \eqref{Kintdef}, \eqref{rhozero}, and \eqref{rhononzero} we obtain the Fourier expansion 
\begin{align} \label{thetafourier}
    \vartheta_{i}(\sigma_{j} w) = 4\pi \Psi_{ij}(0)  v^{3/4} +\frac{(2 \pi)^{5/4}}{\Gamma(5/4)} v \sum_{0 \neq \delta \in \Lambda^{*}} N(\delta)^{1/8} \Psi_{ij}(\delta) K_{1/4}(4 \pi |{\delta}|v) \check{e}(\delta z), \qquad
\end{align}
where for $\delta \in \Lambda^* = \lambda^{-6}\mathcal{O}_K$ we denote 
\begin{equation} \label{psidef}
    \Psi_{ij}(\delta):=\Res_{s=\frac{5}{4}} \varphi_{ij}(\delta,s,\chi_4).
\end{equation} 
The constant term in \eqref{thetafourier} is square-integrable on $\Gamma(\lambda^4) \backslash \mathbb{H}^3$ with respect to the measure $d \mu(w)$, thus $\vartheta_{i} \in L^2(\Gamma(\lambda^4) \backslash \mathbb{H}^3,\chi_4)$. The function $\vartheta_{i}$ is both a Laplace and Hecke eigenfunction, since so is the Eisenstein series that originated it.

\subsection{Dirichlet series for quartic Gauss sums}

For any $0 \neq \delta \in \mathcal{O}_K^* = \lambda^{-2} \mathcal{O}_K$ and $u \equiv 1 \pmod{\lambda^3}$, denote 
\begin{equation*}
    \psi(\delta,s;u) := \sum_{\substack{c \in \mathcal{O}_K \\ c \equiv u \pmod{4}}}  \frac{g_4(\delta, c)}{N(c)^s},
\end{equation*}
which converges absolutely if $\Re(s)>\frac{3}{2}$. By \cite[(2.37)]{Dia} or \cite[Proposition 2]{Suz1}, for the essential cusps $\eta_{i_1} := 0$, $\eta_{i_2} := \frac{\lambda^4}{1+\lambda^3}$, $\eta_1=\infty$ we have
\begin{equation} \label{connect}
    \psi(\delta, s; u)= \mathcal{V} \cdot
    \begin{cases}
        \varphi_{i_1 1}(\lambda^{-4} \delta,s,\chi_4) & \quad \text{if} \quad u \equiv 1 \pmod{4}, \\
        - \varphi_{i_2 1}(\lambda^{-4} \delta,s,\chi_4) & \quad  \text{if} \quad u \equiv 1+\lambda^3 \pmod{4}. 
    \end{cases}
\end{equation}
The meromorphic continuation of $\psi(\delta,s;u)$ to $\mathbb{C}$ follows from \eqref{rhotilde}, \cref{merovarphi}, and \eqref{connect}. Since $G(s) \zeta_{\lambda}(4s-3)$ is both holomorphic and zero-free for $\Re(s)>1$, the function $\psi(\delta, s;u)$ has at most one (simple) pole in the half-plane $\Re(s)>1$, at $s=\frac{5}{4}$. For $0\neq \delta \in \lambda^{-2}\mathcal{O}_K$ we deduce from \eqref{psidef} and \eqref{connect} that 
\begin{equation} \label{resD}
    \tau_4(\delta;u):=\Res_{s=\frac{5}{4}} \psi(\delta,s;u)= \mathcal{V} \cdot
    \begin{cases}
        \Psi_{i_1 1}(\lambda^{-4} \delta) & \text{if} \quad u \equiv 1 \pmod{4}, \\
        -\Psi_{i_2 1}(\lambda^{-4} \delta) & \text{if} \quad u \equiv 1+\lambda^3 \pmod{4}.
    \end{cases}
\end{equation}

No general closed form formula is known for the $\tau_4(\delta;u)$. Patterson \cite{Pat2}, with refinements by Eckhardt--Patterson \cite[Conjecture~2.11]{EckPat}, conjectured essentially that $N(\pi)^{1/4} \tau_4(\pi;u)^2$ is proportional to $\overline{\widetilde{g}_4(\pi)}$ for $\pi \equiv 1 \pmod{4}$ prime. 
Patterson's conjecture was based on heuristics for certain Rankin--Selberg convolutions of the quartic metaplectic theta function.
We refer the reader to \cite{BH} for a detailed description. 

Suzuki \cite{Suz1} obtained partial information about these residues. In what follows, let $0 \neq \delta \in \lambda^{-2}\mathcal{O}_K$ and $m \in \mathcal{O}_K$ satisfy $m \equiv 1 \pmod{\lambda^3}$. Suzuki \cite[Theorem~3]{Suz1} established the periodicity relation
\begin{equation} \label{quartrel}
    \tau_4(m^4 \delta;u)=\tau_4(\delta;u).
\end{equation}
For $(m, \delta) = 1$ and $\mu^2(m)=1$, he also showed \cite[Theorems 4, 5, 6]{Suz1} that
\begin{equation} \label{cuberel}
    \tau_4(m^3 \delta;u)=0 \quad \text{for} \quad m \neq 1,
\end{equation} 
\begin{equation} \label{squarerel}
    \tau_4(m^2 \delta;u)= \frac{\overline{g_4(\delta,m)}}{N(m)^{3/4}} \cdot \tau_4(\delta;mu) \cdot
    \begin{cases}
        -1 & \text{if} \quad m \equiv u \equiv 1+\lambda^3 \pmod{4}, \\
        1 &
        \text{otherwise},
    \end{cases}
\end{equation}
and $\tau_4(m \delta;u) =0$ unless
\begin{equation} \label{linrel2} 
    g_2 \Big(\frac{m \delta}{m_1},m_1 \Big)=N(m_1)^{1/2} \qquad \text{ for all } m_1 \equiv 1 \pmod{\lambda^3} \text{ such that } m_1 \mid m.
\end{equation}


\section{Applications of Rankin--Selberg regularization}

Recall that $\{\kappa_1:=\infty, \kappa_2, \dots,\kappa_h\}$ is the set of all cusps of $\Gamma(\lambda^4)$, and $\{\eta_1:=\infty, \eta_2, \dots, \eta_m \}$ is the subset of all $m=24$ cusps of $\Gamma(\lambda^4)$ that are essential with respect to $\chi_4$. We assume that the cusps are enumerated such that $\kappa_j = \eta_j$ for each integer $1 \leq j \leq m$.

\subsection{Products of metaplectic Eisenstein series}\label{sec:products_Eis}

For $s \in \mathbb{C}$ with $\Re(s)>1$ and $s \neq \frac{5}{4}$, an integer $1 \leq i \leq m$, and $w \in \mathbb{H}^3$, let 
\begin{equation} \label{Fidef}
F_{i}(w,s):= |E_{i}(w,s,\chi_4)|^2.
\end{equation} 
These are continuous $\Gamma(\lambda^4)$-invariant functions on $\mathbb{H}^3$.

By \cref{rmk:non_essential_cusps}, if $\Re(s)>2$, then the functions $E_i(\sigma_\kappa w, s, \chi_4)$ have exponential decay in $v$ at cusps $\kappa$ of $\Gamma(\lambda^4)$ that are not essential with respect to $\chi_4$. This decay also holds for every $s\in \C$ away from the poles of $E_i(\sigma_\kappa w, s, \chi_4)$. Indeed, this follows from \eqref{Efunceq} (see also \eqref{eq:Eisenstein_functional_equation}), estimates for the scattering matrix that follow from \cref{scatter},
the fact that $E_i(w,s,\chi_4)$ is a finite-order meromorphic function \cite[Theorem~3.7.1(4)]{FriedJS},
 and the Phragm{\'e}n--Lindel{\"o}f principle. Therefore for each $\kappa \in \{\kappa_1,\ldots,\kappa_h\} \setminus \{\eta_1,\ldots,\eta_m\}$ with scaling matrix $\sigma_\kappa$ and all $N \in \N$ we have 
\begin{equation}
    F_i(\sigma_\kappa w, s) = O_{N, s}(v^{-N}) \quad \text{as} \quad v \to \infty.
\end{equation}
In fact, if $10 \geq \Re(s) \geq 1+\varepsilon$ and $|s - \frac{5}{4}| \geq \varepsilon>0$, then the procedure above implies\footnote{Alternatively, this also follows from an argument analogous to \cref{lemma:truncated_Eis_decay} for Eisenstein series with non-trivial character.} 

\begin{equation}\label{eq:non_essential_RS_bound}
    \int_{\C/\Lambda} |{F_i(\sigma_\kappa (z, v), s)}| \, dx \, dy \ll_{\varepsilon} |s|^{1000} v^{-100} \quad \text{ for }\quad v \geq \varepsilon.
\end{equation}

In preparation to apply \cref{zagreg}, for each integer $1 \leq j \leq m$ we denote the Fourier expansion
\begin{equation}\label{eq:F_i_Fourier}
    F_i(\sigma_{j}w, s) = \sum_{\delta \in \Lambda^{*}} a^{(i)}_{j}(\delta,v, s) \check{e}(\delta z)
\end{equation}
at the essential (with respect to $\chi_4$) cusp $\eta_j$. Using \eqref{eisfourier} we obtain
\begin{equation} \label{constantFexact}
    a_{j}^{(i)}(\theta,v,s) = \sum_{\delta \in \Lambda^{*} } \rho_{i j}(\theta + \delta,v,s,\chi_4) \cdot \overline{\rho_{i j}(\delta,v,s,\chi_4)}
\end{equation}
for $\theta \in \Lambda^*$. If $0 \neq \delta \in \Lambda^*$, note from \eqref{Kintdef} and \eqref{rhononzero} that 
\begin{equation} \label{rhononzerounpack}
    \rho_{i j}(\delta,v,s,\chi_4) = \frac{(2\pi)^s}{\Gamma(s)} v |{\delta}|^{s-1}  K_{s-1}(4 \pi|{\delta}| v) \varphi_{i j}(\delta,s,\chi_4).
\end{equation}

By \cite[Proposition 9]{HM06}, for any $\varepsilon, \sigma > 0$ and $|{\Re(\nu)}| \leq \sigma$ we have the uniform decay
\begin{equation}\label{eq:Bessel_general_bound}
    K_{\nu}(x) \ll_{\sigma, \varepsilon} 
    \begin{cases}
        \big( \frac{1 + |{\Im(\nu)}|}{x} \big)^{\sigma + \varepsilon} e^{-\frac{\pi |{\Im(\nu)}|}{2}} & \quad \text{ if } 0 < x \leq 1 + \frac{\pi |{\Im(\nu)}|}{2}, \\
        x^{-1/2} e^{-x} & \quad \text{ if } 1 + \frac{\pi |{\Im(\nu)}|}{2} < x. 
    \end{cases}
\end{equation}
From \eqref{varphigenell} and \eqref{convexbd}, the estimate above for the Bessel function, and the expansions \eqref{eq:F_i_Fourier} and \eqref{constantFexact}, we deduce for any $N \in \mathbb{N}$ that
\begin{equation} \label{decay}
    F_i(\sigma_{j}w, s) = |\rho_{i j}(0,v,s,\chi_4)|^2 + O_{N,s}(v^{-N}) \quad \text{as} \quad v \to \infty.
\end{equation}
For future reference, note that if $10 \geq \Re(s) \geq 1+\varepsilon$ and $|s - \frac{5}{4}| \geq \varepsilon>0$, then being more precise in \eqref{decay} yields the uniform bound
\begin{equation}\label{eq:essential_RS_bound}
    F_i(\sigma_{j}w, s) = |\rho_{i j}(0,v,s,\chi_4)|^2 + O_{\varepsilon}\big(|{s}|^{1000} v^{-100} \big) \quad \text{for} \quad v \geq \varepsilon.
\end{equation}

Recall that \eqref{rhozero} gives
\begin{equation} \label{rhotildezero2}
    \rho_{i j}(0,v,s,\chi_4)=\mathbf{1}_{i j} v^{s}  + \frac{\pi}{s-1} v^{2-s} \varphi_{i j}(0,s,\chi_4).
\end{equation}
By \eqref{decay} and \eqref{rhotildezero2}, for any $N \in \mathbb{N}$ we get
\begin{equation} \label{constantFasymp}
    F_i(\sigma_{j}w, s) = \varphi_{j}^{(i)}(v,s) + O_{N,s}(v^{-N}) \quad \text{as} \quad v \to \infty,
\end{equation}
where
\begin{equation} \label{varphiFdefine}
    \varphi_{j}^{(i)} (v,s) = |{\rho_{i j}(0,v,s,\chi_4)}|^2=: \sum_{k=1}^{4} c^{(i)}_{jk}(s) \cdot v^{\alpha^{(i)}_{jk}(s)}
\end{equation}
for
\begin{equation}
    \begin{aligned} \label{c23}
        c^{(i)}_{j1}(s)&=\mathbf{1}_{i j},  \\
        c^{(i)}_{j2}(s)&= \overline{c^{(i)}_{j3}(s)}= \mathbf{1}_{i j} \cdot \frac{\pi}{\overline{s}-1} \cdot 
        \overline{\varphi_{i j}(0,s,\chi_4)},  \\
        c^{(i)}_{j4}(s)&= \frac{\pi^2}{|s-1|^2} \cdot |\varphi_{i j}(0,s,\chi_4)|^2, 
    \end{aligned}
\end{equation}
and
\begin{align}\label{alphas}
    \alpha^{(i)}_{j1}(s)=s+\overline{s},\qquad
    \alpha^{(i)}_{j2}(s)=\overline{\alpha^{(i)}_{j3}(s)}=2-\overline{s}+s,  \qquad
    \alpha^{(i)}_{j4}(s)=4-s-\overline{s}.  
\end{align}

Therefore the hypotheses of \cref{zagreg} are satisfied. For integers $1 \leq i,j \leq m$, we now compute Dirichlet series expressions for the Rankin--Selberg transforms $R_j^{(i)}(s, u) := R_{j}(F_{i}(\cdot,s),u)$ defined in \eqref{RStransform1}. Note that we could also consider $R_{\kappa}^{(i)}(s, u)$ for cusps $\kappa \in \{\kappa_1,\ldots \kappa_h\} \setminus \{\eta_1,\ldots,\eta_m \}$, but this is not of direct interest to us. If $\Re(u)$ is sufficiently large in terms of $\Re(s)$, then inserting \eqref{constantFexact}, \eqref{rhononzerounpack}, \eqref{constantFasymp}, and \eqref{varphiFdefine} into \eqref{RStransform1} gives
\begin{equation} \label{Rdevelop}
    R_{j}^{(i)}(s, u) = \sum_{0 \neq \delta \in \Lambda^{*}} |\varphi_{i j}(\delta,s,\chi_4)|^2 \cdot \int_{0}^{\infty} v^{1-s-\overline{s} +u} \cdot |K(\delta v,s)|^2 \, dv.
\end{equation}
Recall that $\Re(s)>1$ and $s \neq \frac{5}{4}$. By \eqref{varphiFdefine}, \eqref{alphas}, and \cref{absconvrem}, it follows that \eqref{Rdevelop} is valid for $\Re(u) > 2+2 \Re(s)$. Substitute \eqref{Kintdef} into \eqref{Rdevelop} to obtain
\begin{align}\label{eq:R_j_i_def}
    R_{j}^{(i)}(s, u) &= \beta(s, u) \cdot \sum_{0 \neq \delta \in \Lambda^{*} } \frac{|\varphi_{i j}(\delta,s,\chi_4)|^2}{N(\delta)^{u/2+1-\Re(s)}}
\end{align}
for $\Lambda^{*} = \lambda^{-6}\mathcal{O}_K$ and
\begin{equation} \label{betadef}
    \beta(s, u) := \frac{1}{(4 \pi)^{u}}\frac{(2 \pi)^{s+\overline{s}}}{ |\Gamma(s)|^2} \cdot \int_{0}^{\infty} v^{u-1}  |K_{s-1}(v)|^2 \, dv.
\end{equation}
The Bessel integral can be computed using \cite[(6.576.4)]{GR}, which shows that 
\begin{equation} \label{betaeval}
    \beta(s, u)= \frac{(2\pi)^{2 \Re(s)-u}}{8 \cdot |\Gamma(s)|^{2} \cdot \Gamma(u)} \cdot \Gamma \Big ( \frac{u+s+\overline{s}-2}{2} \Big) \Gamma \Big( \frac{u-s+\overline{s}}{2} \Big) \Gamma \Big( \frac{u+s-\overline{s}}{2} \Big) \Gamma \Big(\frac{u-s-\overline{s}+2}{2} \Big)
\end{equation}
for $\Re(u) > 2\left|\Re(s) - 1\right|$. 

We are now ready to analyze the poles and growth of $R_{j}^{(i)}(s,u)$.

\begin{lemma}[Rankin--Selberg transform of $|{E_i}|^2$]\label{Filem}
    Let $\varepsilon>0$, and $s \in \mathbb{C}$ satisfy $10 \geq \Re(s)\geq 1 + \varepsilon$ and $\big|s - \frac{5}{4}\big| \geq \varepsilon > 0$. For any integers $1\leq i, j\leq m$, the function $R_{j}^{(i)}(s, u)$ is holomorphic for $\Re(u)> 2 \Re(s)$. It satisfies the uniform bound
    \begin{equation}\label{eq:R_j_uniform_bound}
        R_{j}^{(i)}(s, u) \ll_\varepsilon |{s}|^{1000} \cdot |{u}|^{200} \quad \text{when} \quad 10 \geq \Re(u) \geq 2\Re(s) + \varepsilon.
    \end{equation}
\end{lemma}

\begin{remark}\label{rmk:improved_bound}
    It should be possible to improve the bound above. Observe from Stirling's formula that $\beta(s, u)$ decays exponentially in $u$ (restricted to a fixed vertical strip), hence so does $R_{j}^{(i)}(s, u)$ in the region of absolute convergence. Using \eqref{Rstarfunc}, \eqref{basicstate3}, \cite[Theorem~3.7.1(4)]{FriedJS},  and the Phragm{\'e}n--Lindel{\"o}f principle for $R_{j}^{(i)}(s, u) \cdot \beta(s, u)^{-1}$ should then be profitable. We refrain from this technical work since it will not be needed and requires further information on the scattering matrix $\boldsymbol{\Phi}(u)$ for trivial character.
\end{remark}

\begin{proof}
    We apply \eqref{possiblepoles} in \cref{zagreg} to determine the poles of $R_{j}^{(i)}(s, u)$. Recalling \eqref{alphas}, note that $\Re(\alpha_{jk}^{(i)}) \in \big\{2\Re(s), 2, 4-2\Re(s)\big\}$. Furthermore, $\boldsymbol{\Phi}(u)$ is holomorphic for $\Re(u) > 2$, since \eqref{varphidef} implies that $\varphi_{ij}(0, u) \ll \zeta_K(\Re(u)-1)$. Therefore \eqref{possiblepoles} implies that $R_{j}^{(i)}(s, u)$ is holomorphic for $\Re(u) > 2\Re(s)$. 

    The bound \eqref{eq:R_j_uniform_bound} follows from \cref{prop:RS_growth_estimate}, using \eqref{eq:non_essential_RS_bound} and \eqref{eq:essential_RS_bound} as inputs. Indeed, these allow us to take $M =|{s}|^{1000}$, while \eqref{c23} and \cref{scatter} give $\gamma \ll_\varepsilon |{s}|^{1000}$. Therefore \eqref{eq:RS_growth_estimate} implies
    \begin{equation*}
        |{R_{j}^{(i)}(s, u)}| = |{R_{j}(F_{i}(\cdot,s),u)}| \ll_\varepsilon |{s}|^{1000} \cdot |{u}|^{200}
    \end{equation*}
    for $10 \geq \Re(u) \geq 2\Re(s) + \varepsilon$, since in that range we have $u \in \Omega_\varepsilon$, as $\Phi(\cdot)$ has no poles with real part strictly 
    greater than $2$.
\end{proof}

\subsection{Products of metaplectic theta functions}

For $w \in \mathbb{H}^3$ and any integer $1\leq i\leq m$, let 
\begin{equation} \label{Hidef}
    H_{i}(w):=|\vartheta_{i}(w)|^2.
\end{equation}
This is a family of continuous $\Gamma(\lambda^4)$-invariant functions on $\mathbb{H}^3$.

After expressing the residue of $E_{i}(w,s,\chi_4)$ at $s = \frac{5}{4}$ via a contour integral, the same arguments as in \cref{sec:products_Eis} show that $H_i(\sigma_\kappa w)$ decays exponentially in $v$ at a cusp $\kappa$ of $\Gamma(\lambda^4)$ that is non-essential with respect to $\chi_4$. Therefore, if $\kappa \in \{\kappa_1,\ldots,\kappa_h\} \setminus \{\eta_1,\ldots,\eta_m\}$, then for every $N \in \N$ and $\varepsilon > 0$ we have
\begin{equation}\label{eq:non_ess_theta_RS_decay}
    H_i(\sigma_\kappa w) = O_{N, \varepsilon}(v^{-N}) \quad \text{for} \quad v \geq \varepsilon.
\end{equation}

For each integer $1\leq j\leq m$, the Fourier coefficients of $H_i(\sigma_j w)$ are obtained from \eqref{thetafourier}. The bound for $\varphi_{ij}(\delta, s, \chi_4)$ given in \cref{merovarphi} implies $\Psi_{ij}(\delta) \ll_\varepsilon N(\delta)^{1/8+\varepsilon}$ via contour integration. Thus once again, for any $N \in \N$ and $\varepsilon>0$, the decay of the Bessel function gives 
\begin{equation} \label{constantGasymp}
    H_i(\sigma_j w) = (4\pi)^2 v^{3/2} |\Psi_{i j}(0)|^2 + O_{N, \varepsilon}(v^{-N}) \quad \text{for} \quad v \geq \varepsilon,
\end{equation}
where $\Psi_{i j}(0)$ is given in \eqref{psidef}. 

Thus, for integers $1 \leq i,j \leq m$, we can apply \cref{zagreg} to obtain the meromorphic continuation to $u \in \C$ of the Rankin--Selberg transforms $R_j^{(i)}(u) := R_{j}(H_{i},u)$ defined in \eqref{RStransform1}. Note that we could also consider $R_{\kappa}^{(i)}(u)$ for cusps $\kappa \in \{\kappa_1,\ldots \kappa_h\} \setminus \{\eta_1,\ldots,\eta_m \}$, but this is not of direct interest to us. Using \eqref{thetafourier} and \eqref{constantGasymp} in \eqref{RStransform1} leads to 
\begin{equation} \label{RGdevelop}
    R_j^{(i)}(u) = \gamma_4(u) \cdot \sum_{0 \neq \delta \in \Lambda^{*}} \frac{N(\delta)^{1/4} |\Psi_{i j}(\delta)|^2}{N(\delta)^{u/2}},
\end{equation}
where $\Lambda^{*} = \lambda^{-6}\mathcal{O}_K$, and
\begin{equation} \label{gammadef}
    \gamma_4(u) :=\frac{1}{(4 \pi)^u} \frac{(2 \pi)^{5/2}}{\Gamma(5/4)^2} \cdot \int_0^{\infty} v^{u-1} K_{1/4}(v)^2 \, dv.
\end{equation}
By \cref{absconvrem} and \eqref{constantGasymp} it follows that \eqref{RGdevelop} is valid for $\Re(u)> \frac{7}{2}$.
We compute the Bessel integral via \cite[(6.576.4)]{GR} exactly as before, obtaining
\begin{align} \label{gammaeval}
    \gamma_4(u)= \frac{(2\pi)^{5/2-u} }{8\cdot \Gamma(5/4)^{2} \cdot \Gamma(u)} \cdot \Gamma \Big( \frac{2u+1}{4} \Big) \Gamma \Big( \frac{u}{2} \Big)^2 \Gamma \Big( \frac{2u-1}{4} \Big)
\end{align}
for $\Re(u)> \frac{1}{2}$. 

\begin{lemma}[Rankin--Selberg transform of $|{\vartheta_i}|^2$] \label{Gilem}
    There exists $0<\nu \leq 1/2$ such that the following holds. For any integers $1\leq i, j\leq m$, the function $R_{j}^{(i)}(u)$ given in \eqref{RGdevelop} is holomorphic for $\Re(u)>2 - \nu$ and $u \neq 2$. It has at most a simple pole at $u = 2$, with residue
    \begin{equation}\label{eq:theta_RS_residue}
        \Res_{u=2} R_{j}^{(i)}(u) = \frac{1}{\vol(\Gamma(\lambda^4) \backslash \mathbb{H}^3)} \cdot \int_{\Gamma(\lambda^4) \backslash \mathbb{H}^3} |{\vartheta_i(z)}|^2 \, d \mu(z) =: \norm{\vartheta_i}^2.
    \end{equation}
    Furthermore, for each $\varepsilon > 0$, it satisfies the bound
    \begin{equation}\label{eq:theta_RS_bound}
        R_{j}^{(i)}(u) \ll_\varepsilon |u|^{200} \qquad \text{when} \quad 10 \geq \Re(u) \geq 2-\nu + \varepsilon \quad \text{and} \quad |u-2| \geq \varepsilon.
    \end{equation}
\end{lemma}

\begin{remark}
    It should be possible to improve  \eqref{eq:theta_RS_bound}, as in \cref{rmk:improved_bound}.
\end{remark}

\begin{proof} 
    We deduce from \eqref{possiblepoles} in \cref{zagreg} that $R_j^{(i)}(u) := R_{j}(H_{i},u)$ is meromorphic in $\C$, and by \eqref{constantGasymp} its only possible poles are at $\big\{\frac{1}{2}, \frac{3}{2}\big\}$ or at poles of $\boldsymbol{\Phi}(u)$. By \cref{trivial_char_section}, the poles of $\boldsymbol{\Phi}(u)$ with $\Re(u) > 1$ are finite in number and satisfy $u \in (1, 2]$. Thus if $\nu > 0$ is sufficiently small (in particular $\nu \leq \frac{1}{2}$), the only possible pole of $R_j^{(i)}(u)$ with $\Re(u) > 2 - \nu$ is at $u=2$. By \eqref{Rstares2} (where in our case $\alpha = \frac{3}{2}$), this pole is at most simple with residue given in \eqref{eq:theta_RS_residue}. 
    
    Finally, the bound \eqref{eq:theta_RS_bound} follows directly from \cref{prop:RS_growth_estimate}. Indeed, using \eqref{eq:non_ess_theta_RS_decay} and \eqref{constantGasymp} as inputs, we have $M \ll_\varepsilon 1$ and $\gamma \ll |\Psi_{i j}(0)|^2 \ll 1$.
\end{proof}


\section{\texorpdfstring{$L^2$}{}-bounds for Fourier coefficients}

In this section we apply the previous results on Rankin--Selberg regularization to Fourier coefficients of certain quartic Eisenstein series and theta functions. This will provide the necessary average bounds for our analysis of Dirichlet series with quartic Gauss sums.

To avoid convergence issues, it will be convenient to consider $V, W : \R_{>0} \to \R$ given by
\begin{equation}\label{eq:V_W_defs}
    V(u) :=\frac{2}{\sqrt{\pi}}\int_{-\infty}^{\frac12\log(1/u)}e^{-r^2}\,dr \qquad \text{and} \qquad W(u) := - (\log{u}) \cdot e^{-(\log{u}+2)^2/4}.
\end{equation}
These are bounded smooth functions, with $V(u) \geq 1$ for $0<u\leq1$ and $V(u)>0$ for all $u>0$,
while $W(u) > 0$ for $0 < u < 1$ and $W(u) < 0$ for $u > 1$. The main point is that they have Mellin transforms
\begin{equation} \label{Vtilde}
    \widetilde{V}(z) := \int_0^\infty
     V(u) u^{z - 1} \, du = \frac{2e^{z^2}}{z} \quad \text{for }\Re(z)>0,
\end{equation}     
and 
\begin{equation} \label{Wtilde}     
\widetilde{W}(z) := \int_0^\infty W(u) u^{z - 1} \, du =2 \sqrt{\pi} (2-2z) e^{z^2-2z} \quad \text{for } z \in \mathbb{C},
\end{equation}
which both decay very rapidly in fixed vertical strips (away from the pole at $z=0$ for $\widetilde{V}(z)$).
The decay in both cases is faster than any linear exponential.

\begin{prop}[Second moment of $\psi$] \label{prop:psi_second_moment}
    Let $\varepsilon>0$ and $s\in \C$ be such that $2 + \varepsilon \geq \Re(s) \geq 1+\varepsilon$ and $\big|s - \frac{5}{4}\big| \geq \varepsilon > 0$. Then for any $A \geq 1$ and $u \equiv 1 \pmod{\lambda^3}$, we have
    \begin{equation*}
        \sum_{\substack{0 \neq a \in \mathcal{O}^*_K \\ N(a) \leq A}} |\psi(a, s; u)|^2 \ll_{\varepsilon} |{s}|^{2000}  A^{1+\varepsilon}.
    \end{equation*}
\end{prop}

\begin{proof}
    Recall that $\mathcal{O}_K^* = \lambda^{-2} \mathcal{O}_K$. For the function $V$ defined in \eqref{eq:V_W_defs}, it suffices to show that
    \begin{equation} \label{eq:psi_lindelof_replacement}
        \sum_{0 \neq a \in \mathcal{O}^*_K} |\psi(a, s; u)|^2 \cdot V \Big(\frac{N(a)}{A}\Big) \ll_{\varepsilon} |{s}|^{2000} A^{1+\varepsilon}.
    \end{equation}
    By \eqref{connect} we have $|\psi(a, s; u)|^2 = \mathcal{V}^{2}\cdot |\varphi_{i_u 1}(\lambda^{-4} a, s, \chi_4)|^2$ for the integer $1 \leq i_u \leq m$ corresponding to the cusp
    \begin{equation}\label{eq:i_u_cusp_def}
        \eta_{i_u} = \begin{cases}
            0 & \text{if } u \equiv 1 \pmod{4},\\
            \frac{\lambda^4}{1+\lambda^3} & \text{if } u \equiv 1 + \lambda^3 \pmod{4}.
        \end{cases}
    \end{equation}
    Using Mellin inversion and denoting $\delta := \lambda^{-4}a$, we have
    \begin{equation*}
        \sum_{0 \neq a \in \mathcal{O}^*_K} |\psi(a, s; u)|^2 \cdot V\Big(\frac{N(a)}{A}\Big) = \frac{\mathcal{V}^{2}}{2\pi i} \int\limits_{(2 + \varepsilon)} \widetilde{V}(z)  \Big(\frac{A}{N(\lambda^4)}\Big)^{z} \sum_{0 \neq \delta \in \lambda^{-6} \mathcal{O}_K} \frac{|\varphi_{i_u 1}(\delta, s, \chi_4)|^2}{N(\delta)^z} \, dz.
    \end{equation*}
    By \eqref{eq:R_j_i_def}, the sum on the right is equal to $R_{1}^{(i_u)}(s, z') \cdot \beta(s, z')^{-1}$ for $z' := 2z + 2\Re(s) - 2$, where we exchanged the order of summation and integration by the rapid decay of $\widetilde{V}$ and absolute convergence, since $\Re(z') > 2 + 2\Re(s)$.

    Now shift the line of integration to $\Re(z) = 1+\varepsilon$. Convergence is guaranteed by the polynomial bound \eqref{eq:R_j_uniform_bound}, which implies $R_1^{(i_u)}(s, z') \ll_{\varepsilon} |{s}|^{1000} \cdot |{z}|^{200}$ for $2 + \varepsilon \geq \Re(z) \geq 1+\varepsilon$. Indeed, it suffices to use a very coarse bound $|{\beta(s, z')}| \gg_\varepsilon |{s}|^{-1000} e^{-1000 |{z}|}$, which follows from Stirling's formula applied to \eqref{betaeval}. The very rapid decay of the Mellin transform $\widetilde{V}(z)$ in \eqref{Vtilde} then ensures absolute convergence of the integral and leads to the desired bound \eqref{eq:psi_lindelof_replacement}. 
\end{proof}

\begin{prop}[Second moment of $\tau_4$] \label{prop:tau_4_second_moment}
    For any $u \equiv 1 \pmod{\lambda^3}$ and sufficiently large $A\geq 1$ we have
    \begin{equation*}
        \sum_{\substack{0 \neq a \in \mathcal{O}^*_K \\ N(a) \leq A}} |\tau_4(a ; u)|^2 \asymp A^{3/4}.
    \end{equation*}
\end{prop}

\begin{proof}
    By \eqref{resD} we have $|{\tau_4(a; u)}|^2 = \mathcal{V}^{2}\cdot |\Psi_{i_u 1}(\lambda^{-4} a)|^2$ for the integer $1 \leq i_u \leq m$ given in \eqref{eq:i_u_cusp_def}. 
    
    Let us first prove the lower bound for the sum over $a$. Recalling that $\mathcal{O}_K^* = \lambda^{-2} \mathcal{O}_K$ and $\Lambda^* = \lambda^{-6} \mathcal{O}_K$, then for $W$ given in \eqref{eq:V_W_defs} and $\delta := \lambda^{-4} a$ we obtain
    \begin{align}
        & \sum_{\substack{0 \neq a \in \mathcal{O}^*_K \\ N(a) \leq A}} |\tau_4(a ; u)|^2 \geq \frac{1}{1000} \sum_{0 \neq a \in \mathcal{O}^*_K } |\tau_4(a ; u)|^2 \cdot W \Big( \frac{N(a)}{A} \Big) \nonumber \\
        = & \frac{1}{1000} \frac{\mathcal{V}^{2}}{2\pi i} \int\limits_{(2 + \varepsilon)} \widetilde{W}(z)  \Big(\frac{A}{N(\lambda^4)}\Big)^{z} \sum_{0 \neq \delta \in \lambda^{-6} \mathcal{O}_K} \frac{|\Psi_{i_u 1}(\delta)|^2}{N(\delta)^z} \, dz. \label{eq:lower_bd_theta_RS_interm}
    \end{align}
By \eqref{RGdevelop}, the remaining sum is $R_1^{(i_u)}(z') \cdot \gamma_4(z')^{-1}$ for $z' := 2z+\frac{1}{2}$, where exchanging the order of summation and integration is justified since we are in the region of absolute convergence $\Re(z') > \frac{7}{2}$. Shift the line of integration to $\Re(z') = 2-\frac{\nu}{2}$, where $0< \nu \leq \frac{1}{2}$ is given in \cref{Gilem}, hence $\Re(z) = \frac{3 -\nu}{4}$. By \cref{Gilem} we pick up a simple pole at $z'=2$, corresponding to $z=\frac{3}{4}$, whose contribution to \eqref{eq:lower_bd_theta_RS_interm} is
    \begin{equation}\label{eq:main_term_theta}
    \frac{1}{1000} \frac{\mathcal{V}^2 \widetilde{W}(3/4)}{N(\lambda^3)} A^{3/4} \cdot \frac{1}{2} \Res_{z'=2} R_1^{(i_u)}(z') \cdot \gamma_4(z')^{-1} = \frac{1}{1000} \frac{16 \widetilde{W}(3/4)}{\gamma_4(2)} \norm{\vartheta_{i_u}}^2 A^{3/4} \geq c A^{3/4}
    \end{equation}
    for some absolute constant $c>0$. Note that in the above display we used the fact that $\vartheta_{i_u}$ is not identically zero, since $E_{i_u}$ must have a pole at $s = \frac{5}{4}$, due to its constant terms computed in \eqref{zeroinfty}. We also used that $\widetilde{W}(3/4)>0$ and $\gamma_4(2)>0$.

Along the vertical strip relevant for the shifting procedure, we have the polynomial bound $R_1^{(i_u)}(z')  \ll_\varepsilon |{z}|^{200}$ by \eqref{eq:theta_RS_bound}, and the coarse bound $\gamma_4(z') \gg_\varepsilon e^{-1000|{z}|}$. The very rapid decay of the Mellin transform $\widetilde{W}(z)$ in \eqref{Wtilde} implies absolute convergence of the integral. We conclude that the remaining contour at $\Re(z) = \frac{3-\nu}{4}$ contributes $\ll A^{3/4 - \nu/4}$. Thus if $A$ is sufficiently large, this contribution is subsumed by the main term \eqref{eq:main_term_theta}, and the lower bound follows.

    The proof of the upper bound for the sum over $a$ is analogous. One then replaces
    the minorant $\frac{1}{1000} W$ with the majorant $V$.
\end{proof}

\begin{remark}
    With more work one can establish an asymptotic formula
    \begin{equation*}
        \sum_{\substack{0 \neq a \in \mathcal{O}^*_K \\ N(a) \leq A}} |\tau_4(a ; u)|^2  \sim \frac{64 \norm{\vartheta_{i_u}}^2}{3 \cdot \gamma_4(2)} \cdot A^{3/4} \qquad \text{as } A \to \infty.
    \end{equation*}
    This can be done either by choosing a family of majorants and minorants with the required decay of Mellin transforms, or by improving the bound in \cref{Gilem} to capture the decay of $\gamma_4(u)$ as previously alluded to (allowing the use of any Schwartz test function).
\end{remark}

We will need a slight refinement of the lower bound given in the previous lemma, where the sum is restricted to squarefree values. To avoid sieving (which would require dealing with metaplectic forms of varying level), we use a trick based on Suzuki's partial evaluation of $\tau_4$ \cite{Suz1}.

\begin{lemma}[Lower bound for $\tau_4$ along squarefree values]\label{lemma:squarefree_tau_4}
    For any sufficiently large $A\geq 1$ we have
    \begin{equation*}
        \max_{\substack{u \equiv 1 \pmod{\lambda^3} \\ \eta^4 =1 \\ 0 \leq k \leq 3}} \sum_{\substack{a \in \Z[i] \\ a \equiv 1 \pmod{\lambda^3} \\ N(a) \leq A}} \mu^2(a) \cdot |\tau_4(\eta \lambda^k a ; u)|^2 \gg A^{3/4}.
    \end{equation*}
\end{lemma}

\begin{proof}
    For any $0 \neq n \in \mathcal{O}_K^* =  \lambda^{-2}\mathcal{O}_K$, we can uniquely decompose $n = \eta \lambda^k ab^2c^3d^4$ for $\eta \in \boldsymbol{\mu}_K =\{\pm 1, \pm i \}$, $k \in \Z$ with $k \geq -2$, and $a, b, c, d\equiv 1 \pmod{\lambda^3}$ with $\mu^2(abc)=1$. For any $u \equiv 1 \pmod{\lambda^3}$, Suzuki's properties for $\tau_4$ then give
    \begin{align*}
        \tau_4(n; u) = \tau_4(\eta \lambda^k a b^2 c^3 d^4; u) \overset{\eqref{quartrel}}{=} \tau_4(\eta \lambda^k a b^2 c^3; u) \overset{\eqref{cuberel}}{=} \mathbf{1}_{c=1} \cdot \tau_4(\eta \lambda^k a b^2; u),
    \end{align*}
    and finally
    \begin{equation*}
        |\tau_4(n;u)| \overset{\eqref{squarerel}}{=} \mathbf{1}_{c=1} \cdot \Big|\frac{\overline{g_4(\eta \lambda^k a,b)}}{N(b)^{3/4}}\Big| \cdot|\tau_4(\eta \lambda^k a; bu)| = \frac{\mathbf{1}_{c=1}}{N(b)^{1/4}} \cdot |\tau_4(\eta \lambda^k a; bu)|.
    \end{equation*}
    Therefore 
    \begin{equation*}
        |\tau_4(n;1)|^2 + |\tau_4(n;1+\lambda^3)|^2 \leq \frac{\mathbf{1}_{c=1}}{\sqrt{N(b)}} \cdot \big(|\tau_4(\eta \lambda^k a; 1)|^2 + |\tau_4(\eta \lambda^k a; 1+\lambda^3)|^2\big).
    \end{equation*}
    From the equation above and \cref{prop:tau_4_second_moment}, for every sufficiently large $A\geq 1$ we obtain
    \begin{align} \label{sqintermed}
        A^{3/4} &\ll \sum_{\substack{0 \neq n \in \mathcal{O}^*_K \\ N(n) \leq A}} |\tau_4(n;1)|^2 + |\tau_4(n;1+\lambda^3)|^2 \nonumber \\
        &= \sum_{\substack{\eta \in \boldsymbol{\mu}_K \\ k \in \Z_{\geq -2}}} \sum_{\substack{a \equiv 1 \pmod{\lambda^3} \\ N(\lambda^k a) \leq A}} \mu^2(a) \big(|\tau_4(\eta \lambda^k a; 1)|^2 + |\tau_4(\eta \lambda^k a; 1+\lambda^3)|^2\big) \sum_{\substack{b, d \equiv 1 \pmod{\lambda^3} \\ N(b^2d^4) \leq \frac{A}{N(\lambda^ka)} \\ (a, b)=1}} \frac{\mu^2(b)}{\sqrt{N(b)}} \nonumber \\
        & \ll A^{1/4} \sum_{\substack{\eta \in \boldsymbol{\mu}_K \\ k \in \Z_{\geq -2}}} \sum_{\substack{a \equiv 1 \pmod{\lambda^3} \\ N(\lambda^k a) \leq A}} \mu^2(a) \cdot \frac{|\tau_4(\eta \lambda^k a; 1)|^2 + |\tau_4(\eta \lambda^k a; 1+\lambda^3)|^2}{N(\lambda^k a)^{1/4}} \cdot \log\Big(2 + \frac{A}{N(\lambda^k a)}\Big). 
    \end{align}
    
    For any parameter $C \geq 2$, the contribution of the pairs $(a, k)$ with $N(\lambda^k a) < \frac{A}{C}$ to the display above is (writing $r = \eta \lambda^k a$)
    \begin{align*}
        &\ll A^{1/4} \sum_{\substack{0 \neq r \in \mathcal{O}^*_K \\ N(r) < \frac{A}{C}}} \frac{|\tau_4(r; 1)|^2 + |\tau_4(r; 1+\lambda^3)|^2}{N(r)^{1/4}} \cdot \log\Big(2 + \frac{A}{N(r)}\Big) \\
        &\ll A^{1/4} \sum_{\substack{\ell \in \Z \\ 1 \ll R := 2^\ell < \frac{A}{C}}} \frac{1}{R^{1/4}} \log{\Big(\frac{A}{R}\Big)} \sum_{\substack{0 \neq r \in \mathcal{O}^*_K \\ R < N(r) \leq 2R}} |\tau_4(r;1)|^2 + |\tau_4(r;1+\lambda^3)|^2.
    \end{align*}
    Applying the upper bound of \cref{prop:tau_4_second_moment}, this contribution is 
    \begin{equation}\label{eq:small_a_bound}
        \ll A^{1/4} \sum_{\substack{\ell \in \Z \\ 1 \ll R := 2^\ell < \frac{A}{C}}} R^{1/2} \log{\Big(\frac{A}{R}\Big)} \ll \frac{A^{3/4}}{C^{1/2}} \log{C}.
    \end{equation}
 
   Note that for integers $k \equiv k' \pmod{4}$ we have $\tau_4(\eta \lambda^k a; u) = \tau_4(\eta \lambda^{k'} a; u)$. Therefore inserting \eqref{eq:small_a_bound} into the right side of 
    \eqref{sqintermed} leads to
    \begin{align*}
        A^{3/4} &\ll \frac{A^{3/4}}{C^{1/2}} \log{C} + C^{1/4} \log{C} \sum_{\substack{\eta \in \boldsymbol{\mu}_K \\ k \in \Z_{\geq -2}}} \sum_{\substack{a \equiv 1 \pmod{\lambda^3} \\ \frac{A}{C} \leq N(\lambda^k a) \leq A}} \mu^2(a) \cdot \big(|\tau_4(\eta \lambda^k a; 1)|^2 + |\tau_4(\eta \lambda^k a; 1+\lambda^3)|^2\big)  \\
        & \ll \frac{A^{3/4}}{C^{1/2}} \log{C} + C^{1/4} (\log{C})^2 \sum_{\substack{\eta^4 =1 \\ 0 \leq k \leq 3}} \sum_{\substack{a \equiv 1 \pmod{\lambda^3} \\ N(a) \leq 4A}} \mu^2(a) \cdot \big(|\tau_4(\eta \lambda^k a; 1)|^2 + |\tau_4(\eta \lambda^k a; 1+\lambda^3)|^2\big) .
    \end{align*}
    Choosing $C\geq 2$ to be a sufficiently large absolute constant, we can make the first term on the right side of the display above smaller than half the left side. Hence the second term on the right is $\gg A^{3/4}$, and this finishes the proof of \cref{lemma:squarefree_tau_4}.
\end{proof}


\section{Proof of \cref{thm:quartic_main}}

\begin{proof}[Proof of \cref{thm:quartic_main}]
    The duality principle for the large sieve \cite[(7.9) to (7.11)]{IK} and reciprocity for the quartic symbol imply \cite[Lemma~3.1]{BGL} that
    \begin{equation} \label{eq:duality}
        \Xi_4(A, B) \asymp \Xi_4(B, A).
    \end{equation}

We split the proof into different ranges for $A$ and $B$. Since the desired bound \eqref{eq:quartic_lower_bound} is symmetric in $A$ and $B$, by \eqref{eq:duality} it suffices to prove it for $A \geq B \geq 1$. We assume this from now on, so that \eqref{eq:quartic_lower_bound} becomes
    \begin{equation}\label{eq:desired_lower_bd}
        \Xi_4(A, B) \gg_\varepsilon A^{-\varepsilon} \big(A + A^{3/4} B^{1/2} \big).
    \end{equation}

    \subsection{Case 1: \texorpdfstring{$A \geq B^2$}{}} 

    Take $\boldsymbol{\beta}$ given by $\beta_{b} = \boldsymbol{1}_{b = 1}$ in \eqref{eq:Xi_def}. This shows that 
    \begin{equation*}
        \Xi_4(A, B) \gg A, 
    \end{equation*}
    which for $A \geq B^2$ implies the desired bound \eqref{eq:desired_lower_bd}.

 \subsection{Case 2: $B \leq A < B^2$}

    We may assume that $B$ (hence also $A$) is larger than any given constant, otherwise $1 \leq A, B \ll 1$ and \eqref{eq:desired_lower_bd} is trivial. Take $\boldsymbol{\beta}$ given by
    \begin{equation*}
        \beta_b = \overline{\widetilde{g}_4(\eta \lambda^k, b)} \cdot \mathbf{1}_{b \equiv u \pmod{4}} \cdot H\Big(\frac{N(b)}{B} \Big),
    \end{equation*}
    where $H:(0,\infty) \to \mathbb{R}_{\geq 0}$ is a smooth function with compact support in $(\frac{1}{2},1)$, which is fixed and not identically zero. Here $\eta \in \{\pm1, \pm i\}$, $0 \leq k\leq 3$, and $u \equiv 1 \pmod{\lambda^3}$ are chosen (depending on $A$) so that
    \begin{equation}\label{eq:tau_lower_assumption}
        \sum_{\substack{a \in \Z[i] \\ a \equiv 1 \pmod{\lambda^3} \\ N(a) \leq A}} \mu^2(a) \cdot |\tau_4(\eta \lambda^k a ; u)|^2 \gg A^{3/4},
    \end{equation}
    which is possible by \cref{lemma:squarefree_tau_4}, since we may assume that $A$ is sufficiently large. 
    
    It follows from $|{\widetilde{g}_4(\eta \lambda^k, b)}| = |{\widetilde{g}_4(b)}| = \mu^2(b)$ for $b \equiv 1 \pmod{\lambda^3}$ that
    \begin{equation}\label{eq:beta_L2_norm}
        \norm{\boldsymbol{\beta}}^2_2 \asymp B
    \end{equation}
    for sufficiently large $B$ (depending on $H$). Furthermore we have
    \begin{equation*}
        \Big(\frac{a}{b}\Big)_4 \cdot \overline{\widetilde{g}_4(\eta \lambda^k, b)} = \boldsymbol{1}_{(a, b) = 1} \cdot \overline{\widetilde{g}_4(\eta \lambda^k a, b)}.
    \end{equation*}
Thus the support condition on $H$ gives
    \begin{align} 
         \mathcal{X}(A, B) := & \sum_{\substack{ a \in \mathbb{Z}[i] \\ a \equiv 1 \pmod{\lambda^3} \\ N(a) \leq A  }} \mu^2(a) \cdot \Bigg | \sum_{\substack{ b \in \mathbb{Z}[i] \\ b \equiv 1 \pmod{\lambda^3} \\ N(b) \leq B }} \mu^2(b) \beta_b \Big(\frac{a}{b}\Big)_4  \Bigg|^2 \nonumber \\
         = & \sum_{\substack{ a \in \mathbb{Z}[i] \\ a \equiv 1 \pmod{\lambda^3} \\ N(a) \leq A  }} \mu^2(a) \cdot  \Bigg | \sum_{\substack{ b \in \mathbb{Z}[i] \\ b \equiv u \pmod{4} \\ (a, b) = 1}} \widetilde{g}_4(\eta \lambda^k a, b) H\Big(\frac{N(b)}{B}\Big) \Bigg |^2. \label{counter1}
    \end{align}

In preparation to evaluate the sum over $b$, let us introduce some notation from \cite[Section 6]{CDD26}. For any $\alpha \in \Z[i]$ with $\alpha\equiv 1 \pmod{\lambda^3}$ and $\mu^2(\alpha)=1$, $r \in \lambda^{-2} \mathbb{Z}[i]$, and $v \equiv 1 \pmod{\lambda^3}$, denote
    \begin{equation*}
        \psi_\alpha(r,s; v) := \sum_{\substack{c \in \Z[i] \\ c \equiv v \pmod{4} \\ (c, \alpha)=1}}  \frac{g_4(r, c)}{N(c)^s},
    \end{equation*}
    which converges absolutely for $\Re(s)>\frac{3}{2}$. Also denote
    \begin{equation} \label{deltastardef}
        \Delta_{\alpha}^{*}(r, s) := \prod_{\substack{\pi \text{ prime} \\ \pi \equiv 1 \pmod{\lambda^3} \\ \pi \mid \alpha}}  \Big(1+ g_2\Big(\frac{r}{\pi},\pi\Big) N(\pi)^{1-2s} \Big).
    \end{equation}
    By convention we set $g_2(\frac{r}{\pi}, \pi) = 0$ if $\pi \nmid r$.

    After Mellin inversion of the Schwartz function $H$, the sum over $b$ in \eqref{counter1} is equal to
    \begin{equation} \label{counter2}
        \frac{1}{2 \pi i} \int_{(2)} \widetilde{H}(s) B^s \psi_a\big(\eta \lambda^k a, s + \tfrac{1}{2}; u\big) \, ds.
    \end{equation}
    Now we use the crucial fact that \cite[Lemma 6.1 (iii)]{CDD26} implies, for $a \equiv 1 \pmod{\lambda^3}$ with $\mu^2(a)=1$, that
    \begin{equation*}
        \psi_{a}(\eta \lambda^k a, s; u) \cdot \Delta^{*}_{a}(\eta \lambda^k a, s) = \psi(\eta \lambda^k a, s; u).
    \end{equation*}
    In particular, by \eqref{connect} and \cref{merovarphi} we obtain the meromorphic continuation of $\psi_{a}(\eta \lambda^k a, s; u)$ to $s \in \C$, and polynomial bounds in vertical strips as in \eqref{convexbd}.
    
    Shifting the contour in \eqref{counter2} to $\Re(s)=\frac{1}{2}+\varepsilon$, we encounter (at most) a simple pole at $s=\frac{3}{4}$. Again by \eqref{resD} we deduce that \eqref{counter2} is equal to a term
    \begin{align}  \label{bias}
        \mathcal{P}(\eta \lambda^k a; u) := \widetilde{H}(\tfrac{3}{4}) \cdot B^{3/4} \cdot \Delta^{*}_{a}(\eta \lambda^k a, \tfrac{5}{4})^{-1} \cdot \tau_4(\eta \lambda^k a; u),
    \end{align}
    coming from the pole at $s = \frac{3}{4}$, plus the remaining integral term
    \begin{equation}\label{eq:integral_term}
        \mathcal{I}(\eta \lambda^k a; u) := \frac{1}{2 \pi i} \int_{(\frac{1}{2}+\varepsilon)} \widetilde{H}(s) \cdot B^s \cdot \Delta^{*}_{a}(\eta \lambda^k a, s + \tfrac{1}{2})^{-1} \cdot \psi(\eta \lambda^k a, s+\tfrac{1}{2}; u) \, ds.
    \end{equation}
    Since $|{g_2(\frac{\eta \lambda^k a}{\pi}, \pi)}| \leq N(\pi)^{1/2}$, for $\Re(s) \geq \frac{1}{2}$ we have from \eqref{deltastardef} that
    \begin{equation*}
        N(a)^{-\varepsilon} \ll_\varepsilon |{\Delta^{*}_{a}(\eta \lambda^k a, s + \tfrac{1}{2})}| \ll_\varepsilon N(a)^{\varepsilon}.
    \end{equation*}
    Thus 
    \begin{equation}\label{eq:P_lower_bd}
        |{\mathcal{P}(\eta \lambda^k a; u)}|^2 \gg_{\varepsilon} A^{-2\varepsilon} B^{3/2} \cdot |{\tau_4(\eta \lambda^k a; u)}|^2,
    \end{equation}
    and by Cauchy--Schwarz also
    \begin{equation}\label{eq:I_upper_bd}
        |{\mathcal{I}(\eta \lambda^k a; u)}|^2 \ll_{\varepsilon} A^{2\varepsilon} B^{1+2\varepsilon} \int_{(\frac{1}{2}+\varepsilon)} |{\widetilde{H}(s)}| \cdot |{\psi(\eta \lambda^k a, s+\tfrac{1}{2}; u)}|^2 \, |{ds}|.
    \end{equation}

It follows from the AM-GM inequality that $|{x+y}|^2 \geq \frac{1}{2}|{x}|^2 - |{y}|^2$, and so we conclude from \eqref{counter1} that
    \begin{equation} \label{Xexpression}
        \mathcal{X}(A, B) = \sum_{\substack{ a \in \mathbb{Z}[i] \\ a \equiv 1 \pmod{\lambda^3} \\ N(a) \leq A}} \mu^2(a) \cdot \big| \mathcal{P}(\eta \lambda^k a; u) + \mathcal{I}(\eta \lambda^k a; u) \big|^2 \geq \frac{\mathcal{M}(A, B)}{2} - \mathcal{E}(A, B),
    \end{equation}
    where
    \begin{align} \label{Mdef}
        \mathcal{M}(A, B):= \sum_{\substack{ a \in \mathbb{Z}[i] \\ a \equiv 1 \pmod{\lambda^3} \\ N(a) \leq A  }} \mu^2(a) \cdot |{\mathcal{P}(\eta \lambda^k a; u)}|^2
    \end{align}
    and
    \begin{equation} \label{Edef}
        \mathcal{E}(A, B) := \sum_{\substack{ a \in \mathbb{Z}[i] \\ a \equiv 1 \pmod{\lambda^3} \\ N(a) \leq A  }} \mu^2(a) \cdot |{\mathcal{I}(\eta \lambda^k a; u)}|^2.
    \end{equation}
   
 By \eqref{eq:P_lower_bd} and \eqref{eq:tau_lower_assumption}, the main term satisfies
    \begin{equation}\label{eq:M_lower_bd}
        \mathcal{M}(A, B) \gg_\varepsilon A^{-\varepsilon} B^{3/2} \sum_{\substack{ a \in \mathbb{Z}[i] \\ a \equiv 1 \pmod{\lambda^3} \\ N(a) \leq A  }} \mu^2(a) \cdot |{\tau_4(\eta \lambda^k a; u)}|^2 \gg (AB)^{-\varepsilon} A^{3/4} B^{3/2}.
    \end{equation}

    Finally, using \eqref{eq:I_upper_bd} and \cref{prop:psi_second_moment} we see that the error term satisfies
    \begin{align}
        \mathcal{E}(A, B) &\ll_{\varepsilon} (AB)^{\varepsilon} B \int_{(\frac{1}{2}+\varepsilon)} |{\widetilde{H}(s)}| \cdot \sum_{\substack{0 \neq a' \in \mathcal{O}^*_K \\ N(a') \leq 8A}} |{\psi(a', s+\tfrac{1}{2}; u)}|^2 \, |{ds}| \nonumber \\
        & \ll_{\varepsilon} (AB)^{1 + \varepsilon} \int_{(\frac{1}{2}+\varepsilon)} |{\widetilde{H}(s)}| \cdot |{s}|^{2000} \, |{ds}| \ll (AB)^{1 + \varepsilon}. \label{eq:E_upper_bd}
    \end{align}
    
We now split the proof into two further cases. First, assume that $\mathcal{M}(A, B) \geq 4\cdot  \mathcal{E}(A, B)$. Then \eqref{Xexpression} and \eqref{eq:M_lower_bd} give
that
    \begin{equation*}
        \mathcal{X}(A, B) \geq \frac{\mathcal{M}(A, B)}{2} - \mathcal{E}(A, B) \geq \frac{\mathcal{M}(A, B)}{4} \gg_\varepsilon (AB)^{-\varepsilon} A^{3/4} B^{3/2}.
    \end{equation*}
    Using \eqref{eq:beta_L2_norm} and $A < B^2$, this implies \eqref{eq:desired_lower_bd} and finishes the proof.
    
Otherwise we have $\mathcal{M}(A, B) < 4 \cdot \mathcal{E}(A, B)$, which by \eqref{eq:M_lower_bd} and \eqref{eq:E_upper_bd} give
that
    \begin{equation}\label{eq:AB_constraint}
        (AB)^{-\varepsilon} A^{3/4} B^{3/2} \ll_\varepsilon \mathcal{M}(A, B) \ll \mathcal{E}(A, B) \ll_\varepsilon (AB)^{1+\varepsilon}.
    \end{equation}
    Thus
    \begin{equation*}
         B^{1/2 - \varepsilon} \ll_\varepsilon A^{1/4+\varepsilon}.
    \end{equation*}
In that case the trivial lower bound $\Xi_4(A, B) \gg A$, (arguing as in Case 1), implies the desired estimate \eqref{eq:desired_lower_bd}. This finishes the proof of \cref{thm:quartic_main}.
\end{proof}


\bibliographystyle{amsalpha}
\bibliography{RS} 

@article {Zag,
    AUTHOR = {Zagier, D.~},
     TITLE = {The {R}ankin-{S}elberg method for automorphic functions which
              are not of rapid decay},
   JOURNAL = {J. Fac. Sci. Univ. Tokyo Sect. IA Math.},
  FJOURNAL = {Journal of the Faculty of Science. University of Tokyo.
              Section IA. Mathematics},
    VOLUME = {28},
      YEAR = {1981},
    NUMBER = {3},
     PAGES = {415--437},
      ISSN = {0040-8980},
   MRCLASS = {10D20 (10D24)},
  MRNUMBER = {656029},
MRREVIEWER = {R. A. Rankin},
}

@article {KP,
    AUTHOR = {Kazhdan, D. A. and Patterson, S. J.},
     TITLE = {Metaplectic forms},
   JOURNAL = {Inst. Hautes \'{E}tudes Sci. Publ. Math.},
  FJOURNAL = {Institut des Hautes \'{E}tudes Scientifiques. Publications
              Math\'{e}matiques},
    NUMBER = {59},
      YEAR = {1984},
     PAGES = {35--142},
      ISSN = {0073-8301},
   MRCLASS = {22E55 (11F72 11R39 22E50)},
  MRNUMBER = {743816},
MRREVIEWER = {Jean-Loup Waldspurger},
       URL = {http://www.numdam.org/item?id=PMIHES_1984__59__35_0},
}

@book {FriedJS,
    AUTHOR = {Friedman, J.~S.~},
     TITLE = {The {S}elberg trace formula and {S}elberg zeta-function for
              cofinite {K}leinian groups with finite dimensional unitary
              representations},
      NOTE = {Thesis (Ph.D.)--State University of New York at Stony Brook},
 PUBLISHER = {ProQuest LLC, Ann Arbor, MI},
      YEAR = {2005},
     PAGES = {75},
      ISBN = {978-0542-20066-3},
   MRCLASS = {99-05},
  MRNUMBER = {2707460},
       URL =
              {http://gateway.proquest.com/openurl?url_ver=Z39.88-2004&rft_val_fmt=info:ofi/fmt:kev:mtx:dissertation&res_dat=xri:pqdiss&rft_dat=xri:pqdiss:3179592},
}

@incollection {Deli,
    AUTHOR = {Deligne, P.},
     TITLE = {Sommes de {G}auss cubiques et rev\^etements de {${\rm SL}(2)$}\ [d'apr\`es {S}. {J}. {P}atterson]},
 BOOKTITLE = {S\'eminaire {B}ourbaki (1978/79)},
    SERIES = {Lecture Notes in Math.},
    VOLUME = {770, Exp. No. 539},
     PAGES = {244--277},
 PUBLISHER = {Springer, Berlin},
      YEAR = {1980}
}

@incollection {BH,
    AUTHOR = {Bump, D.~ and Hoffstein, J.~},
     TITLE = {Some conjectured relationships between theta functions and
              {E}isenstein series on the metaplectic group},
 BOOKTITLE = {Number theory ({N}ew {Y}ork, 1985/1988)},
    SERIES = {Lecture Notes in Math.},
    VOLUME = {1383},
     PAGES = {1--11},
 PUBLISHER = {Springer, Berlin},
      YEAR = {1989},
   MRCLASS = {11F66 (11F30 11F55 11F70)},
  MRNUMBER = {1023915},
MRREVIEWER = {Minking Eie},
       DOI = {10.1007/BFb0083566},
       URL = {https://doi.org/10.1007/BFb0083566},
}

@article {Liu,
    AUTHOR = {Liu, Z.~},
     TITLE = {Explicit quadratic large sieve inequality},
   JOURNAL = {Acta Arith.},
  FJOURNAL = {Acta Arithmetica},
    VOLUME = {223},
      YEAR = {2026},
    NUMBER = {3},
     PAGES = {227--252},
      ISSN = {0065-1036,1730-6264},
   MRCLASS = {11L40 (11L26 11N35 11N36)},
  MRNUMBER = {5077440},
       DOI = {10.4064/aa250722-27-1},
       URL = {https://doi.org/10.4064/aa250722-27-1},
}

@incollection {Pat2,
    AUTHOR = {Patterson, S. J.},
     TITLE = {Whittaker models of generalized theta series},
 BOOKTITLE = {S{\'e}min. {Th{\'e}or}. {Nombres}, {Paris} 1982--83, {Prog}. {Math}. 51},
    SERIES = {Progr. Math.},
    VOLUME = {51},
     PAGES = {199--232},
 PUBLISHER = {Birkh\"{a}user Boston, Boston, MA},
      YEAR = {1984},
   MRCLASS = {11F70 (11F27 22E55)},
  MRNUMBER = {791596},
}

@article {EckPat,
    AUTHOR = {Eckhardt, C. and Patterson, S. J.},
     TITLE = {On the {F}ourier coefficients of biquadratic theta series},
   JOURNAL = {Proc. London Math. Soc. (3)},
  FJOURNAL = {Proceedings of the London Mathematical Society. Third Series},
    VOLUME = {64},
      YEAR = {1992},
    NUMBER = {2},
     PAGES = {225--264},
      ISSN = {0024-6115},
   MRCLASS = {11F30 (11F27)},
  MRNUMBER = {1143226},
MRREVIEWER = {Solomon Friedberg},
       DOI = {10.1112/plms/s3-64.2.225},
       URL = {https://doi.org/10.1112/plms/s3-64.2.225},
}

@article {DR,
    AUTHOR = {Dunn, A.~ and Radziwi\l\l, M.~},
     TITLE = {Bias in cubic {G}auss sums: {P}atterson's conjecture},
   JOURNAL = {Ann. of Math. (2)},
  FJOURNAL = {Annals of Mathematics. Second Series},
    VOLUME = {200},
      YEAR = {2024},
    NUMBER = {3},
     PAGES = {967--1057},
      ISSN = {0003-486X,1939-8980},
   MRCLASS = {11L20 (11F27 11F30 11L05 11L15 11N36)},
  MRNUMBER = {4816436},
MRREVIEWER = {Olivier\ Bordell\`es},
       DOI = {10.4007/annals.2024.200.3.3},
       URL = {https://doi.org/10.4007/annals.2024.200.3.3},
}

@article{DDHL,
      title={Quartic {G}auss sums over primes and metaplectic theta functions}, 
      author={David, C.~ and Dunn, A.~ and  Hamieh, H.~ and Lin, H.~},
        JOURNAL = {to appear in Algebra \& Number Theory},
      year={2025},
      note = {\href{https://arxiv.org/abs/2306.11875}{\texttt{arXiv:2306.11875}}}
}

@article {Suz1,
    AUTHOR = {Suzuki, T.~},
     TITLE = {Some results on the coefficients of the biquadratic theta
              series},
   JOURNAL = {J. Reine Angew. Math.},
  FJOURNAL = {Journal f\"{u}r die Reine und Angewandte Mathematik. [Crelle's
              Journal]},
    VOLUME = {340},
      YEAR = {1983},
     PAGES = {70--117},
      ISSN = {0075-4102},
   MRCLASS = {10D24 (10A15 10D12)},
  MRNUMBER = {691962},
MRREVIEWER = {A. I. Vinogradov},
       DOI = {10.1515/crll.1983.340.70},
       URL = {https://doi.org/10.1515/crll.1983.340.70},
}

@article {FHL,
    AUTHOR = {Friedberg, S.~ and Hoffstein, J.~ and Lieman, D.~},
     TITLE = {Double {D}irichlet series and the {$n$}-th order twists of
              {H}ecke {$L$}-series},
   JOURNAL = {Math. Ann.},
  FJOURNAL = {Mathematische Annalen},
    VOLUME = {327},
      YEAR = {2003},
    NUMBER = {2},
     PAGES = {315--338},
      ISSN = {0025-5831,1432-1807},
   MRCLASS = {11R42 (11F66 11M41 11R47)},
  MRNUMBER = {2015073},
MRREVIEWER = {M.\ Ram\ Murty},
       DOI = {10.1007/s00208-003-0455-4},
       URL = {https://doi.org/10.1007/s00208-003-0455-4},
}

@book {IK,
    AUTHOR = {Iwaniec, H.~ and Kowalski, E.~},
     TITLE = {Analytic number theory},
    SERIES = {American Mathematical Society Colloquium Publications},
    VOLUME = {53},
 PUBLISHER = {American Mathematical Society, Providence, RI},
      YEAR = {2004},
     PAGES = {xii+615},
      ISBN = {0-8218-3633-1},
   MRCLASS = {11-02 (11Fxx 11Lxx 11Mxx 11Nxx)},
  MRNUMBER = {2061214},
MRREVIEWER = {K.\ Soundararajan},
       DOI = {10.1090/coll/053},
       URL = {https://doi.org/10.1090/coll/053},
}

@book {GR,
    AUTHOR = {Gradshteyn, I. S. and Ryzhik, I. M.},
     TITLE = {Table of integrals, series, and products},
   EDITION = {8th edition},
      NOTE = {Translated from the Russian.
              Translation edited and with a preface by V. Moll and D. Zwillinger},
 PUBLISHER = {Elsevier/Academic Press, Amsterdam},
      YEAR = {2015},
     PAGES = {xlvi+1133},
      ISBN = {978-0-12-384933-5},
   MRCLASS = {00A22 (33-00)},
  MRNUMBER = {3307944},
}

@misc{SarCoh,
  author       = {Cohen, P.~ and Sarnak, P.~},
  title        = {Notes on the {T}race {F}ormula, {C}hapter 6: Eisenstein series for hyperbolic manifolds},
  year         = {1980},
  note         = {Available at: \url{https://publications.ias.edu/sarnak}},
}

@misc{SarCoh_chapt7,
  author       = {Cohen, P.~ and Sarnak, P.~},
  title        = {Notes on the {T}race {F}ormula, {C}hapter 7: continuous spectrum and the trace formula in the co-finite case},
  year         = {1980},
  note         = {Available at: \url{https://publications.ias.edu/sarnak}},
}

@article {Dia,
    AUTHOR = {Diaconu, A.~},
     TITLE = {Mean square values of {H}ecke {$L$}-series formed with
              {$r$}-th order characters},
   JOURNAL = {Invent. Math.},
  FJOURNAL = {Inventiones Mathematicae},
    VOLUME = {157},
      YEAR = {2004},
    NUMBER = {3},
     PAGES = {635--684},
      ISSN = {0020-9910,1432-1297},
   MRCLASS = {11F66 (11F67 11R42)},
  MRNUMBER = {2092772},
MRREVIEWER = {Emmanuel\ P.\ Royer},
       DOI = {10.1007/s00222-004-0363-6},
       URL = {https://doi.org/10.1007/s00222-004-0363-6},
}

@book {Kub,
    AUTHOR = {Kubota, T.~},
     TITLE = {On automorphic functions and the reciprocity law in a number
              field},
    SERIES = {Lectures in Mathematics, Department of Mathematics, Kyoto
              University},
    VOLUME = {2},
 PUBLISHER = {Kinokuniya Book Store, Tokyo},
      YEAR = {1969},
     PAGES = {iii+65},
   MRCLASS = {10.20 (20.00)},
  MRNUMBER = {255490},
MRREVIEWER = {F.\ van der Blij},
}

@inproceedings {Sel,
    AUTHOR = {Selberg, A.~},
     TITLE = {Discontinuous groups and harmonic analysis},
 BOOKTITLE = {Proc. {I}nternat. {C}ongr. {M}athematicians ({S}tockholm,
              1962)},
     PAGES = {177--189},
 PUBLISHER = {Inst. Mittag-Leffler, Djursholm},
      YEAR = {1963},
   MRCLASS = {32.65 (20.65)},
  MRNUMBER = {0176097},
}

@article {BGL,
    AUTHOR = {Blomer, V.~ and Goldmakher, L.~ and Louvel, B.~},
     TITLE = {{$L$}-functions with {$n$}-th-order twists},
   JOURNAL = {Int. Math. Res. Not. IMRN},
  FJOURNAL = {International Mathematics Research Notices. IMRN},
      YEAR = {2014},
    NUMBER = {7},
     PAGES = {1925--1955},
      ISSN = {1073-7928,1687-0247},
   MRCLASS = {11M41 (11F66 11R42)},
  MRNUMBER = {3190355},
MRREVIEWER = {Timothy\ Lee\ Gillespie},
       DOI = {10.1093/imrn/rns257},
       URL = {https://doi.org/10.1093/imrn/rns257},
}

@article {Pat3,
    AUTHOR = {Patterson, S. J.},
     TITLE = {On the distribution of {K}ummer sums},
   JOURNAL = {J. Reine Angew. Math.},
  FJOURNAL = {Journal f\"ur die Reine und Angewandte Mathematik. [Crelle's
              Journal]},
    VOLUME = {303/304},
      YEAR = {1978},
     PAGES = {126--143},
      ISSN = {0075-4102,1435-5345},
   MRCLASS = {10G10},
  MRNUMBER = {514676},
MRREVIEWER = {Bruce\ C.\ Berndt},
       DOI = {10.1515/crll.1978.303-304.126},
       URL = {https://doi.org/10.1515/crll.1978.303-304.126},
}

@article {Miz,
    AUTHOR = {Mizuno, Y.},
     TITLE = {The {R}ankin-{S}elberg convolution for {C}ohen's {E}isenstein
              series of half integral weight},
   JOURNAL = {Abh. Math. Sem. Univ. Hamburg},
  FJOURNAL = {Abhandlungen aus dem Mathematischen Seminar der Universit\"{a}t
              Hamburg},
    VOLUME = {75},
      YEAR = {2005},
     PAGES = {1--20},
      ISSN = {0025-5858},
   MRCLASS = {11F37 (11F12 11F66)},
  MRNUMBER = {2187576},
MRREVIEWER = {M. Manickam},
       DOI = {10.1007/BF02942033},
       URL = {https://doi.org/10.1007/BF02942033},
}

@article {DG,
    AUTHOR = {Dutta Gupta, S.~},
     TITLE = {On the {R}ankin-{S}elberg method for functions not of rapid
              decay on congruence subgroups},
   JOURNAL = {J. Number Theory},
  FJOURNAL = {Journal of Number Theory},
    VOLUME = {62},
      YEAR = {1997},
    NUMBER = {1},
     PAGES = {115--126},
      ISSN = {0022-314X},
   MRCLASS = {11F66 (11F12)},
  MRNUMBER = {1430005},
MRREVIEWER = {Solomon Friedberg},
       DOI = {10.1006/jnth.1997.2035},
       URL = {https://doi.org/10.1006/jnth.1997.2035},
}

@book {H,
    AUTHOR = {Hasse, H.~},
     TITLE = {Vorlesungen \"{u}ber {Z}ahlentheorie},
    SERIES = {Die Grundlehren der mathematischen Wissenschaften in
              Einzeldarstellungen mit besonderer Ber\"{u}cksightigung der
              Anwendungsgebiete. Band LIX},
 PUBLISHER = {Springer-Verlag, Berlin-G\"{o}ttingen-Heidelberg},
      YEAR = {1950},
     PAGES = {xii+474},
   MRCLASS = {10.0X},
  MRNUMBER = {0051844},
MRREVIEWER = {P. T. Bateman},
}

@book {EGM,
    AUTHOR = {Elstrodt, J. and Grunewald, F. and Mennicke, J.},
     TITLE = {Groups acting on hyperbolic space. {Harmonic} analysis and number theory},
    SERIES = {Springer Monographs in Mathematics},
 PUBLISHER = {Springer-Verlag, Berlin},
      YEAR = {1998},
     PAGES = {xvi+524},
      ISBN = {3-540-62745-6},
   MRCLASS = {11F72 (11F12 11M36 22E40 57M50 57S30)},
  MRNUMBER = {1483315},
MRREVIEWER = {Stefan K\"{u}hnlein},
       DOI = {10.1007/978-3-662-03626-6},
       URL = {https://doi-org.proxy2.library.illinois.edu/10.1007/978-3-662-03626-6},
}

@article {HB,
    AUTHOR = {Heath-Brown, D. R.},
     TITLE = {Kummer's conjecture for cubic {G}auss sums},
   JOURNAL = {Israel J. Math.},
  FJOURNAL = {Israel Journal of Mathematics},
    VOLUME = {120},
      YEAR = {2000},
     PAGES = {97--124},
      ISSN = {0021-2172,1565-8511},
   MRCLASS = {11L05 (11L40)},
  MRNUMBER = {1815372},
MRREVIEWER = {Matti\ Jutila},
       DOI = {10.1007/s11856-000-1273-y},
       URL = {https://doi.org/10.1007/s11856-000-1273-y},
}

@article {Pat1,
    AUTHOR = {Patterson, S. J.},
     TITLE = {A cubic analogue of the theta series},
   JOURNAL = {J. Reine Angew. Math.},
  FJOURNAL = {Journal f\"{u}r die Reine und Angewandte Mathematik. [Crelle's
              Journal]},
    VOLUME = {296},
      YEAR = {1977},
     PAGES = {125--161},
      ISSN = {0075-4102},
   MRCLASS = {10D20},
  MRNUMBER = {563068},
       DOI = {10.1515/crll.1977.296.125},
       URL = {https://doi-org.proxy2.library.illinois.edu/10.1515/crll.1977.296.125},
}

@misc{DLMF,
         key = "{\relax DLMF}",
       title = "{\it NIST Digital Library of Mathematical Functions}",
howpublished = "\url{https://dlmf.nist.gov/}, Release 1.2.5 of 2025-12-15",
         url = "https://dlmf.nist.gov/",
        note = "F.~W.~J. Olver, A.~B. {Olde Daalhuis}, D.~W. Lozier, B.~I. Schneider,
                R.~F. Boisvert, C.~W. Clark, B.~R. Miller, B.~V. Saunders,
                H.~S. Cohl, and M.~A. McClain, eds."
}

@article {HM06,
    AUTHOR = {Harcos, G. and Michel, P.},
     TITLE = {The subconvexity problem for {R}ankin-{S}elberg
              {$L$}-functions and equidistribution of {H}eegner points.
              {II}},
   JOURNAL = {Invent. Math.},
  FJOURNAL = {Inventiones Mathematicae},
    VOLUME = {163},
      YEAR = {2006},
    NUMBER = {3},
     PAGES = {581--655}
}

@article {GL13,
    AUTHOR = {Goldmakher, L. and Louvel, B.},
     TITLE = {A quadratic large sieve inequality over number fields},
   JOURNAL = {Math. Proc. Cambridge Philos. Soc.},
  FJOURNAL = {Mathematical Proceedings of the Cambridge Philosophical
              Society},
    VOLUME = {154},
      YEAR = {2013},
    NUMBER = {2},
     PAGES = {193--212}
}

@article{CDD26,
    title={Non-vanishing for quartic {H}ecke {$L$}-functions and ranks of elliptic curves}, 
      author={Castillo, C. and de Faveri, A. and Dunn, A.},
      year={2026},
      note = {\href{https://arxiv.org/abs/2604.01316}{\texttt{arXiv:2604.01316}}}
}

@article {FF04,
    AUTHOR = {Fisher, B. and Friedberg, S.},
     TITLE = {Double {D}irichlet series over function fields},
   JOURNAL = {Compos. Math.},
  FJOURNAL = {Compositio Mathematica},
    VOLUME = {140},
      YEAR = {2004},
    NUMBER = {3},
     PAGES = {613--630}
}

@book {Mon71,
    AUTHOR = {Montgomery, H. L.},
     TITLE = {Topics in multiplicative number theory},
    SERIES = {Lecture Notes in Mathematics},
    VOLUME = {Vol. 227},
 PUBLISHER = {Springer-Verlag, Berlin-New York},
      YEAR = {1971},
     PAGES = {ix+178}
}

@article {HB95,
    AUTHOR = {Heath-Brown, D. R.},
     TITLE = {A mean value estimate for real character sums},
   JOURNAL = {Acta Arith.},
  FJOURNAL = {Acta Arithmetica},
    VOLUME = {72},
      YEAR = {1995},
    NUMBER = {3},
     PAGES = {235--275}
}

@article {IL07,
    AUTHOR = {Iwaniec, H. and Li, X.},
     TITLE = {The orthogonality of {H}ecke eigenvalues},
   JOURNAL = {Compos. Math.},
  FJOURNAL = {Compositio Mathematica},
    VOLUME = {143},
      YEAR = {2007},
    NUMBER = {3},
     PAGES = {541--565}
}

@article{DDDS24,
    title={Non-vanishing for cubic {H}ecke {$L$}-functions}, 
      author={David, C. and de Faveri, A. and Dunn, A. and Stucky, J.},
      year={2024},
      note = {\href{https://arxiv.org/abs/2410.03048}{\texttt{arXiv:2410.03048}}}
}

@article {Xia24,
    AUTHOR = {Li, X.},
     TITLE = {Moments of quadratic twists of modular {$L$}-functions},
   JOURNAL = {Invent. Math.},
  FJOURNAL = {Inventiones Mathematicae},
    VOLUME = {237},
      YEAR = {2024},
    NUMBER = {2},
     PAGES = {697--733}
}

@article{SS24,
    title={The fourth moment of quadratic {D}irichlet {$L$}-{F}unctions {II}}, 
      author={Shen, Q. and Stucky, J.},
      journal = {to appear in Algebra \& Number Theory},
      year={2024},
      note = {\href{https://arxiv.org/abs/2402.01497}{\texttt{arXiv:2402.01497}}}
}

\end{document}